%% file: GBSforWigner.tex
\documentclass{article}
\usepackage{amsmath,latexsym,amssymb,pifont,paralist,graphicx}
\usepackage[colorlinks=true]{hyperref}
\hypersetup{urlcolor=blue, citecolor=red}

\newtheorem{theorem}{Theorem}[section]

\newtheorem{lemma}[theorem]{Lemma}
\newtheorem{proposition}{Proposition}

\newtheorem{remark}{Remark}
\newcommand{\ep}{\varepsilon}

\newenvironment{proof}
{$\!\!\!\!\!\!\!\!\!$\emph{Proof: }}
{{\null\hfill{$\Box$}\par}}

\newcommand{\uuepI}{\underline{u_\ep^I}}
\newcommand{\uvepI}{\underline{v_\ep^I}}
\newcommand{\udtuep}{\underline{\dt u_\ep}}
\newcommand{\udtuepInd}{\underline{\dt u_{\ep,\indx}}}

\newcommand{\udxuepInd}{\underline{\dpp_x u_{\ep,\indx}}}
\newcommand{\udxjuep}{\underline{\dpp_{x_j} u_\ep}}
\newcommand{\udxjuepI}{\dpp_{x_j}\underline{ u_\ep^I}}
\newcommand{\underxjuepI}{\underline{\dpp_{x_j} u_\ep^I}}
\newcommand{\udxjuepInd}{\underline{\dpp_{x_j} u_{\ep,\indx}}}

\newcommand{\IN}{\mathbb{N}}
\newcommand{\IR}{\mathbb{R}}
\newcommand{\IC}{\mathbb{C}}

\newcommand{\BB}{B}
\newcommand{\Bic}{\mathcal{\BB}}
\newcommand{\CC}{C}
\newcommand{\Co}{\mathcal\CC}
\newcommand{\EE}{E}
\newcommand{\nrj}{\mathcal\EE}
\newcommand{\ron}{\, o \,}
\newcommand{\FF}{F}
\newcommand{\Fo}{\mathcal{\FF}}
\newcommand{\MM}{M}
\newcommand{\Ma}{\mathcal\MM}
\newcommand{\NN}{N}
\newcommand{\Nm}{\mathcal{\NN}}
\newcommand{\RR}{R}
\newcommand{\Ref}{\mathcal{\RR}} %
\newcommand{\SSS}{S}
\newcommand{\So}{\mathcal\SSS}
\newcommand{\UU}{U}
\newcommand{\Um}{\mathcal{\UU}}

\newcommand{\dist}{\mathrm{dist}}
\newcommand{\supp}{\mathrm{supp}}
\newcommand{\Tr}{\mathrm{Tr}}
\renewcommand{\Im}{\mathrm{Im}\,}
\renewcommand{\Re}{\mathrm{Re}\,}

\newcommand{\alp}{\alpha}
\newcommand{\gb}{\omega} 
 
\newcommand{\ox}{\otimes}
\newcommand{\x}{\times}
\newcommand{\dpp}{\partial}

\newcommand{\Cob}{\overset{o}{T^*\Omega}}
\newcommand{\dom}{\dpp \Omega}
\newcommand{\dt}{\dpp_t}
\newcommand{\indx}{r_0,r_\infty}
\newcommand{\liSu}{\underset{\ep \rightarrow 0}{\limsup}\,}
\newcommand{\nd}{\frac{n}{2}}
\newcommand{\nq}{\frac{n}{4}}
\newcommand{\phiu}{\gamma'}
\newcommand{\phsi}{\psi_{\text{inc}}}
\newcommand{\phsr}{\psi_{\text{ref}}}
\newcommand{\rhou}{\rho'}
\newcommand{\Rcomp}{\mathnormal{R}_\eta}
\newcommand{\tnd}{\frac{3n}{2}}
\newcommand{\tnq}{\frac{3n}{4}}
\newcommand{\ud}{\frac{1}{2}}

\newcommand\Char[1]{\mbox{{\boldmath{$1$}}$_{#1}$}}
\newcommand\nord[1]{\|#1\|_{L^2}}
\newcommand\nordom[1]{\|#1\|_{L^2(\Omega)}}
\newcommand\nordye[1]{\|#1\|_{L^2_{y,\eta}}}
\newcommand\norhom[1]{\|#1\|_{H^1(\Omega)}}
\newcommand\nrjFun[1]{\nrj\Par{#1}}
\newcommand\prive[1]{\backslash\{ #1 \}}
\newcommand\tendz[1]{{#1}\rightarrow 0} 
\newcommand\trak[1]{{\widetilde{#1}}^k}
\newcommand\tral[1]{{\widetilde{#1}}^l}
\newcommand\asp[2]{\overset{#1}{\underset{#2}{\asymp}}}
\newcommand\GaE[2]{-(#1)\cdot #2 (#1)}
\newcommand\nordvois[2]{\|#1\|_{L^2({#2})}}
\newcommand\somme[2]{ \overset {#2} { \underset {#1} {\sum} }  \,}
\newcommand\GasE[3]{-(#1)\cdot #2 (#1)/{#3}}

\newcommand\Croch[1]{\left[#1\right]}
\newcommand\Modul[1]{\left|#1\right|}
\newcommand\Par[1]{\left(#1\right)}

\title{A Gaussian beam approach for computing Wigner measures in convex domains}

\author{Jean-Luc Akian\footnote{Aeroelasticity and Structural Dynamics Department, ONERA, 92320 Ch\^atillon, France (jean-luc.akian@onera.fr).}{} , Radjesvarane Alexandre\footnote{Department of Mathematics and Institute of Natural Sciences, Shanghai Jiao Tong University, 200900 Shanghai, PRC China (alexandreradja@gmail.com).}{ }  and Salma Bougacha\footnote{Mechanics, Structures and Materials Laboratory, \'Ecole Centrale Paris, 92295 Ch\^atenay-Malabry, France (salma.bougacha@ecp.fr).}{}}

\begin{document}

\maketitle

\begin{abstract}
A Gaussian beam method is presented for the analysis of the energy of the high frequency solution to the mixed problem of the scalar wave equation in an open and convex subset $\Omega$ of $\IR^n$, with initial conditions compactly supported in $\Omega$, and Dirichlet or Neumann type boundary condition. The transport of the microlocal energy density along the broken bicharacteristic flow at the high frequency limit is proved through the use of Wigner measures. Our approach consists first in computing explicitly the Wigner measures under an additional control of the initial data allowing to approach the solution by a superposition of first order Gaussian beams. The results are then generalized to standard initial conditions.



\end{abstract}

\section{Introduction} 																						
\label{intro} 			\input{Ch2Intro}

\section{Tool-box and construction of the asymptotic solution} 		
\setcounter{equation}{0}
\label{GBS} 				\input{Ch2Part1}

\section{Wigner transforms and measures} 													
\setcounter{equation}{0}
\label{wigComput} 	\input{Ch2Part2}

\section{Appendix}
\setcounter{equation}{0}
\subsection{Reflected first order and higher order beams} 				\label{AppA} 				\input{AppendixA}

\subsection{Approximation operators} 															\label{AppB} 				\input{AppendixB}

\subsection{Results related to the FBI and the Wigner transforms} \label{AppC} 				\input{AppendixC}

\end{document}

%% file: Ch2Intro.tex
We are interested in the high frequency limit of the initial-boundary value problem (IBVP) for the wave equation 
\begin{subequations}
\label{MPb:gp} 
\begin{gather}
Pu_\ep=\dpp_t^2u_\ep-\somme{j=1}{n}\dpp_{x_j} \Par{c^2(x)\dpp_{x_j} u_\ep}=0 \text{ in } [0,T]\x \Omega, \label{MPb:gp1}\\
B u_\ep=0 \text{ in }[0,T]\x \dom,\label{MPb:gp2}\\
u_\ep|_{t=0}=u^I_\ep,\, \dt u_\ep|_{t=0}=v^I_\ep  \text{ in } \Omega,\label{MPb:gp3}
\end{gather}
\end{subequations}
where $B$ stands for a Dirichlet or Neumann type boundary operator. 

Above, $T>0$ is fixed, $\Omega$ is a bounded domain of ${\mathbb R}^n$ with a $\Co^\infty$ boundary and the wave propagation velocity  $c$ is in $\Co^{\infty}(\bar\Omega)$, though this assumption may be relaxed. The initial data depend on a small wavelength parameter $\ep>0$ and we assume that
\begin{align}
\label{CIboundBis} 
&u^I_\ep \text{ and } v^I_\ep \text{ are uniformly bounded w.r.t. } \ep \text{ respectively in } H^1(\Omega) \nonumber \\
&\text{and }L^2(\Omega).
\tag{\textup{H1}}
\end{align}
We are interested in the description of the behavior of the local energy density $\ud |\dt u_\ep|^2+ \ud \somme{j=1}{n} c^2 |\dpp_{x_j} u_\ep|^2$, at the high frequency limit $\tendz{\ep}$, in which case, it is well known that  this quantity can be computed through the use of Wigner measures. 

The Wigner transform is a phase space distribution introduced by E. Wigner \cite{Wigner} in 1932 to study quantum corrections to classical statistical mechanics. In the 90's, mathematicians became increasingly interested by these transforms and related measures, see for example \cite{LiPa,MaMa,MaMaPo,MaPiPo} for the semiclassical limit of Schr\"odinger equations. A general theory for their use in the homogenization of energy densities of dispersive equations was laid out by G\'erard et al. in \cite{GeMaMaPo}, see also \cite{Gerard91a,GaMa}. Wigner measures are also related to the H-measures and microlocal defect measures introduced in \cite{Tartar} and \cite{Gerard91b},  see also \cite{Bu97,Alexandre}. Whereas there is no notion of scale for the latter measures, Wigner transforms are associated to a small parameter tending to zero. In quantum mechanics, this parameter is the rescaled Planck constant, while it will be typically the distance between two points of the medium's periodic structure for homogenization problems. 

The Wigner transform, at the scale $\varepsilon$, is defined for a given sequence $(a_\ep ,b_\ep)$ in $\So'(\IR^n)^p\x \So'(\IR^n)^p$  
as the tempered distribution
\begin{equation*}
w_\ep(a_\ep,b_\ep)(x,\xi) = (2 \pi)^{-n} \int_{\IR^{n}} e^{-i v \cdot \xi} a_\ep(x+\frac{\ep}{2}v) b^*_\ep(x-\frac{\ep}{2}v) dv.
\end{equation*}
If $a_\ep$ is uniformly bounded w.r.t. $\ep$ in $L^2(\IR^n)^p$,
then $w_\ep[a_\ep] := w_\ep(a_\ep,a_\ep)$ converges as $\ep$ goes to $0$ in $\Ma_p\Par{\So'(\IR^n_x \x \IR^n_\xi)}$ to a positive hermitian matrix measure (modulo the extraction of a subsequence), which is called a Wigner measure associated to $(a_\ep)$ and denoted $w[a_\ep]$. The Wigner measures associated to the solution of the wave equation (and hyperbolic problems in general, see e.g. \cite{GeMaMaPo,PaRh}) are related to the energy density in the high frequency limit. More precisely, under suitable hypotheses (see Proposition 1.7 in \cite{GeMaMaPo}), the density of energy associated to the solution $u_\ep^C$ of the Cauchy problem for the scalar wave equation converges as $\ep \rightarrow 0$ in the sense of measures to
\begin{equation*}
	 \int_{\IR^n} \nrjFun{u_\ep^C(t,.)}(x,d\xi) ,
\end{equation*}
where
\begin{equation*}
\nrjFun{u_\ep^C(t,.)} = \ud w[\dt u_\ep^C(t,.)] + \ud \somme{j=1}{n} w[ c \dpp_{x_j} u_\ep^C(t,.)].
\end{equation*}
Moreover, $\nrjFun{u_\ep^C(t,.)}$ is the sum of two measures  satisfying transport equations of Liouville type (see e.g. \cite{GeMaMaPo}).

For the Dirichlet or Neumann initial boundary value problem connected with the wave equation, we shall study the same quantity after extending $\dt u_\ep(t,.)$ and $c \dpp_{x_j} u_\ep(t,.)$, $j=1,\dots,n$, to functions of $L^2(\IR^n)$ by setting $\udtuep= \Char{\Omega} \dt u_\ep$, $\udxjuep= \Char{\Omega} \dpp_{x_j} u_\ep$ and extending $c$ outside $\bar \Omega$ in a smooth way. Hence, we call microlocal energy density of $u_\ep$ the distribution
\begin{equation*}
 \ud w_\ep\Croch{\udtuep(t,.)} + \ud  \somme{j=1}{n} w_\ep\Croch{c \udxjuep(t,.)}
\end{equation*}
and its high frequency limit the measure
\begin{equation*}
\nrjFun{u_\ep(t,.)} = \ud w\Croch{\udtuep(t,.)} + \ud  \somme{j=1}{n} w\Croch{c \udxjuep(t,.)}.
\end{equation*}
$\uvepI = \Char{\Omega} v^I_\ep $ and $\underxjuepI = \Char{\Omega}\dpp_{x_j}u^I_\ep$ ($j=1,\dots,n$) will satisfy the usual assumptions needed in the general context of the study of Wigner measures: their Wigner measures are supposed unique and
\begin{align}
\label{oscil} 
&\uvepI \text{ and } \underxjuepI,\, j=1,\dots,n,\text{ are } \ep\text{-oscillatory (see \eqref{oscilDefv} and \eqref{oscilDefu}),}
\tag{\textup{H2}}\\
\label{wigZero} 
&\text{the Wigner measures of }\Par{\uvepI} \text{ and } \Par{\underxjuepI},\, j=1,\dots,n,\text{ do not charge the set}\nonumber \\
&\IR^n \x \{\xi=0\}.
\tag{\textup{H3}}
\end{align}
Our present study will be restricted to the case where the rays starting from the support of the initial data do not face diffraction on the boundary, nor do they glide along $\dom$. Therefore, we also assume that
\begin{equation}	
\label{CIcompBis}
u^I_\ep \text{ and } v^I_\ep \text{ have supports contained in a fixed compact set of } \Omega \text{ independent of }\ep,
\tag{\textup{H4}}
\end{equation}
\begin{align*}
 &\!\!\!\!\Omega \text{ is convex with respect to the bicharacteristics of the wave operator, that is }\\
 &\!\!\!\!\text{every ray originating from } \Omega \text{ hits the boundary twice and transversally,}
\end{align*}
and
\begin{align*}
 &\!\text{the boundary has no dead-end trajectories, that is infinite number of successive}\\
 &\!\text{reflections cannot occur in a finite time.}
\end{align*}
These geometric hypotheses insure that the only phenomena occurring at the boundary is the reflection according to the geometrical optics laws.

Wigner measures for the wave equation in presence of a boundary or an interface have been studied by Miller \cite{Miller} who proved refraction results for sharp interfaces 
and Burq \cite{Burq} who described their support for a Dirichlet boundary condition.
Similar results have been established for other problems \cite{BuLe,Duyckaerts04,Fouassier07}, in particular the eigenfunctions for the Dirichlet problem \cite{Zelditch,GeLe} and for the Neumann and Robin problems \cite{Burq03}. All these works are based on pseudo-differential calculus, and in particular the use of a tangential pseudo differential calculus.

In this paper, we present an approach to compute Wigner measures based on the Gaussian beam formalism. Therefore, we avoid any use of adapted pseudo-differential calculus. Though a Gaussian beam technic requires much more work, compared to the above mentioned papers, one advantage is that we are able to give asymptotic estimates for remainders terms, which could be useful for numerical purposes for instance.

Let us recall that Gaussian beams (or the related coherent states) are waves with a Gaussian shape at any instant, localized near a single ray \cite{Babich,Ralston82}. They play the role of a basis of fundamental solutions of wave motion and furthermore can be used to study general solutions of partial differential equations (PDEs). For example, they can help for the understanging of propagation of singularities \cite{Ralston82}, to prove lack of observability \cite{MaZu} and to study semiclassical measures \cite{PaUr96} and trace formulas \cite{Wilkinson,CoRaRo}. 

To describe non localized solutions of PDEs, one can use the Gaussian beam summation method  \cite{KaPo81,CePoPs,Klimes84}. The initial field is expanded as a sum of Gaussian beams. Each individual beam is computed and the solution is then obtained at an observation point by superposing a selection of Gaussian beams. The summation strategies are numerous. The sum can be discrete \cite{MoRu08,TaEnTs09,ArEn10} or continuous \cite{LiRaAcou,LiRuTa2010}, the selection of the beams to be superposed can be done according to several criteria. In \cite{BoAkAl09}, a weighted integral of Gaussian beams was designed to build an approximate solution of the IBVP \eqref{MPb:gp} under an additional assumption \eqref{CIsmoothBis} on the initial data (see p.\pageref{CIsmoothBis}). See also \cite{LeQi2009, LeQi2010, QiYi2010} for recent numerical implementations related to this method.

Gaussian beams seem to be very well suited for the study of Wigner measures. Indeed, the Wigner transform of two different beams vanishes when $\ep$ goes to zero. Even better, the Wigner measure of one individual beam is a Dirac mass localized on the corresponding bicharacteristic. Thus Gaussian beams act as an orthogonal family for the Wigner measure. Using these elementary solutions for studying Wigner measures is not new, see for example in the whole space domain the work by Robinson \cite{Robinson} for the Schr\"odinger equation, and more recently the paper by Castella \cite{Castella} who used a coherent states approach for the Helmholtz equation.

As the microlocal energy density of one individual beam is concentrated near its associated bicharacteristic, one would expect that the Wigner measure of a summation of weighted Gaussian beams will yield easily that the associated weights are transported along the broken bicharacteristic flow (see p.\pageref{refFlow} for the construction of reflected flows and p.\pageref{bbf} for the definition of the broken flow). Unfortunately this result is not immediate as even different beams become infinitely close to each other. 
However, we shall show by elementary computations that this intuition is indeed true and that the microlocal energy density of the considered approximate solution is transported at the high frequency limit along the broken bicharacterisitic flow. Since the asymptotic solution is close to the exact solution $u_\ep$, we may deduce the same consequence for $\nrjFun{u_\ep(t,.)}$.

The additional hypothesis \eqref{CIsmoothBis} consists in assuming that in the frequency space, the initial data are supported in a compact that does not contain $0$ (modulo infinitely small residues). When studying Wigner measures by the pseudo-differential calculus techniques, the frequency behavior of the initial conditions is only controlled by the less restrictive hypotheses  \eqref{oscil} and \eqref{wigZero}. Hence, the assumption \eqref{CIsmoothBis} is artificial (though not for numerical purposes) and is required only by the Gaussian beam summation method we have chosen. However, for $\ep$-oscillatory initial data with Wigner measures not charging the set $\IR^n \x\{\xi=0\}$, a truncation of their frequency support at infinity and at zero does not affect the energy density of the solution as $\ep \rightarrow 0$. By achieving such a truncation, we succeed to derive the transport property of the energy density under the traditional hypotheses \eqref{oscil} and \eqref{wigZero}:

\begin{theorem}
Assume the hypotheses \eqref{CIboundBis}-\eqref{CIcompBis} on the initial conditions hold true. Let $\nrj^\pm=\ud w\Croch{\uvepI \pm ic |D| \uuepI}$ and denote by $\varphi^t_b$ the broken bicharacteristic flow associated to $-i\dpp_t- c |D|$ obtained after successive reflections on the boundary $\dom$.
Then 
\begin{equation*}
\nrjFun{u_\ep(t,.)}=\ud\big(\nrj^+ \ron (\varphi_b^{-t})^{-1}+\nrj^- \ron (\varphi_b^{t})^{-1} \big) \text{ in } \Omega \x (\IR^n\prive{0}). 
\end{equation*}
\end{theorem}

As mentioned already in our Introduction, this result is well known. But our method of proof is able to give more precise estimations than those stated above for the Wigner measures. In particular, we have estimations on the Wigner transforms of the solutions.

The rest of the paper is organized as follows. In Section 2, we recall the construction of first order Gaussian beams and the structure of the asymptotic solutions obtained as an infinite sum of such beams. The derivatives of the asymptotic solutions are then expressed using Gaussian type integrals.
We simplify the expression of the Wigner transform of such integrals in Section \ref{wigComput}, following initial computations of \cite{Robinson} in the Schr\"odinger case.
We then compute the microlocal energy density of the asymptotic solution by exploiting the expressions of the beams phases and amplitudes and using the dominated convergence theorem.  We prove the propagation along the broken flow of $\nrjFun{u_\ep(t,.}$ at the high frequency limit, with the help of assumptions \eqref{oscil} and \eqref{wigZero} on the initial data. Some complementary results are collected in an Appendix, Section 4.


Let us end this Introduction with a few notations which will be used hereafter.

A vector $x \in \IR^d$ will be denoted by $(x_1,\dots,x_d)$, the inner product of two vectors $a,b \in \IR^d$ by $a \cdot b$, and the transpose of a matrix $A$ by $A^T$. If $E$ is a subset of $\IR^d$, we denote  $E^c$ its complementary and  $\Char{E}$ its characteristic function. For a function $f \in L^2(\Omega)$, we let $\underline{f} = \Char{\Omega} f$.
For $r>0$, $\chi_r$ denotes a cut-off function in $\Co_0^\infty(\IR^n,[0,1])$ such that
\begin{equation*}
	\chi_r(x)=1 \text{ if } |x|\leq r/2 \text{ and }	\chi_r(x)=0 \text{ if } |x|\geq r.
\end{equation*}
We use the following definition of the Fourier transform 
\begin{equation*}
\Fo_{x}{u}(\xi)=\int_{\IR^d} u(x) e^{-ix \cdot \xi} dx \text{ for } u \in L^2(\IR^d). 
\end{equation*}
If no confusion is possible, we shall omit the reference to the lower index $x$. 
\newline
We keep the standard multi-index notations. 
For a scalar function $f\in \Co^\infty(\IR_x^d,\IC)$, $\dpp_x f$ will denote its gradient vector $(\dpp_{x_j} f)_{1\leq j \leq d}$ and $\dpp_x^2 f$ will denote its Hessian matrix $(\dpp_{x_j} \dpp_{x_k} f)_{1\leq j,k \leq d}$.
For a vector function $g\in \Co^\infty(\IR^d,\IC^p)$, the notation $D g$ is used for its Jacobian matrix $(D g)_{j,k}=\dpp_{x_k} g_j$. If $g$ is a function in $\Co^\infty(\IR^n_y \x \IR^n_{\eta},\IC^p)$, we denote $(D_y g)_{j,k} =\dpp_{y_k} g_j $ and $(D_\eta g)_{j,k} =\dpp_{\eta_k} g_j $.
We use the letter $C$ to denote a (possible different at each occurence) positive constant.
\\
For $(y_\ep)$ and $(z_\ep)$ sequences of $\IR_+$ with $\ep \in]0,\ep_0]$, we use the notation $y_\ep \lesssim z_\ep$ if there exists $C>0$ independent of $\ep$ such that $y_\ep \leq C z_\ep$ for $\ep$ small enough. We write $y_\ep \lesssim \ep^\infty$ or $y_\ep = O(\ep^\infty)$ if for any $ s \geq 0$ there exists $C_s>0$ s.t.  for $\ep$ small enough  $y_\ep \leq C_s \ep^s$.
\\
Finally, if $E$ is in an open subset of $\IR^{2n}$ and $\nu_\ep$, $\nu_\ep'$ are two distributions s.t.
\begin{equation*}
\underset{\tendz{\ep}}{\lim}  (\nu_\ep-\nu_\ep') = 0 \text{ in } E, 
\end{equation*}
we shall write
\begin{equation*}
\nu_\ep \approx \nu_\ep' \text{ in } E.
\end{equation*}

%% file: Ch2Part1.tex
We recall the construction made in \cite{BoAkAl09} of an asymptotic solution as a superposition of Gaussian beams and give the expression of its time and spatial derivatives with the help of so called Gaussian integrals.

\subsection{First order Gaussian beams}

\subsubsection{Beams in the whole space} \label{beams}

Let $h_+(x,\xi)=c(x)|\xi|$ and $(x^t,\xi^t)$
be a Hamiltonian flow for $h_+$, 
that is a solution of the system
\begin{equation*}
\displaystyle{\frac{d}{dt}}x^t=\dpp_\xi h_+(x^t,\xi^t)=c(x^t)\frac{\xi^t}{|\xi^t|},		\qquad \displaystyle{\frac{d}{dt}}\xi^t=-\dpp_x h_+(x^t,\xi^t)=-\dpp_x c(x^t)|\xi^t|.
\end{equation*}
The curves $(t,x^{\pm t})$ of $\IR^{n+1}$ are called the rays of $P$.

An individual first order (Gaussian) beam for the wave equation associated to a ray $(t,x^t)$ has the form
\begin{equation*}
	\gb_\ep(t,x) =  a_0( t,x) e^{i\psi(t,x)/\ep},
\end{equation*}
with a complex phase function $\psi$ real-valued on $(t,x^{t})$, an amplitude function $a_0$ null outside a neighborhood of $(t,x^{t})$, and such that
\begin{equation*}
	 \underset{t\in[0,T]}{\sup}\nordom{P \gb_\ep(t,.)} = O(\ep^m),
\end{equation*}
for some $m>0$.

The construction of such a beam is achieved by making the amplitudes of $P\gb_\ep$ vanish on the ray up to fixed suitable orders \cite{Ralston82,KaKuLa,MaZu} 
\begin{equation}
\label{Pgb}
	P \gb_\ep = \Big( \ep^{-2} p(x,\dt\psi,\dpp_x\psi) a_0 + \ep^{-1} i\Par{2\dt \psi \dt a_0 - 2 c^2 \dpp_x \psi \dpp_x a_0 + P\psi a_0}+ \text{h.o.t.} \Big) e^{i\psi/\ep},
\end{equation}
where $p(x,\tau,\xi)=  c^2(x) |\xi|^2 - \tau^2 $ is the principal symbol of $P$ and h.o.t. denotes higher order terms.
The first equation is then the eikonal equation
\begin{equation}
\label{eikonalTwo}
p\Par{x,\dt \psi(t,x),\dpp_x\psi(t,x)}=0
\end{equation}
on $x=x^t$ up to order $2$ (see Remark 2.1 in \cite{BoAkAl09} for an explanation of the choice of this specific order),  which means
\begin{equation*}
	\dpp_{x}^\alp \Croch{p\Par{x,\dt \psi(t,x),\dpp_x\psi(t,x)}}|_{x=x^t}=0 \text{ for } |\alp|\leq 2.
\end{equation*}
Orders $0$ and $1$ of the previous equation are fulfilled on the ray by setting 
\begin{equation}
\label{PhaseOneBis}
	\dt \psi (t,x^{t})= -h_+(x^{t},\xi^{t}) \text{ and } \dpp_x \psi (t,x^{t})=\xi^{t}. 
\end{equation}
Choosing
$\displaystyle{\psi(0,x^0)} \text{ as a real quantity}$, it follows that
\begin{equation}
\label{PhaseZeroBis}
	\psi(t,x^t) \text{ is real}. 
\end{equation}
Order $2$ of eikonal \eqref{eikonalTwo} on the ray may be written as a Riccati equation
\begin{equation}
\label{riccatiBis}
\begin{split}
&	\displaystyle{\frac{d}{dt}}\Par{\dpp^2_{x}\psi(t,x^t)}+ H_{21}(x^t,\xi^t) \dpp^2_{x}\psi(t,x^t) + \dpp^2_{x}\psi(t,x^t) H_{12}(x^t,\xi^t)  \\
&+ \dpp^2_{x}\psi(t,x^t) H_{22}(x^t,\xi^t)  \dpp^2_{x}\psi(t,x^t) + H_{11}(x^t,\xi^t) = 0,
\end{split}
\end{equation}
where $H=\left(\begin{array}{cc}
H_{11} & H_{12} \\
H_{21} & H_{22} \\
\end{array}\right)$ is the Hessian matrix of $h_+$.
This nonlinear Riccati equation has a unique global symmetric solution which satisfies the fundamental property
\begin{equation}
\label{PhaseTwoBis}
	\Im \dpp_{x}^2 \psi\Par{t,x^t} \text{ is positive definite},
\end{equation}
given an initial symmetric matrix $\dpp_{x}^2 \psi\Par{0,x^0}$ with a positive definite imaginary part (see the proof of Lemma 2.56 p.101 in \cite{KaKuLa}).

The phase is defined beyond the ray as a polynomial of order $2$ w.r.t. $(x-x^t)$ \cite{TaQiRa} 
\begin{equation}
\label{quadPsi}
 	\psi(t,x) = \psi(t,x^t) + \xi^t \cdot (x-x^t) + \ud (x-x^t) \cdot \dpp_x^2\psi(t,x^t) (x-x^t).
\end{equation}
Next, we make the term associated to the power $\ep^{-1}$ in the expansion \eqref{Pgb} vanish on $(t,x^t)$
\begin{equation}
\label{evolBis}
 2\dt \psi \dt a_0 - 2 c^2 \dpp_x \psi \dpp_x a_0 + P\psi a_0= 0 \text{ on }	(t,x^t),
\end{equation}
which leads to a linear ordinary differential equation (ODE) on $a_0(t,x^t)$. The amplitude is then chosen under the form
\begin{equation*}
	a_0(t,x)= \chi_d(x-x^t) a_0(t,x^t),
\end{equation*}
where $d$ is a positive parameter. The constructed beams are thus defined for all $(t,x) \in \IR^{n+1}$ and they satisfy the estimate
\begin{equation*}
\nordom{\ep^{-\nq+1}P\gb_\ep(t,.)}=O(\sqrt \ep) \text{ uniformly w.r.t. } t\in [0,T].	
\end{equation*}
Note that Gaussian beams for $P$ associated to the ray $(t,x^{-t})$ are $\gb_\ep(-t,x)$.

\subsubsection{Incident and reflected beams in a convex domain} \label{Inc&ref}

Assume that $c(x)$ is constant for $\dist(x,\bar \Omega)$ larger than some constant $C>0$. Given a point $(y,\eta)$ in the phase space $\overset{o}{T^*\IR^{n}}$, where $\overset{o}{T^* U}$ denotes  $U \x \Par{\IR^n \prive{0}}$ if  $U$ is an open set of $\IR^n$, the Hamiltonian flow $\varphi^t_0(y,\eta) = \Par{x^t_0(y,\eta),\xi^t_0(y,\eta)}$ satisfying:
\begin{align*}
\frac{d}{dt}x^t_0=c(x^t_0)\frac{\xi^t_0}{|\xi^t_0|},&\,		\frac{d}{dt}\xi^t_0=-\dpp_x c(x^t_0)|\xi^t_0|,	\\	
  \\
x^t_0|_{t=0}=y,&\,	\xi^t_0|_{t=0}=\eta,\eta \ne 0,
\end{align*}
is called incident flow. A beam associated to the incident ray $(t,x^t_0)$ is denoted $\gb_\ep^0$ and called an incident beam.
Since we have dependence w.r.t. the initial conditions $(y,\eta )$, we write the incident beam as
\begin{equation*}
	\gb^0_\ep(t,x,y,\eta) =  a_0(t,x,y,\eta) e^{i\psi_0(t,x,y,\eta)/\ep}.
\end{equation*}
Let $\Ref$ be the reflection involution
\begin{equation}
\label{refInv}\begin{split}
	\Ref:\overset{o}{T^*\IR^n}|_{\dom}&\rightarrow \overset{o}{T^*\IR^n}|_{\dom} \\
	(X,\Xi)& \mapsto \Par{X,(Id - 2 \nu(X)\nu(X)^T)\Xi},
\end{split}\end{equation}
where $\nu$ denotes the exterior normal field to $\dom$. We shall only consider initial points $(y,\eta) \in \Bic= \cup_{t\in \IR}\varphi_0^t(\Cob)$ giving rise to rays that enter the domain $\Omega$ at some instant. Each associated flow $\varphi^t_0(y,\eta)$ hits the boundary twice. Reflection of $\varphi^t_0(y,\eta)$ at the exit time $t=T_1(y,\eta)$ s.t.
\begin{equation*}
x_{0}^{T_1(y,\eta)}(y,\eta)\in \dom \text{ and }\dot{x}_0^{T_1(y,\eta)}(y,\eta) \cdot \nu\Par{x_{0}^{T_1(y,\eta)}(y,\eta)}>0
\end{equation*}
gives birth to the reflected flow $\varphi^t_1(y,\eta)=\Par{x_1^t(y,\eta),\xi_1^t(y,\eta)}$ defined by the condition
\begin{equation*}
\label{refFlow}
	\varphi^{T_{ 1}(y,\eta)}_1(y,\eta)=\Ref\,o\varphi^{T_{ 1}(y,\eta)}_0(y,\eta). 
\end{equation*}

Similarly, we also define the reflection time $T_{-1}(y,\eta)$ and the flow $\varphi_{-1}^t(y,\eta)$ by reflecting $\varphi^{t}_0(y,\eta)$ as follows
\begin{equation*}\begin{split}
&x_{0}^{T_{-1}(y,\eta)}(y,\eta)\in \dom \text{ and }\dot{x}_0^{T_{-1}(y,\eta)}(y,\eta) \cdot \nu\Par{x_{0}^{T_{-1}(y,\eta)}(y,\eta)}<0,\\
&\varphi_{-1}^{T_{-1}(y,\eta)}(y,\eta)= \Ref \ron \varphi_0^{T_{-1}(y,\eta)}(y,\eta).
\end{split}\end{equation*}
We denote, for $k=\pm 1$, the reflected beams by 
\begin{equation*}
	\gb^k_\ep(t,x,y,\eta) = a^k_0(t,x,y,\eta) e^{i\psi_k(t,x,y,\eta)/\ep}.
\end{equation*}
These beams are associated to the reflected bicharacteristics $\varphi_k^t$. Let us introduce, for $k=0,\pm 1$, the boundary amplitudes $d_{-m_B+j}^k$ s.t. 
$$B  \gb_\ep^k = \somme{j= 0}{m_B} \ep^{-m_B+j} d_{-m_B+j}^k e^{i\psi_k/\ep}.$$
Above, $m_B$ denotes the order of $B$ ($m_B=0$ for Dirichlet and $m_B=1$ for Neumann). The construction of the reflected phases and amplitudes is achieved by imposing that
\begin{enumerate}
\item \label{refPhaseBis} the time and tangential derivatives of $\psi_k$ equal at $(T_k,x_0^{T_k})$ those of $\psi_0$ up to order $2$,
\item \label{refAmp} $\Par{d_{-m_B}^0 + d_{-m_B}^k}(T_k,x_0^{T_k})=0$,
\end{enumerate}
for $k=\pm 1$.
These constraints \emph{uniquely} determine the reflected phases and amplitudes, once the incident ones are fixed \cite{Ralston82}. If $T$ is sufficiently small, at most one reflection occurs in the interval $[0,T]$ and in the interval $[-T,0]$ for a fixed starting position and vector speed $(y,\eta) \in \Cob$, and the following boundary estimates are satisfied \cite{Ralston82} 
\begin{align*}
\|B\Par{\ep^{-\nq+1} \gb^0_\ep(.,y,\eta) + \ep^{-\nq+1}\gb^{1}_\ep(.,y,\eta)}\|_{H^{s}( [0,T]\x\dom)} &=O(\ep^{-m_B-s+\frac{3}{2}}),\\
 \text{and } \|B\Par{\ep^{-\nq+1} \gb^0_\ep(.,y,\eta) + \ep^{-\nq+1}\gb^{-1}_\ep(.,y,\eta)}\|_{H^s([-T,0] \x \dom)}&=O(\ep^{-m_B-s+\frac{3}{2}}),
\end{align*}
for $s\geq 0$.

\subsection{Gaussian beam summation}

The construction of asymptotic solutions to the IBVP \eqref{MPb:gp1}-\eqref{MPb:gp2} with initial conditions \eqref{CIcut} having a suitable frequency support (see below) is recalled, through the Gaussian beam summation introduced in \cite{BoAkAl09}. 
We focus on a superposition of first order beams, for which exact expressions of the phases and amplitudes are displayed in Subsection \ref{fstBeams}.
These beams lead to a first order approximate solution, close to the exact one up to $\sqrt \ep$. Then, the derivatives of the first order solution will be approximated by some Gaussian type integrals.

\subsubsection{Construction of the approximate solution}

In \cite{BoAkAl09}, we have constructed a family of asymptotic solutions to the IBVP for the wave equation for initial data satisfying \eqref{CIboundBis}, \eqref{CIcompBis} and an additional hypothesis \eqref{CIsmoothBis} concerning their FBI transforms.

Let us recall here that the FBI transform (see \cite{Martinez}) is, for a given scale $\varepsilon$, the operator $T_\ep:L^2(\IR^n)\rightarrow L^2(\IR^{2n})$ defined by
\begin{equation}
\label{FBI1Bis}
T_\ep (a) (y,\eta)=c_n \ep^{-\tnq} \int_{\IR^n}a(x)e^{i \eta \cdot(y-x)/\ep-(y-x)^2/(2\ep)} dx,\, c_n=2^{-\nd}\pi^{-\tnq}, \ a\in L^2(\IR^n) ,
\end{equation}
with adjoint operator given by
\begin{equation*}
T_\ep^* (f) (x)=c_n \ep^{-\tnq} \int_{\IR^{2n}} f(y,\eta )e^{ i \eta \cdot (x-y) /\ep-(x-y)^2/(2\ep)} dy d\eta,\,  f\in L^2(\IR^{2n}) .
\end{equation*}
As the Fourier transform, the FBI transform is an isometry, satisfying $T_\ep^* T_\ep = Id$. The extra assumption on the initial data needed in \cite{BoAkAl09} is
\begin{equation}
\label{CIsmoothBis} 
\nordvois{T_\ep \uuepI}{\IR^n \x  \Rcomp^c} = O(\ep^\infty) \text{ and } \nordvois{T_\ep \uvepI}{\IR^n \x  \Rcomp^c} = O(\ep^\infty),
\tag{\textup{H5}}
\end{equation}
where $\Rcomp^c$ denotes the complementary in $\IR^n$ of some ring $\Rcomp=\{\eta\in\IR^n,r_0\leq |\eta| \leq r_\infty\}$, $0<r_0 \ll r_\infty$.

In general, this assumption may be not satisfied. 

Therefore, we construct a family of initial data $(u_{\ep,\indx}^I,v_{\ep,\indx}^I)$ close to 
\\
$(\uuepI,\uvepI)$, satisfying the same assumptions as \eqref{CIboundBis}, \eqref{CIcompBis} and having FBI transforms small in $L^2(\IR^n \x  \Rcomp^c)$. Letting $r_0$ go to $0$ and $r_\infty$ go to $+\infty$ makes these data approach $(\uuepI,\uvepI)$ in a sense that will be specified in Section \ref{truncation}. In any case, the needed convergence is weaker than a $L^2$ convergence since we are interested in the study of Wigner measures. 

Let us first truncate $T_\ep \uuepI$ and $T_\ep \uvepI$ outside $\Rcomp$ by multiplying them by a cut-off $\gamma_{\indx} \in \Co_0^\infty(\IR^n,[0,1])$ supported in the interior of $\Rcomp$ 
\begin{equation}
\label{cutet}
		\gamma_{\indx} = \chi_{r_\infty/2} (1-\chi_{4 r_0}).
\end{equation}
Lemma \ref{cuteta} from the Appendix (Section 4) yields
\begin{align*}
\nordvois{T_\ep T_\ep^* \gamma_{\indx}(\eta) T_\ep	\uuepI}{\IR^n \x  \Rcomp^c} &= O(\ep^\infty), \nonumber \\
\text{and } \nordvois{T_\ep T_\ep^* \gamma_{\indx}(\eta) T_\ep	\uvepI}{\IR^n \x  \Rcomp^c} &= O(\ep^\infty).
\end{align*}
In order to have data supported in fixed compact sets of $\Omega$ independent of $\ep$, we multiply $\Par{T_\ep^* \gamma_{\indx}(\eta) T_\ep	\uuepI,T_\ep^* \gamma_{\indx}(\eta) T_\ep	\uvepI}$
by a cut-off $\rho \in \Co_0^\infty(\IR^n,[0,1])$ supported in $\Omega$, and consider
\begin{equation}
\label{CIcut}
	u_{\ep,\indx}^I = \rho T_\ep^* \gamma_{\indx}(\eta) T_\ep \uuepI \text{ and } v_{\ep,\indx}^I = \rho T_\ep^* \gamma_{\indx}(\eta) T_\ep \uvepI. \tag{\ref{MPb:gp3}'}
\end{equation}
It is assumed that $\rho(x)= 1$ if $\dist(x,\supp u_\ep^I \cup \supp v_\ep^I) \leq C$ for some $C>0$.
The required estimates
\begin{equation}
 \label{CIsmooth2} 
  \nordvois{T_\ep 	u_{\ep,\indx}^I}{\IR^n \x  \Rcomp^c} = O(\ep^\infty) \text{ and }\nordvois{T_\ep	v_{\ep,\indx}^I}{\IR^n \x  \Rcomp^c} = O(\ep^\infty) 
\tag{\ref{CIsmoothBis}'}
\end{equation}
are fulfilled since Lemma \ref{cutx} from the Appendix implies that
\begin{equation}
\label{estim1}
	\|(1-\rho) T_\ep^* \gamma_{\indx}(\eta) T_\ep \uuepI\|_{L^2_x} \lesssim e^{-C/\ep} \text{ and } 	\|(1-\rho) T_\ep^* \gamma_{\indx}(\eta) T_\ep \uvepI\|_{L^2_x} \lesssim e^{-C/\ep}.
\end{equation}
Using the boundedness of the operator $T_\ep^* \gamma_{\indx} T_\ep$ from $L^2(\IR^n)$ to $L^2(\IR^n)$ and the relations
\begin{equation}
	\label{derivFBI}
 	\dpp_{y_j} T_\ep = T_\ep \dpp_{x_j},\,	\dpp_{x_j} T_\ep^*= T_\ep^* \dpp_{y_j},
\end{equation}
obtained by integrations by parts in the expressions of $T_\ep$ and $T_\ep^*$, one can show that the new initial data $(u_{\ep,\indx}^I,v_{\ep,\indx}^I)$ is also uniformly bounded w.r.t. $\ep$ in $H^1(\Omega)\x L^2(\Omega)$. 

Let $\rhou$ be a cut-off of $\Co_0^\infty(\IR^n,[0,1])$ supported in a compact $K_y\subset\Omega$ and satisfying 
\begin{equation*}
	\rhou(y)=1 \text{ if } \dist(y,\supp \rho)<C \text{ for some } C>0,
\end{equation*}
and $\phiu$ a cut-off of $\Co_0^\infty(\IR^n,[0,1])$ supported in $K_\eta \subset \IR^n\prive{0}$ s.t. $\phiu \equiv 1$ on $\Rcomp$. Without loss of generality, we assume that either the incident ray or the reflected one propagating in the positive sense is in the interior of the domain at the instant $T$ ($x_{0}^{T}(y,\eta) \in \Omega$ or $x_{1}^{T}(y,\eta) \in \Omega$) when $y$ varies in $K_y$ and $\eta$ in $\IR^n\prive{0}$. This is always possible upon reducing $T$ because the number of reflections for initial position and vector speed varying in $K_y \x \Par{\IR^n\prive{0}}$ is uniformly bounded (see Section 2.3 of \cite{BoAkAl09} for similar arguments). And similarly for the instant $-T$ for rays propagating in the negative sense. Then, the IBVP \eqref{MPb:gp1}-\eqref{MPb:gp2} with initial conditions \eqref{CIcut} has a family of approximate solutions $u_{\ep,\indx}^{appr}$ in $\Co^0([0,T],H^1(\Omega)) \cap \Co^1([0,T],L^2(\Omega))$ obtained as a summation of first order beams. A general result using a superposition of beams of any order was proven in \cite{BoAkAl09}, and it reads for first order beams as follows:
\begin{proposition}\emph{(\cite{BoAkAl09}, Theorem 1.1).}
Denote for $t\in [0,T]$ and $x \in \IR^n$ the following superposition of Gaussian beams
\begin{equation*}\begin{split}
&{u}_{\ep,\indx}^{appr}(t,x)\\
=&\ud \ep^{-\tnq+1}c_n\int_{\IR^{2n}} \rhou(y) \phiu(\eta)  T_\ep v_{\ep,\indx}^I(y,\eta) \Big(\somme{k=0,1}{} {\gb^k_\ep}'(t,x,y,\eta)\\
& \qquad \qquad \qquad \quad- \somme{k=0,-1}{} {\gb_\ep^{k}}'(-t,x,y,\eta)\Big)dy d\eta\\
&+ \ud \ep^{-\tnq+1}c_n\int_{\IR^{2n}}\rhou(y) \phiu(\eta) \ep^{-1} T_\ep u_{\ep,\indx}^I(y,\eta) \Big(\somme{k=0,1}{} \gb^{k}_\ep(t,x,y,\eta)\\
&\qquad \qquad \qquad \qquad + \somme{k=0,-1}{} \gb^{k}_\ep(-t,x,y,\eta)\Big) dy d\eta.
\end{split}\end{equation*}
Above, $\gb^0_\ep$, ${\gb^0_\ep}'$ are incident Gaussian beams with the same phase $\psi_0$ satisfying at $t=0$
\begin{equation}
\label{IncPhaseCIBis}
	\psi_0(0,x,y,\eta) = \eta \cdot (x-y)+ \frac i 2 (x-y)^2
\end{equation}
and different amplitudes $a^0_0$, ${a^0_0}'$ satisfying
\begin{equation}
\label{IncAmpCIBis}
	a_0^0(0,x,y,\eta) = \chi_d(x-y),\,
\Par{i\dt \psi_0 {a_0^0}'}(0,x,y,\eta) = \chi_d(x-y) + O(|x-y|).
\end{equation}
$\gb^{\pm 1}_\ep$ and ${\gb^{\pm 1}_\ep}'$ denote the associated reflected beams.
Then $u_{\ep,\indx}^{appr}$ is asymptotic to $u_{\ep,\indx}$ the exact solution of the problem \eqref{MPb:gp1}-\eqref{MPb:gp2} with initial conditions \eqref{CIcut} in the sense that
\begin{align*}
\underset{t\in[0,T]}{\sup}\norhom{u_{\ep,\indx} - u_{\ep,\indx}^{appr}} &\leq  C(\indx,\Omega,T)\, \sqrt \ep,\\
 \text{ and } \underset{t\in[0,T]}{\sup}\nordom{\dt u_{\ep,\indx} - \dt u_{\ep,\indx}^{appr}} &\leq  C(\indx,\Omega,T)\, \sqrt \ep.
\end{align*}
\end{proposition}
We refer to  \cite{BoAkAl09} for further details, and just mention that the proof relies on the use of a family of approximate operators acting from $L^2(\IR^{2n})$ to $L^2(\IR^{n})$. A simple version of the estimate of these operators norms is recalled in Section \ref{AppB} of the Appendix.

\subsubsection{Expression of the phases and amplitudes} \label{fstBeams}

In order to compute the first order beams, we begin by analyzing the relationship between the incident phase and amplitudes, and the Jacobian matrix of the incident flow. A similar relationship involving the reflected phases and amplitudes and the reflected flows will be also given.

The requirement \eqref{PhaseOneBis} for the incident phase implies that 
$$\frac{d}{dt}\Par{ \psi_0(t,x_0^t)} = \dt \psi_0(t,x_0^t) +\dpp_x \psi_0(t,x_0^t) \cdot \dot{x_0}^t=0.$$
Taking into account the initial null value $\psi_0(0,y) = 0$ chosen in  \eqref{IncPhaseCIBis}, one gets a null phase on the ray
\begin{equation*}
	\psi_0(t,x^t_0) =0.
\end{equation*}
With the aim of computing $\dpp_x^2\psi_0(t,x^t_0)$, we note that the Jacobian matrix of the bicharacteristic $F_0^t=D\varphi_0^t$ satisfies the linear ordinary differential system
\begin{equation*}
\left\{
\begin{array}{l}
\frac{d}{dt} F_0^t= J H(x_0^t,\xi_0^t) F_0^t, \\
F_0^0=Id,                                    \\
\end{array}
\right.
\end{equation*}
where $J=\left(\begin{array}{cc}
0      &Id  \\
-Id   &0    \\
\end{array}\right)$ is the standard symplectic matrix. 
Writing $F_0^t$ as
\begin{equation*}
F_0^t=\left(\begin{array}{cc}
D_y x_0^t        &D_\eta x_0^t  \\
D_y \xi_0^t      &D_\eta \xi_0^t     \\
\end{array}\right)
\end{equation*}
leads to the following ordinary differential system on $(U_0^t,V_0^t)=(D_y x_0^t+i D_\eta x_0^t,$
\\
$D_y \xi_0^t+i D_\eta \xi_0^t)$
\begin{align}
\label{eqU}
\frac{d}{dt} U_0^t&=H_{21}(x_0^t,\xi_0^t)U_0^t+H_{22}(x_0^t,\xi_0^t)V_0^t,\\
\label{eqV}
\frac{d}{dt} V_0^t&=-H_{11}(x_0^t,\xi_0^t)U_0^t-H_{12}(x_0^t,\xi_0^t)V_0^t.
\end{align}
Note that $F_0^t$ is a symplectic matrix, i.e.
\begin{equation*} 
(F_0^t)^TJ F_0^t=J.
\end{equation*}
Using the symmetry of the following matrices
\begin{equation*}
(D_y x_0^t)^T D_y \xi_0^t, (D_\eta x_0^t)^T D_\eta \xi_0^t, D_y x_0^t (D_\eta x_0^t)^T,\text{ and }D_y \xi_0^t  (D_\eta \xi_0^t)^T 
\end{equation*}
and the relations
\begin{equation*}
(D_y x_0^t)^T D_\eta \xi_0^t -(D_y \xi_0^t)^T D_\eta x_0^t =Id \text{ and }D_y x_0^t (D_\eta \xi_0^t)^T-{D_\eta x_0^t} (D_y \xi_0^t)^T =Id,
\end{equation*}
one has
\begin{equation}
\label{propos}
(U_0^t)^T V_0^t = (V_0^t)^T U_0^t, \, (V_0^t)^T\bar U_0^t - (U_0^t)^T\bar V_0^t=2i Id \text{ and }U_0^t \text{ is invertible}.
\end{equation}
Putting together \eqref{eqU}, \eqref{eqV} and \eqref{propos} shows that $V_0^t(U_0^t)^{-1}$ is a symmetric matrix with a positive definite imaginary part and fulfills the Riccati equation \eqref{riccatiBis} with initial value $iId$. Since this is the initial condition for $\dpp_x^2 \psi_0(t,x_0^t)$ given in \eqref{IncPhaseCIBis}, it follows that
\begin{equation}
\label{incHess}
	\dpp_x^2 \psi_0(t,x_0^t) = V_0^t(U_0^t)^{-1}.
\end{equation}
The incident beams amplitudes are computed as follows. Using \eqref{PhaseOneBis} and the Hamiltonian system satisfied by $(x_0^t,\xi_0^t)$, the equation \eqref{evolBis} at order zero yields the following transport equation for the value of the amplitude on the ray \cite{KaKuLa}
\begin{equation}
\label{eqAmpl}
\frac{d}{dt}\Par{{a_0^{0}}^{(')}(t,x_0^t)}+\ud\Tr\Par{H_{21}(x_0^t,\xi_0^t)+H_{22}(x_0^t,\xi_0^t)\dpp_x^2\psi_0(t,x_0^t)}{a_0^0}^{(')}(t,x_0^t)=0,
\end{equation}
which may be written using the matrices $U_0^t$ and $V_0^t$ as
\begin{equation*}
\frac{d}{dt}\Par{{a_0^0}^{(')}(t,x_0^t)}+\ud\Tr\Croch{\Par{H_{21}(x_0^t,\xi_0^t)U_0^t+H_{22}(x_0^t,\xi_0^t)V_0^t}(U_0^t)^{-1}}{a_0^0}^{(')}(t,x_0^t)=0. 
\end{equation*}
The time evolution for $U_0^t$, see \eqref{eqU}, combined with the choice of the initial values ${a_0^0}(0,y) =1$ and ${a_0^0}'(0,y) = (-ic(y)|\eta|)^{-1}$ from \eqref{IncAmpCIBis}, yields
\begin{equation*}
a_0^0(t,x_0^t)= \Par{\det U_0^t}^{-\ud} \text{ and } {a_0^0}'(t,x_0^t)= i (c(y)|\eta|)^{-1} \Par{\det U_0^t}^{-\ud}.
\end{equation*}
Above the square root is defined by continuity in $t$ from $1$ at $t=0$.

The expression of the reflected phases $\psi_k$, $k=\pm 1$, is similar to the incident phase. In fact, since $\frac{d}{dt} \Par{\psi_k(t,x_k^t)} = 0$ and $\psi_k(T_k,x_0^{T_k})=\psi_{0}(T_k,x_0^{T_k})$ because of the requirement \ref{refPhaseBis} p.\pageref{refPhaseBis}, we get
\begin{equation*}
	\psi_k(t,x_k^t) =0.
\end{equation*}
We then apply the general relation between incident and reflected beams phases given in Lemma \ref{Inc&Ref} from the Appendix, to compute the Hessian matrices of $\psi_{\pm 1}$ on the rays. As far as we know, the result stated in this Lemma is particularly simple enough so that we stated it in the Appendix \ref{AppA}. The matrices $\dpp_x^2\psi_{\pm 1}(t,x_{\pm 1}^{\pm t})$ can also be computed by solving the Riccati equations with the proper values at the instants of reflections $t = T_{\pm 1}$ (see eg. \cite{Norris2,TaEnTs09}).
One gets (see Appendix \ref{AppA})
\begin{equation*}
\dpp_x^2 \psi_k(t,x_k^t)=V_k^t (U_k^t)^{-1} \text{ where }U_k^t=D_y x_k^t+i D_\eta x_k^t \text{ and }V_k^t=D_y \xi_k^t+i D_\eta \xi_k^t.
\end{equation*}
As $\varphi_k^t$ is symplectic, $(U_k^t,V_k^t)$ share the same properties \eqref{propos} as $(U_0^t,V_0^t)$
\begin{equation}
\label{propk}
(U_k^t)^T V_k^t = (V_k^t)^T U_k^t,\, (V_k^t)^T\bar U_k^t - (U_k^t)^T\bar V_k^t=2iId \text{ and } U_k^t \text{ is invertible}.
\end{equation}
The reflected amplitudes evaluated on the rays have an expression similar to the incident amplitudes (see Appendix \ref{AppA})
\begin{equation*}
a_0^{k}(t,x_k^t) = -s i \Par{\det U_k^t}^{-\ud} \text{ and } {a_0^{k}}'(t,x_k^t)= s (c(y)|\eta|)^{-1} \Par{\det U_k^t}^{-\ud} \text{ for } k= \pm 1,
\end{equation*}
where the square root is defined by continuity from $i \Par{\det U_0^{T_k}}^{-\ud}$ at $t=T_k$, $s=-1$ for the Dirichlet boundary condition and $s=1$ for the Neumann condition.

We summarize the previous form of the beams in the following result:
\begin{lemma}
\label{FirstGB}
For $k = 0, \pm 1$, the incident and reflected beams $\gb_\ep^k$ have the form
\begin{equation*}
	{\gb_\ep^k}^{(')}(t,x) = \beta_k  \chi_d(x-x_k^t) a_k^{(')}(t) e^{i \psi_k/\ep},
\end{equation*}
with
\begin{align*}
	\beta_0=1 , \, \beta_{1} = \beta_{-1} = -s i, \\
	 a_k(t)=  [\det U_k^t]^{-\ud} ,\,  a_k'(t) =  i (c(y)|\eta|)^{-1} [\det U_k^t]^{-\ud}, \\
	\psi_k = \xi_k^t\cdot(x-x_k^t)+\frac{i}{2}(x-x_k^t)\cdot \Lambda_k(t)(x-x_k^t),\, \text{and } \Lambda_k(t) = -i V_k^t (U_k^t)^{-1}.
\end{align*}
\end{lemma}

\subsubsection{Gaussian integrals}

It follows that the approximate solution $u_{\ep,\indx}^{appr}$ has the form (recall the dependence of Gaussian beams w.r.t. variables $(y,\eta )$)
\begin{equation*}\begin{split}
&{u}^{appr}_{\ep,\indx}(t,x)\\
=&\ud\ep^{-\tnq+1}c_n\int_{\IR^{2n}} \rhou(y) \phiu(\eta)  \somme{k=0,1}{} \chi_d(x-x_k^t)\beta_k p_{\ep,k}(t,y,\eta) e^{i\psi_k(t,x,y,\eta)/\ep} dy d\eta  \\
&+\ud\ep^{-\tnq+1}c_n\int_{\IR^{2n}} \rhou(y) \phiu(\eta)  \sum_{k=0,-1} \chi_d(x-x_{k}^{-t}) \beta_k q_{\ep,k}(-t,y,\eta)\\ 
& \qquad \qquad \qquad \qquad \qquad \qquad \qquad \qquad e^{i\psi_{k}(-t,x,y,\eta)/\ep} dy d\eta,
\end{split}\end{equation*}
with 
\begin{align*}
           p_{\ep,k}(t,y,\eta) &= a_k(t,y,\eta)\ep^{-1} T_\ep u_{\ep,\indx}^I(y,\eta)+a_k'(t,y,\eta) T_\ep v_{\ep,\indx}^I(y,\eta),  \nonumber\\
\text{and }q_{\ep,k}(t,y,\eta) &= a_k(t,y,\eta)\ep^{-1} T_\ep u_{\ep,\indx}^I(y,\eta)-a_k'(t,y,\eta) T_\ep v_{\ep,\indx}^I(y,\eta). 
\end{align*}

Because of the phases expression given in \eqref{quadPsi}, time and spatial derivatives of ${u}^{appr}_{\ep,\indx}$ may be written as a sum of integrals of the form
\begin{equation*}\begin{split}
 \ep^{-\tnq}  \int_{\IR^{2n}} &\rhou(y)  \phiu(\eta) f_\ep(y,\eta)  \ep^{j}(x-x_k^{t})^\alp r_{j,\alp}^k(t,x,y,\eta)\\
 & e^{i\psi_k( t,x,y,\eta) /\ep} dy d\eta , \, j,k=0,1, \, |\alp|\leq 2,
\end{split}\end{equation*}
arising from differentiation of $\gb_\ep^{0}(t,.)$ and $\gb_\ep^{1}(t,.)$. Other terms of the same form originate from derivatives of $\gb_\ep^{0}(-t,.)$ and $\gb_\ep^{-1}(-t,.)$.
$f_\ep$ stands for $\ep^{-1} T_\ep u_{\ep,\indx}^I$ or $T_\ep v_{\ep,\indx}^I$ and $r_{j,\alp}^k$ are smooth functions vanishing for $|x-x_k^t| \geq d$.
\\
For a function $f$ depending on $(t,x,z,\theta) \in \IR^{n+1} \x \Bic$ and $k=0,\pm 1$, let
\begin{equation*}
\trak{f}(t,x,z,\theta) =  f(t,x,\Par{\varphi_k^{t}}^{-1}(z,\theta)). 
\end{equation*}
Set $K_{z,\theta}^k(t) = \varphi_k^t(K_y\x K_\eta)$. Let $\Pi_k(t)$ be a cut-off of $\Co_0^\infty(\IR^{2n}, [0,1])$ supported in $\Bic$ and satisfying $\Pi_k(t) \equiv 1$ on $K_{z,\theta}^k(t)$.
The volume preserving change of variables 
\begin{equation*}
	(z,\theta)=\varphi_k^{t}(y,\eta)
\end{equation*}
transforms the previous integrals as
\begin{equation}
\label{dtutilde}
 \ep^{-\tnq}    \int_{\IR^{2n}} \Pi_k(t) \trak{\rhou \ox \phiu} \trak{ f_\ep} \ep^{j}(x-z)^\alp\trak{\Par{  r_{j,\alp}^k}}  e^{i\trak{\psi_k} (t,x,z,\theta) /\ep} dz d\theta,\, j,k=0,1, \, |\alp|\leq 2.
\end{equation}
We can write the leading terms obtained for $j=0$ and $\alp=0$ using Gaussian type integrals $I_\ep(h,\Phi)$ defined as
\begin{equation*}
I_\ep(h,\Phi)(t,x)=\ep^{-\tnq}c_n\int_{\IR^{2n}} h(t,z,\theta) e^{i\Phi(t,x,z,\theta)/\ep} dz d\theta,
\end{equation*}
for a given phase function $\Phi \in \Co^\infty(\IR^{n+1}_{t,x} \x \Bic,\IC)$ polynomial of order $2$ in $x-z$ and satisfying, for $t\in [0,T]$ and $(z,\theta)\in \Bic$
\begin{equation}
\label{propPhaseBis}
\Phi(t,z,z,\theta) \text{ is real},\,\dpp_x\Phi(t,z,z,\theta)=\theta,\,\Im \dpp_x^2\Phi(t,z,z,\theta) \text{ is positive definite},
\end{equation}
and a given amplitude function $h\in \Co^0([0,T],L^2(\IR^{2n}_{z,\theta}))$ supported for every fixed \\
$t \in [0,T]$  in a compact of $\Bic$.
By Proposition \ref{appOp} in the Appendix, one has
$$\|\int_{\IR^{2n}} h(t,z,\theta) \chi(x-z) e^{i\Phi(t,x,z,\theta)/\ep} dz d\theta \|_{L^2_x} \lesssim  \|h(t,.)\|_{L^2_{z,\theta}}.$$
Noticing that $e^{i\Phi/\ep}$ is exponentially decreasing for $|x-z| \geq 1$, one can use the following crude estimate 
\begin{equation}
\label{noCutof}
\|\int_{|x-z|\geq a}   h(t,z,\theta)  e^{i\Phi(t,x,z,\theta)/\ep} dz d\theta\|_{L^2_x} \lesssim  e^{-C/\ep}\|h(t,.)\|_{L^2_{z,\theta}} 
\text{ for }a>0
\end{equation}
to deduce that $I_\ep(h,\Phi)(t,.)$ is uniformly bounded  w.r.t. $\ep$ in $L^2_x$. The same notation $I_\ep(h,\Phi)$ will be also used for a vector valued function $h$. 

The contribution of the terms \eqref{dtutilde} with $j=1$ or $|\alp| \geq 1$ to the derivatives of $u^{appr}_{\ep,\indx}$ is of order $\sqrt \ep$ as stated in the following Lemma, whose proof is given Appendix \ref{AppB} and relies on the approximation operators defined therein.
\begin{lemma}
\label{dtudxu}
$\dt u^{appr}_{\ep,\indx}(t,.)$ is uniformly bounded w.r.t. $\ep$ in $L^2(\IR^n)$ and satisfies
\begin{equation*}
\begin{split}
\dt u^{appr}_{\ep,\indx}(t,x) =& \ud \Par{ v^+_{t,\ep}(t,x) - v^-_{t,\ep}(-t,x)} \\
&+O(\sqrt\ep) \text{ in } L^2(\IR^n) \text{ uniformly w.r.t. } t \in [0,T],	
\end{split}
\end{equation*}
where $(v^+_{t,\ep})$ and $(v^-_{t,\ep})$ are sequences of $L^2(\IR^n)$ uniformly bounded w.r.t. $\ep$ given by 
\begin{align*}
v^+_{t,\ep} &= \somme{k=0,1}{} \beta_k I_\ep(-i c(z)|\theta| \Pi_k \trak{\rhou\ox \phiu} \trak{p_{\ep,k}} , \trak{\psi_k}), \\
v^-_{t,\ep} &= \somme{k=0,-1}{}\beta_k I_\ep(-i c(z)|\theta| \Pi_k  \trak{\rhou\ox \phiu} \trak{q_{\ep,k}}, \trak{\psi_k}).	
\end{align*}
Likewise, $\dpp_x u^{appr}_{\ep,\indx}(t,.)$ is uniformly bounded w.r.t. $\ep$ in $L^2(\IR^n)^n$ and satisfies
\begin{equation*}
\begin{split}
\dpp_x u^{appr}_{\ep,\indx}(t,x) =& \ud \Par{v^+_{x,\ep}(t,x) + v^-_{x,\ep}(-t,x)} \\
&+O(\sqrt\ep) \text{ in } L^2(\IR^n)^n \text{ uniformly w.r.t. } t \in [0,T],	
\end{split}
\end{equation*}
where $(v^+_{x,\ep})$ and $(v^-_{x,\ep})$ are sequences of $L^2(\IR^n)^n$ uniformly bounded w.r.t. $\ep$ given by 
\begin{align*}
 v^+_{x,\ep} &= \somme{k=0,1}{} \beta_k I_\ep(i\theta \Pi_k \trak{\rhou\ox \phiu} \trak{p_{\ep,k}} , \trak{\psi_k}), \\
 v^-_{x,\ep} &= \somme{k=0,-1}{} \beta_k I_\ep(i\theta \Pi_k \trak{\rhou\ox \phiu} \trak{q_{\ep,k}} , \trak{\psi_k}).	
\end{align*}
\end{lemma}

%% file: Ch2Part2.tex
We now compute the scalar measures associated to the sequences $\Par{\dt u^{appr}_{\ep,\indx}(t,.)}$ and $\Par{c \dpp_x u^{appr}_{\ep,\indx}(t,.)}$.
As $|\beta_k|=1$, the Wigner transform associated to $ \Par{v_{t,\ep}^+(t,.)}$ is a finite sum of terms of the form
\begin{equation*}
 w_\ep\Par{I_\ep(f_{t,\ep}^k ,\Phi_k)(t,.),\  I_\ep(f_{t,\ep}^l , \Phi_l)(t,.)},
\end{equation*}
where $k,l=0,1$, $f_{t,\ep}^k = c|\theta|\Pi_k \trak{\rhou\ox \phiu} \trak{p_{\ep,k}}$ and $\Phi_{k} = \trak{\psi_{k}}$.
As regards the Wigner transforms associated to $\Par{c v_{x,\ep}^+( t,.)}$, since $c$ is uniformly continuous on $\IR^n$, one has
by a classical result (\cite{GeMaMaPo}, p.8)
\begin{equation}
\label{WigSmoo}
 w_\ep\Par{ c v_{x,\ep}^+ (t,.), c v_{x,\ep}^+ (t,.)} \approx c^2 w_\ep\Par{v_{x,\ep}^+ (t,.), v_{x,\ep}^+ (t,.)} \text{ in } \IR^{2n},
\end{equation}
and therefore the involved quantities have the form
\begin{equation*}
 c^2 w_\ep\Par{I_\ep(f_{x,\ep}^k  , \Phi_k)(t,.), I_\ep(f_{x,\ep}^l  , \Phi_l)(t,.)},
\end{equation*}
with $f_{x,\ep}^k = \theta\Pi_k \trak{\rhou\ox \phiu} \trak{p_{\ep,k}}$.

Similarly, we  define for $k=0,-1$ the sequences $g_{t,\ep}^k = c|\theta|\Pi_k \trak{\rhou\ox \phiu}  \trak{q_{\ep,k}}$, which are needed when considering
the Wigner transform associated with $\Par{c v_{t,\ep}^-(-t,.)}$ and the cross Wigner transform between $\Par{c v_{t,\ep}^+(t,.)}$ and $\Par{c v_{t,\ep}^-(-t,.)}$, as well as $g_{x,\ep}^k = \theta \Pi_k \trak{\rhou\ox \phiu} \trak{q_{\ep,k}}$.
Then, forgetting the powers of $\ep$ factors, all the previous Wigner transforms tested on cut-off functions have the form 
\begin{equation}
	\label{intForm1}
	\int_{\IR^{6n}} \trak{T_\ep \kappa_\ep}(z,\theta) \tral{\overline{T_\ep \tau_\ep}}(z',\theta') b_1^{k,l}(z,\theta,z',\theta',x,v) e^{i\Psi_1^{k,l}(z,\theta,z',\theta',x,v)/\ep} dz d\theta dz' d\theta' dx dv,
\end{equation}
with $\kappa_\ep,\tau_\ep = \ep^{-1} u_{\ep,\indx}^I,v_{\ep,\indx}^I$ and $k,l=0,\pm 1$, or after expanding the FBI transforms
\begin{equation}
	\label{intForm2}
	\int_{\IR^{8n}}  \kappa_\ep(w)  \bar \tau_\ep(w') b_2^{k,l}(z,\theta,z',\theta',x,v) e^{i \Psi_2^{k,l}(w,w',z,\theta,z',\theta',x,v)/\ep} dw dw' dz d\theta dz' d\theta' dx dv.
\end{equation}
This type of oscillating integrals is traditionally estimated by the stationary phase theorem. For example, this method was successfully used in \cite{Castella} for the computation of a Wigner measure for smooth data. There the phase was complex and its Hessian matrix restricted to the stationary set was assumed to be non-degenerate in the normal direction to this set. However, in our case, the amplitude is not smooth as no such assumption was made on $\uuepI$ and $\uvepI$, and we cannot estimate immediately the global integral \eqref{intForm2} by the same techniques. One possibility of solving this issue would be to resort to the stationary phase theorem with a complex phase depending on parameters for estimating
$$	\int_{\IR^{6n}}  b_2(z,\theta,z',\theta',x,v) e^{i \Psi_2(w,w',z,\theta,z',\theta',x,v)/\ep} dz d\theta dz' d\theta' dx dv,$$
and then study the whole integral involving $\kappa_\ep(w)  \bar \tau_\ep(w')$.

An alternative method was used in \cite{Robinson}, where an integral of the form \eqref{intForm1} associated to the Wigner transform for the Schr\"odinger equation with a WKB initial condition was simplified by elementary computations into an integral over $\IR^{4n}$. 

Though the method therein faced difficulties in deducing the exact relation between the Wigner measure of the solution and of the initial data, we adapt the result of \cite{Robinson} to our problem in Section \ref{Rob} and complete the analysis to prove the propagation along the flow of the microlocal energy density of $u^{appr}_{\ep,\indx}$ as $\ep \rightarrow 0$ in Section \ref{GBSwigner}. The proof is simple and elementary and the computations are made in an explicit way.
Section \ref{truncation} is devoted to the Wigner measures associated to the derivatives of $u_\ep$ the exact solution of \eqref{MPb:gp}. 

\subsection{Wigner transform for Gaussian integrals} \label{Rob}

The sequences $(f_{t,\ep}^k)$, $(f_{x,\ep}^k)$, $(g_{t,\ep}^l)$ and $(g_{x,\ep}^l)$ are uniformly bounded w.r.t. $\ep$ in $L^2(\IR^{2n})$ and their supports are contained in a fixed compact independent of $\ep$. Slight modifications of the computations of \cite{Robinson} lead to the following more general result:
\begin{lemma}
\label{robin}
Let $(f_\ep)$ and $(g_\ep)$ be sequences uniformly bounded in $L^2(\IR^{2n})$ and having their supports contained in a fixed compact independent of $\ep$. Let $F$ be an open set containing $\supp f_\ep \cup \supp g_\ep$ and $\Phi$, $\Psi$ be phase functions in $\Co^\infty(\IR_x^n \x F,\IC)$   satisfying 
\begin{align*}
\Phi(x,z,\theta)   &= r_\Phi(z,\theta)   +  \theta \cdot (x-z)  + \frac i 2 (x-z) \cdot H_\Phi(z,\theta)   (x-z),   \\
\Psi(x,z',\theta') &= r_\Psi(z',\theta') +  \theta' \cdot (x-z') + \frac i 2 (x-z')\cdot H_\Psi(z',\theta') (x-z'),
\end{align*}
for $x \in \IR^n$ and $(z,\theta), (z',\theta')\in F$, with $r_\Phi, r_\Psi \in \Co^{\infty}(F, \IR)$ and the matrices $ H_\Phi, H_\Psi \in \Co^{\infty}(F, \Ma_n(\IC))$ having positive definite real parts.
Then for $\phi \in \Co_0^\infty(F,\IR)$ 
\begin{equation*}\begin{split}
&<w_\ep\Par{I_\ep(f_\ep,\Phi),I_\ep(g_\ep,\Psi)},\phi>\\ 
=& \int_{\IR^{4n}} \phi(s,\sigma)   f_\ep(s+\sqrt \ep r,\sigma+\sqrt \ep\delta) g_\ep^*(s-\sqrt \ep r,\sigma-\sqrt \ep\delta)\\
&\qquad \, A(\Phi,\Psi)(s,\sigma) e^{i\Theta_\ep(\Phi,\Psi)(s,\sigma,r,\delta)}   dr d\delta ds d\sigma + o(1),
\end{split}\end{equation*}
where 
\begin{equation*}
A(\Phi,\Psi)(s,\sigma) = c_n^2 2^{\frac{5n}{2}}  \pi^{\nd} \Croch{\det\Par{H_\Phi(s,\sigma) + \overline{H_\Psi}(s,\sigma)}}^{-\ud}, 
\end{equation*}
and
\begin{equation*}\begin{split}
\Theta_\ep(\Phi,\Psi)(s,\sigma,r,\delta) =& r_\Phi(s+\sqrt \ep r,\sigma+\sqrt \ep\delta)/\ep -  r_\Psi(s-\sqrt \ep r,\sigma-\sqrt \ep\delta)/\ep\\ 
& -2\sigma \cdot r/\sqrt \ep + i (r,\delta)  \cdot  Q\Par{H_\Phi(s,\sigma) ,\overline{H_\Psi}(s,\sigma)}(r,\delta).
\end{split}\end{equation*}
The matrix $Q\Par{H_\Phi(s,\sigma) ,\overline{H_\Psi}(s,\sigma)}$ and the square root are defined in Lemma \ref{TFGG}.
\end{lemma}

\begin{proof}
It consists in two steps. Firstly, the Fourier transform of a Gaussian type function is computed explicitly. Then, a Gaussian approximation is used for several smooth functions appearing in the Wigner transform integral. 

For simplicity we denote $u(x,z,\theta)$ by $u$ and $u(x,z',\theta')$ by $u'$ when integrating w.r.t. $z,\theta,z',\theta'$. We also omit the index $\ep$ in the notation of $f_\ep$ and $g_\ep$.
\\
\textbf{Step 1. Fourier transform.}
We note that the Wigner transform at point $(x,\xi) \in \IR^{2n}$ may be written as
\begin{equation*}\begin{split}
&w_\ep\Par{I_\ep(f,\Phi),I_\ep(g,\Psi)}(x,\xi)\\
=& \pi^{-n} c_n^2 \ep^{-\frac{5n}{2}}  \int_{\IR^{5n}} f  {g^*}'e^{i r_\Phi/\ep-i r_\Psi'/\ep +ix \cdot (\theta-\theta')/\ep + i(\theta'\cdot z'-\theta \cdot z)/\ep} \\
&\qquad \qquad \qquad \quad \,\Fo_{v}\big(e^{-(v+x-z)\cdot H_\Phi (v+x-z)/(2\ep)}\\
&\qquad \qquad \qquad \qquad \, \, \x e^{-(v-x+z')\cdot  {\overline{H_\Psi}}'(v-x+z')/(2\ep)}\big)\Par{(2\xi-\theta-\theta')/\ep}  \\
&\qquad \qquad \qquad \quad \, dv dz dz'  d\theta d\theta'.
\end{split}\end{equation*}
The Fourier transform of a Gaussian functions product is given by the following Lemma, whose proof is postponed to the end of this Section:
\begin{lemma}
\label{TFGG}
Let $a,b\in \IR^d$ and $M,N\in \Ma_d(\IC)$ symmetric matrices with positive definite real parts, then
\begin{equation*}\begin{split}
&\Fo_{x}\Par{e^{\GasE{x-a}{M}{2}} e^{\GasE{x-b}{N}{2}}}(\xi)\\
=&(2\pi)^{\frac{d}{2}} \Par{\det (M+N)}^{-\ud}  e^{-i\xi\cdot(b+a)/2 \GasE{b-a,\xi}{Q(M,N)}{4}},
\end{split}\end{equation*}
where $Q(M,N)$ is the symmetric symplectic matrix given by
\begin{equation*}
Q(M,N)=\left(\begin{array}{cc}
 2M(M+N)^{-1}N            &         i(N-M)(M+N)^{-1}\\
i(M+N)^{-1}(N-M)          &         2(M+N)^{-1}     \\                                  
\end{array}
\right),
\end{equation*}
and the square root is defined as explained in Section 3.4 of \cite{HormanderPDO1}.
\\
Moreover, $Q(M,N) A(M,N)=B(M,N)$  with $A(M,N)=\left(\begin{array}{cc} Id  &   Id\\ -i N   &   i M \\ \end{array}\right)$
 and $B(M,N)=\left(\begin{array}{cc} N  & M\\ -i Id  &  i Id  \\ \end{array} \right),$ and $Q(M,N)$ has a positive definite real part
\begin{equation*}
\Re Q(M,N)=2{A}(M,N)^{*-1} \left(\begin{array}{cc}
 \Re N            &         0         \\
0                 &         \Re M     \\                                  
\end{array}
\right)A(M,N)^{-1}.
\end{equation*}
\end{lemma}
Hence
\begin{equation*}\begin{split}
&w_\ep\Par{I_\ep(f,\Phi),I_\ep(g,\Psi)}(x,\xi)\\
=& c_n^2 2^{\nd} \pi^{-\nd} \ep^{-2n}  \int_{\IR^{4n}} f  {g^*}' \Par{\det(H_\Phi+{\overline{H_\Psi}}')}^{-\ud} e^{i r_\Phi/\ep-i r_\Psi'/\ep} \\
&\qquad \qquad \qquad \qquad \,\, e^{i(\theta+\theta'-2\xi) \cdot (z-z')/(2\ep)+i(\theta-\theta')\cdot x /\ep+ i (\theta' \cdot z'-\theta \cdot z)/\ep} \\
&\qquad \qquad \qquad \qquad \,\, e^{\GasE{2x-z-z',2\xi-\theta-\theta'}{Q(H_\Phi,{\overline{H_\Psi}}')}{(4\ep)}} dz dz'  d\theta d\theta'.
\end{split}\end{equation*}
Making the changes of variables 
\begin{equation*}
	(z,z')=(s+\sqrt \epsilon r,s-\sqrt \epsilon r), \, (\theta,\theta')=(\sigma+\sqrt \epsilon \delta,\sigma-\sqrt \epsilon \delta),
\end{equation*}
and writing $f_+$ for $f(s+\sqrt \ep r ,\sigma+\sqrt \ep \delta)$ and $g_-$ for $g(s-\sqrt \ep r ,\sigma-\sqrt \ep \delta)$ leads to
\begin{equation*}\begin{split}
&w_\ep\Par{I_\ep(f,\Phi),I_\ep(g,\Psi)}(x,\xi)\\
=&  c_n^2 2^{\frac{5n}{2}} \pi^{-\nd} \ep^{-n} \int_{\IR^{4n}} f_+ g^*_- \Par{\det({H_\Phi}_+ + \bar{H}_{\Psi}{\,}_-)}^{-\ud}  e^{i{r_\Phi }_+/\ep-i{r_\Psi}_-/\ep + 2i\delta \cdot (x-s)/\sqrt \ep }\\
& \qquad \qquad \qquad \qquad \,\, e^{- 2 i\xi \cdot r/\sqrt \ep \GasE{x-s,\xi-\sigma}{Q({H_\Phi}_+ ,\overline{H_\Psi}_- )}{\ep} }dr d\delta ds d\sigma.
\end{split}\end{equation*}
\textbf{Step 2. Gaussian approximations.}
Taking the duality product of the Wigner transform with a test function $\phi \in \Co_0^\infty(F,\IR)$, and after setting $(x',\xi')=(x-s,$ $\xi-\sigma)/\sqrt \ep$, one has
\begin{equation}
\label{WigTrans2}
\begin{split}
&<w_\ep\Par{I_\ep(f,\Phi),I_\ep(g,\Psi)},\phi> \\
=& c_n^2 2^{\frac{5n}{2}} \pi^{-\nd} \int_{\IR^{6n}}\phi(s+\sqrt \ep x',\sigma+\sqrt \ep \xi')  f_+ g^*_- \Par{\det({H_\Phi}_+ + \overline{H_\Psi}_-)}^{-\ud} e^{i{r_\Phi}_+/\ep-i{r_\Psi}_-/\ep} \\
&\qquad \qquad \qquad \quad e^{- 2 i \sigma \cdot r /\sqrt \ep + 2i(x',\xi')\cdot(\delta,-r)\GaE{x',\xi'}{Q({H_\Phi}_+,\overline{H_\Psi}_-)}}  dx' d\xi'  dr d\delta ds d\sigma. 
\end{split}
\end{equation}
Let $\rhou_f$ and $\rhou_g$ be cut-off functions supported in $F$ s.t. $\rhou_f \equiv 1$ on a fixed compact containing $\supp f$ and $\rhou_g \equiv 1$ on a fixed compact containing $\supp g$, and consider
\begin{equation*}\begin{split}
b_\ep:(x',\xi',s,\sigma,r,\delta)\mapsto\big(&\phi(s+\sqrt \ep x',\sigma+\sqrt \ep \xi')\\
&-\phi(s,\sigma)\big) {\rhou_f}_+ {\rhou_g}_- e^{\GaE{x',\xi'}{Q({H_\Phi}_+,\overline{H_\Psi}_-)}}.	
\end{split}\end{equation*}
The r.h.s. of \eqref{WigTrans2} may be written as
\begin{equation}
\label{WigTrans3}
\begin{split}
&<w_\ep\Par{I_\ep(f,\Phi),I_\ep(g,\Psi)},\phi>\\
=&  c_n^2 2^{\frac{5n}{2}} \pi^{-\nd}  \int_{\IR^{6n}} \phi(s,\sigma) f_+ g^*_-  \Par{\det({H_\Phi}_+ + \overline{H_\Psi}_-)}^{-\ud} e^{i{r_\Phi}_+/\ep-i{r_\Psi}_-/\ep  - 2 i\sigma \cdot r / \sqrt \ep}\\
&\qquad \qquad \qquad \quad  e^{ 2 i(x',\xi') \cdot (\delta,-r) \GaE{x',\xi'}{Q({H_\Phi}_+,\overline{H_\Psi}_-)}}  dx' d\xi'  dr d\delta ds d\sigma \\
&+  c_n^2 2^{\frac{5n}{2}} \pi^{-\nd}  \int_{\IR^{4n}}\Par{\det({H_\Phi}_+ + \overline{H_\Psi}_-)}^{-\ud} f_+ g^*_-  e^{i{r_\Phi}_+/\ep-i{r_\Psi}_-/\ep- 2 i\sigma \cdot r / \sqrt \ep} \\
&\qquad \qquad \qquad \qquad   \Fo_{(x',\xi')}{b_\ep}(-2\delta,2r,s,\sigma,r,\delta) dr d\delta ds d\sigma.
\end{split}
\end{equation}
Leibnitz formula yields for a multiindex $\alp$
\begin{equation*}\begin{split}
	&\dpp^\alp_{x',\xi'} b_\ep(x',\xi',s,\sigma,r,\delta)  \\
	=&{\rhou_f}_+ {\rhou_g}_-  \Par{\phi(s+\sqrt \ep x',\sigma+\sqrt \ep  \xi')-\phi(s,\sigma)}   \dpp^\alp_{x',\xi'} \Par{e^{\GaE{x',\xi'}{Q({H_\Phi}_+,\overline{H_\Psi}_-)}}}\\
  &+{\rhou_f}_+ {\rhou_g}_-		 \somme{\beta+\gamma =\alp,\beta\ne 0}{}  C(\beta,\gamma) \ep^{\frac{|\beta|}{2}}\dpp^\beta_{x',\xi'}\Par{\phi(s+\sqrt \ep x',\sigma+\sqrt \ep \xi')}\\
  &\qquad \qquad \qquad \qquad \quad \,\,\x \dpp^\gamma_{x',\xi'} \Par{e^{\GaE{x',\xi'}{Q({H_\Phi}_+,\overline{H_\Psi}_-)}}} .
\end{split}\end{equation*}
As $(s+\sqrt\ep r,\sigma+\sqrt\ep \delta)$ varies in $\supp \rhou_f$ and $(s-\sqrt\ep r,\sigma-\sqrt\ep \delta) $ varies in $\supp \rhou_g$,
one can find by continuity a constant $C >0$ s.t.
\begin{equation*}
	\Re{Q({H_\Phi}_+,\overline{H_\Psi}_-)} \geq  C Id  \text{ on } \supp ({\rhou_f}_+ {\rhou_g}_-).	
\end{equation*}
Since
\begin{equation}
\label{boundSupp}
(s,\sigma) \text{ and } \sqrt\ep(r,\delta) \text{ are bounded on } \supp ({\rhou_f}_+ {\rhou_g}_-),
\end{equation}
it follows that there exists a constant $C'>0$ s.t.
\begin{equation*}
|\dpp^\alp_{x',\xi'} b_\ep(x',\xi',s,\sigma,r,\delta)| \lesssim \sqrt \ep e^{-C' (x',\xi')^2} \text{ for all } (x',\xi',s,\sigma,r,\delta),
\end{equation*}
which leads to
\begin{equation*}
	|\Fo_{(x',\xi')}{b_\ep}(-2\delta,2r,s,\sigma,r,\delta)| \lesssim \sqrt \ep (1+(r,\delta)^2)^{-n-1} \text{ for all } (s,\sigma,r,\delta). 
\end{equation*}
The second integral in the r.h.s. of \eqref{WigTrans3} is then dominated by 
\begin{equation*}
\sqrt \ep   \int_{\IR^{4n}}|f_+| |g_-| (1+(r,\delta)^2)^{-n-1} dr d\delta ds d\sigma. 
\end{equation*}
We deduce by Cauchy-Schwartz inequality w.r.t. $s,\sigma$ that
\begin{equation*}\begin{split}
\Big|&<w_\ep\Par{I_\ep(f,\Phi),I_\ep(g,\Psi)},\phi> \\
& -  c_n^2 2^{\frac{5n}{2}}  \pi^{\nd} \int_{\IR^{4n}} \phi(s,\sigma) \Par{\det({H_\Phi}_+ + \overline{H_\Psi}_-)}^{-\ud} f_+ g^*_- e^{i{r_\Phi}_+/\ep-i{r_\Psi}_-/\ep -2 i\sigma.r/\sqrt \ep} \\
&\,\, \quad \qquad \qquad \qquad e^{\GaE{\delta,-r}{Q({H_\Phi}_+,\overline{H_\Psi}_-)^{-1}}} dr d\delta  ds d\sigma\Big| \lesssim \sqrt \ep  \nord{f} \nord{g}, 
\end{split}\end{equation*}
where we used $\det Q({H_\Phi}_+,\overline{H_\Psi}_-)= 1$ since $Q({H_\Phi}_+,\overline{H_\Psi}_-)$ is symplectic.

Next, we extend $H_\Phi$ and $H_\Psi$ outside $F$ as $\lambda H_\Phi + (1-\lambda) Id$ and $\lambda H_\Psi + (1-\lambda) Id$ by using a cut-off $\lambda \in \Co_0^\infty(\IR^{2n},[0,1])$ supported in $F$ s.t. $\lambda \equiv 1$ on the compact set $\supp \rhou_f \cup \supp \rhou_g \cup \supp \phi$, the extended matrices having positive definite real parts. The smoothness of these matrices implies by the mean value theorem and \eqref{boundSupp} that 
\begin{align*}
 &\Modul{\Par{\det({H_\Phi}_+ + \overline{H_\Psi}_-)}^{-\ud}-\Croch{\det\Par{H_\Phi(s,\sigma) + \overline{H_\Psi}(s,\sigma)}}^{-\ud}} \\
 \lesssim& \sqrt \ep |(r,\delta)| \text{ on }\supp (\phi f_+ g^*_-). 
\end{align*}
By symplecticity and symmetry of $Q({H_\Phi}_+ , \overline{H_\Psi}_-)$, its inverse is $- J Q({H_\Phi}_+ , \overline{H_\Psi}_-) J$.
Thus the quantity 
$$	\Modul{e^{\GaE{\delta,-r}{Q({H_\Phi}_+ , \overline{H_\Psi}_-)^{-1}}}-e^{\GaE{r,\delta}{Q\Par{H_\Phi(s,\sigma) , \overline{H_\Psi}(s,\sigma)}}}}$$
 is dominated by 
\begin{equation*}\begin{split}
&\Big|(r,\delta) \cdot \big[Q({H_\Phi}_+ , \overline{H_\Psi}_-) -Q\Par{H_\Phi(s,\sigma) , \overline{H_\Psi}(s,\sigma)}\big] (r,\delta)\Big|\\
\x& \underset{u\in[0,1]}{\sup}  \Big|e^{-u(r,\delta) \cdot Q({H_\Phi}_+ , \overline{H_\Psi}_-)  (r,\delta)-(1-u)(r,\delta) \cdot Q\Par{H_\Phi(s,\sigma) , \overline{H_\Psi}(s,\sigma)}  (r,\delta)}\Big|.
\end{split}\end{equation*}
The positivity of $\Re Q({H_\Phi}_+ , \overline{H_\Psi}_-)$ and $\Re Q\Par{H_\Phi(s,\sigma) , \overline{H_\Psi}(s,\sigma)}$ and the mean value theorem for the matrix function $Q\Par{\lambda H_\Phi +(1-\lambda) Id, \lambda \overline{H_\Psi} +(1-\lambda) Id}$ give by \eqref{boundSupp}
\begin{equation*}\begin{split}
	\Big|e^{\GaE{\delta,-r}{Q({H_\Phi}_+ , \overline{H_\Psi}_-)^{-1}}}-&e^{\GaE{r,\delta}{Q\Par{H_\Phi(s,\sigma) , \overline{H_\Psi}(s,\sigma)}}}\Big| \lesssim  \sqrt \ep |(r,\delta)|^3 e^{-C (r,\delta)^2}
\end{split}\end{equation*}
for $(s,\sigma) \in \supp \phi$, $(s+ \sqrt \ep r , \sigma + \sqrt \ep \delta ) \in \supp \rhou_f$ and $(s - \sqrt \ep r , \sigma - \sqrt \ep \delta ) \in \supp \rhou_g$. It follows that 
\begin{equation*}\begin{split}
\Big|&<w_\ep\Par{I_\ep(f,\Phi),I_\ep(g,\Psi)},\phi> \\
& - c_n^2 2^{\frac{5n}{2}}  \pi^{\nd}  \int_{\IR^{4n}}  \phi(s,\sigma) (\det[H_\Phi + \overline{H_\Psi}])^{-\ud}(s,\sigma) f_+ g^*_- e^{i{r_\Phi}_+/\ep-i{r_\Psi}_-/\ep -2i\sigma \cdot r/\sqrt \epsilon}\\
&\,\, \quad \qquad \qquad \qquad e^{ \GaE{r,\delta}{Q(H_\Phi, \overline{H_\Psi})(s,\sigma)}}  dr d\delta ds d\sigma\Big| \lesssim \sqrt \ep    \nord{f} \nord{g}. 
\end{split}\end{equation*}
\end{proof}

\begin{proof}[Proof of Lemma \ref{TFGG}]
    The matrix $M+N$ has a positive definite real part and is thus non-singular. By elementary calculus we have
\begin{equation*}\begin{split}
&(x-a)\cdot{M} (x-a) + (x-b)\cdot {N} (x-b) \\
=&(b-a)\cdot {M(M+N)^{-1}N}(b-a) \\
&+\Par{x-(M+N)^{-1}(Ma+Nb)}\cdot{(M+N)}\Par{x-(M+N)^{-1}(Ma+Nb)}. 
\end{split}\end{equation*}
Using the value of the Fourier transform of a Gaussian function (see Theorem 7.6.1 of \cite{HormanderPDO1}), it follows that
\begin{equation*}\begin{split}
&\Fo_{x}{\Par{e^{\GasE{x-a}{M}{2}} e^{\GasE{x-b}{N}{2}} }}(\xi)\\
=&(2\pi)^{\frac{d}{2}} (\det [M+N])^{-\ud} e^{\GasE{b-a}{M(M+N)^{-1}N}{2}} \\
& e^{-i\xi\cdot(M+N)^{-1}(Ma+Nb) - \xi \cdot (M+N)^{-1} \xi /2}.
\end{split}\end{equation*}
Writing $M=1/2(M+N)+1/2(M-N)$ and $N=1/2(M+N)-1/2(M-N)$, we get the expression with the matrix $Q(M,N)$ and the relation
\begin{equation*}
Q(M,N) A(M,N)=B(M,N).
\end{equation*}
One can easily show that 
$$B(M,N)^T J B(M,N) = \Par{\begin{array}{cc} 0 & i(M+N) \\ -i(M+N) & 0 \end{array}} = A(M,N)^T J A(M,N),$$
from which follows the symplecticity of $Q(M,N)$.
Then write 
\begin{equation*}\begin{split}
& Q(M,N)+\overline{Q}(M,N) \\
=&{A}(M,N)^{*-1}\Par{{A}(M,N)^{*}B(M,N)+{B}(M,N)^{*}A(M,N)}A(M,N)^{-1}
\end{split}\end{equation*}
to obtain the value of $\Re Q(M,N)$.
\end{proof}

From now on, we drop the index $\ep$ in the notation of $v^\pm_{t,\ep}$, $v^\pm_{x,\ep}$, $f_{t,\ep}^k$ etc. for simplicity. We fix $t \in [0,T]$ and apply Lemma \ref{robin} with $F = \Bic$ on the sequences $(f_t^k),(f_t^l)$ (respectively $(f_x^k),(f_x^l)$) and the phase functions $\Phi_k,\Phi_l$ for the Wigner transforms associated to $\Par{v^+_t(t,.)}$ (respectively $\Par{v^+_x(t,.)}$). To evaluate the cross Wigner transforms between $\Par{v^+_t(t,.)}$ and $\Par{v^-_t(-t,.)}$ (respectively $\Par{v^+_x(t,.)}$ and $\Par{v^-_x(-t,.)}$) , we use this Lemma on the sequences $(f_t^k),(g_t^l)$ (respectively $(f_x^k),(g_x^l)$).

\subsection{Wigner measures for superposed Gaussian beams} \label{GBSwigner}

We shall prove that the cross Wigner transforms 

$$w_\ep\Par{v^+_t(t,.),v^-_t(-t,.)}, \ w_\ep\Par{v^+_x(t,.),v^-_x(-t,.)} $$
and
$$w_\ep\Par{I_\ep(f_{t,x}^k,\Phi_k),I_\ep(f_{t,x}^l,\Phi_l)},  \ w_\ep\Par{I_\ep(g_{t,x}^k,\Phi_k),I_\ep(g_{t,x}^l,\Phi_l)}$$
 with $k \ne l$ do not contribute to the microlocal energy density limit $\nrjFun{u^{appr}_{\ep,\indx}(t,.)}$ in $\Cob$. We compute $\Theta_\ep(\Phi_k,\Phi_k)$ and $A(\Phi_k,\Phi_k)$ and analyze the transported FBI transforms at points $(s\pm \sqrt \ep r,\sigma\pm \sqrt \ep \delta)$, which will complete the study of the Wigner measures for superposed Gaussian beams. 

Firstly, we note that $\nordye{(1 - \rhou \ox \phiu)\Pi_k p_{\ep,k}}=O(\ep^\infty)$ for $k=0,\pm 1$. Indeed, $\phiu \equiv 1$ on $\Rcomp$ so one gets from \eqref{CIsmooth2} that $T_\ep u^I_{\ep,\indx},T_\ep v^I_{\ep,\indx}$ have infinitely small contributions in $L^2(\IR^n \x \supp(1-\phiu))$. 

On the other hand,  $\dist\Par{\supp (1-\rhou),\supp u^I_{\ep,\indx} \cup \supp v^I_{\ep,\indx}} > C$. Then, 
Lemma \ref{FBIoutBis} implies that $T_\ep u^I_{\ep,\indx},T_\ep v^I_{\ep,\indx}$ have infinitely small contributions in $L^2(\supp(1-\rhou) \x \IR^n)$. Therefore
\begin{equation*}\begin{split}
&w_\ep\Par{I_\ep(f_t^k,\Phi_k),I_\ep(f_t^l,\Phi_l)} \\
\approx&	   A(\Phi_k,\Phi_l) \int_{\IR^{2n}}  \Par{c |\sigma|\Pi_k \trak{p_{\ep,k}}}_+   \Par{c |\sigma| \Pi_l \tral{\bar p_{\ep,l}}}_- e^{i \Theta_\ep(\Phi_k,\Phi_l)} dr d\delta \text{ in } \Cob,
\end{split}\end{equation*}
and a similar relation holds true for $w_\ep\Par{I_\ep(f_x^k,\Phi_k),I_\ep(f_x^l,\Phi_l)}$.
\\
We start by approaching $\Par{c(s)|\sigma|}_+ \Par{c(s)|\sigma|}_-$ by $c(s)^2|\sigma|^2$ in the previous integral 
\begin{equation}
\label{dtStart}
\begin{split}
&w_\ep\Par{I_\ep(f_t^k,\Phi_k),I_\ep(f_t^l,\Phi_l)}	  \\
\approx &A(\Phi_k,\Phi_l)c(s)^2 |\sigma |^2 \int_{\IR^{2n}}   \Par{\Pi_k \trak{p_{\ep,k}}}_+  \Par{\Pi_l 
\tral{\bar p_{\ep,l}}}_-  e^{i\Theta_\ep(\Phi_k,\Phi_l)} dr d\delta \text{ in } \Cob   , 
\end{split}
\end{equation}
and $\sigma_+ \sigma_-^*$ by $\sigma \sigma^*$ in the integral giving $w_\ep\Par{I_\ep(f_x^k,\Phi_k),I_\ep(f_x^l,\Phi_l)}$
\begin{equation}
\label{dxStart}
\begin{split}
&w_\ep\Par{I_\ep(f_x^k,\Phi_k),I_\ep(f_x^l,\Phi_l)}	 \\
\approx &A(\Phi_k,\Phi_l) \sigma \sigma^* \int_{\IR^{2n}}  \Par{\Pi_k  \trak{p_{\ep,k}}}_+ 
\Par{\Pi_l \tral{\bar p_{\ep,l}}}_- e^{i\Theta_\ep(\Phi_k,\Phi_l)} dr d\delta \text{ in } \Cob.
\end{split}
 \end{equation}
Indeed, these approximations are proved with the help of the following Lemma 
\begin{lemma}
\label{approxSmooth}
Let $(f_\ep),(g_\ep)$ and $\Phi,\Psi$ satisfy the hypotheses of Lemma \ref{robin}. 
If $\alp$ and $\beta$ are in $\Co^1(F,\IC)$ then 
\begin{equation*}\begin{split}
w_\ep\Par{I_\ep(\alp f_\ep,\Phi),I_\ep(\beta g_\ep,\Psi)} \approx \alp \bar \beta w_\ep\Par{I_\ep( f_\ep,\Phi),I_\ep( g_\ep,\Psi)} \text{ in }F.
\end{split}\end{equation*}
\end{lemma}

\begin{proof}
The proof relies on the use of Taylor's formula on $\rhou_f \alp$ and $\rhou_g \bar \beta$, where $\rhou_f$ and $\rhou_g$ are the cut-offs used in the proof of Lemma  \ref{robin} (supported in $F$ and equal to $1$ on $\supp f_\ep$ and $\supp g_\ep$ respectively).
\end{proof}

It follows by using \eqref{dtStart} and \eqref{dxStart} that
\begin{equation*}\begin{split}
c^2 \Tr \, w_\ep\Par{I_\ep(f_x^k,\Phi_k),I_\ep(f_x^l,\Phi_l)} &\approx w_\ep\Par{I_\ep(f_t^k,\Phi_k),I_\ep(f_t^l,\Phi_l)} \text{ in } \Cob,
\end{split}\end{equation*}
which leads to
\begin{equation*}
 w_\ep\Par{v^+_t(t,.),v^+_t (t,.)} \approx  c^2 \Tr\, w_\ep\Par{ v^+_x (t,.), v^+_x (t,.)} \text{ in } \Cob.
\end{equation*}
Similarly
\begin{align}
\label{wigt&x}
  w_\ep\Par{v^-_t(-t,.),v^-_t (-t,.)} &\approx  c^2 \Tr\, w_\ep\Par{ v^-_x (-t,.), v^-_x (-t,.)} \text{ in } \Cob, \nonumber \\
 \text{and } w_\ep\Par{v^+_t(t,.),v^-_t (-t,.)} &\approx  c^2 \Tr\, w_\ep\Par{ v^+_x(t,.) ,  v^-_x(-t,.)} \text{ in } \Cob. 
\end{align}
The approximations linking the derivatives of $u^{appr}_{\ep,\indx}$ to $v^\pm_{t,x}$ given in Lemma \ref{dtudxu} and equation \eqref{WigSmoo} lead to
\begin{equation*}\begin{split}
&4 \nrjFun{u^{appr}_{\ep,\indx}(t,.)} \\
\approx & w_\ep[ v_t^+(t,.)]+ c^2\Tr\, w_\ep[  v_x^+(t,.)] +w_\ep[ v_t^-(-t,.)]+ c^2 \Tr\, w_\ep[  v_x^-(-t,.)] \\
&  - w_\ep\Par{ v_t^+(t,.) , v_t^-(-t,.)} + c^2 \Tr\, w_\ep\Par{ v_x^+(t,.) ,  v_x^-(-t,.)}  \\
&  - w_\ep\Par{ v_t^-(-t,.), v_t^+(t,.)} + c^2 \Tr\, w_\ep\Par{ v_x^-(-t,.),  v_x^+(t,.)} \text{ in } \IR^{2n}
\end{split}\end{equation*}
by using the standard estimate (see Proposition 1.1 in \cite{GeMaMaPo})
\begin{equation}
\label{wigDif}
 |<w_\ep (a_\ep,b_\ep),\phi>| \lesssim \nordvois{a_\ep}{\IR^n} \nordvois{b_\ep}{\IR^n}, 
\end{equation}
for sequences $(a_\ep),(b_\ep)$ in $L^2(\IR^n)$ and $\phi \in \Co_0^\infty(\IR^{2n},\IR)$.
The cross terms between $v^+_{t,x}$ and $v^-_{t,x}$ cancel in $\Cob$ by using \eqref{wigt&x}, leading to
\begin{equation}
\label{WigAsymp0}
\nrjFun{u^{appr}_{\ep,\indx}(t,.)} \approx \ud w_\ep[ v_t^+(t,.)] + \ud w_\ep[ v_t^-(-t,.)] \text{ in } \Cob.
\end{equation}
Thus, we are left with the computation of the Wigner measure associated to $(v_t^+)$, computations being similar for $(v_t^-)$.
One has
\begin{equation}
\begin{split}
\label{WigdtVP}
&w_\ep[v_t^+]\\
 \approx& \somme{k,l=0,1}{} c(s)^2	|\sigma|^2 A(\Phi_k,\Phi_l) \int_{\IR^{2n}}  \Par{\Pi_k  \, \trak{p_{\ep,k}}}_+   
 \Par{\Pi_l  \,\tral{\bar p_{\ep,l}}}_- e^{i\Theta_\ep(\Phi_k,\Phi_l)} dr d\delta \text{ in } \Cob.
\end{split} 
\end{equation}
Moreover the inverse of the reflected/incident flow in $\Cob$ is a reflected/incident flow 
\begin{equation*}
\{\varphi_k^{t}\}^{-1} = \varphi_{-k}^{-t}, \, k=0,1.
\end{equation*}
Thus, for $(s,\sigma) \in \Cob$, at most one of the points $x_{-k}^{-t}(s,\sigma)$ and $x_{-l}^{-t}(s,\sigma)$ is in $\Omega$.
Consequently, the contribution of cross terms between different Gaussian beams in \eqref{WigdtVP} vanishes in $\Cob$,
and we need to compute only the limits when $\ep$ goes to zero of the following two distributions: 
\begin{equation}
\label{mukDef}
\mu_{\ep,k}^t = c(s)^2 |\sigma|^2 w_\ep[I_\ep(\Pi_k \trak{p_{\ep,k}},\Phi_k)],   \, k = 0,1.
\end{equation}
Remember that
$\trak{p_{\ep,k}} = \trak{a_k} \ep^{-1}\trak{T_\ep u_{\ep,\indx}^I} + \trak{a_k'}\trak{T_\ep v_{\ep,\indx}^I}$,
so $\mu_{\ep,k}^t$ may be written as 
\begin{equation}
\label{distvP}
\begin{split}
&\mu_{\ep,k}^t 
\\=& c^2(s) |\sigma|^2  w_\ep\Croch{I_\ep(\Pi_k \trak{a_{k}}\ep^{-1} \trak{T_\ep u_{\ep,\indx}^I},\Phi_k)} \\
&+ w_\ep\Croch{I_\ep(\Pi_k \trak{a_{k}} \trak{T_\ep v_{\ep,\indx}^I},\Phi_k)}  \\
&- i c(s) |\sigma| w_\ep\Par{I_\ep(\Pi_k \trak{a_{k}} \ep^{-1} \trak{T_\ep u_{\ep,\indx}^I},\Phi_k),I_\ep(\Pi_k \trak{a_{k}} \trak{T_\ep v_{\ep,\indx}^I},\Phi_k)} \\
&+ i c(s) |\sigma|  w_\ep\Par{I_\ep(\Pi_k \trak{a_{k}} \trak{T_\ep v_{\ep,\indx}^I},\Phi_k),I_\ep(\Pi_k \trak{a_{k}} \ep^{-1} \trak{T_\ep u_{\ep,\indx}^I},\Phi_k)}.
\end{split}
\end{equation}
In the remainder of this Section we prove the following Proposition, compute $\mu_{\ep,k}^t$ and the limit when $\ep\rightarrow 0$ of the microlocal energy density of $u^{appr}_{\ep,\indx}$.
\begin{proposition}
\label{wignerGBS}
Let $(\kappa_\ep),(\tau_\ep)$ be uniformly bounded sequences in $L^2(\IR^n)$. Then
$$w_\ep\Par{I_\ep(\Pi_k \trak{a_{k}} \trak{T_\ep \kappa_\ep},\Phi_k),I_\ep(\Pi_k \trak{a_{k}} \trak{T_\ep \tau_\ep},\Phi_k)} \approx \Pi_k^2  w_\ep(\kappa_\ep,\tau_\ep)\ron \Par{\varphi_k^t}^{-1} \text{ in } \Cob.$$
Above $\varphi_k^t$ is extended outside $\Bic$ as the identity.
\end{proposition}

\begin{proof}
We simplify the integral 
$$A(\Phi_k,\Phi_k) \int_{\IR^{2n}}  \Par{\Pi_k  \, \trak{T_\ep\kappa_\ep}}_+   
 \Par{\Pi_k  \,\trak{\overline{ T_\ep\tau_\ep}}}_- e^{i\Theta_\ep(\Phi_k,\Phi_k)} dr d\delta $$
obtained when applying Lemma \ref{robin} in $\Cob$ by firstly computing the phase $\Theta_\ep$ and the amplitude $A$ and then analyzing the transported FBI transforms.
\textbf{Computation of $\Theta_\ep(\Phi_k,\Phi_k)$ and $A(\Phi_k,\Phi_k)$.} 
We consider $(s,\sigma) \in \Cob$ and start from
\begin{equation*}
\Theta_\ep(\Phi_k,\Phi_k)(s,\sigma,r,\delta) = - 2 \sigma \cdot r/\sqrt \ep +i (r,\delta)  \cdot  Q\Par{\trak{\Lambda_k}(t,s,\sigma),\trak{\bar\Lambda_k}(t,s,\sigma)} (r,\delta).	
\end{equation*}
The particular form of $\Lambda_k(t) = -i V_k^t (U_k^t)^{-1}$, see Lemma \ref{FirstGB}, induces a similar form for the matrix $Q\Par{\trak{\Lambda_k}(t),\trak{\bar\Lambda_k}(t)}$
\begin{equation*}
Q\Par{\trak{\Lambda_k}(t),\trak{\bar\Lambda_k}(t)}Y_k^t = -iZ_k^t,
\end{equation*}
where $Y_k^t$ and $Z_k^t$ are the $2n \x 2n$ matrices
\begin{equation*}
Y_k^t=\left(\begin{array}{cc}
\trak{\bar{U}_k^t}  &  \trak{U_k^t}\\
\trak{\bar{V}_k^t}  &  \trak{V_k^t}
\end{array}\right) \text{ and }
Z_k^t=\left(\begin{array}{cc}
-\trak{\bar V_k^t}  &  \trak{V_k^t}\\
\trak{\bar U_k^t}  &  -\trak{U_k^t}
\end{array}\right).
\end{equation*} 
Replacing $U_k^t$ and $V_k^t$ by their definitions links $Y_k^t$ and $Z_k^t$ to the Jacobian matrix $F_{k}^{t}$
\begin{equation*}
{Y_k^t}=-i \trak{F_k^t} J\left(\begin{array}{cc}
-Id & Id  \\
iId & iId \\
\end{array}
\right) 
\text{ and }
{Z_k^t}=J \trak{F_k^t}\left(\begin{array}{cc}
-Id & Id  \\
iId & iId \\
\end{array}
\right), 
\end{equation*}
so that
\begin{equation*}
Q\Par{\trak{\Lambda_k}(t),\trak{\bar \Lambda_k}(t)}= - J \trak{F_k^t} J\Par{\trak{F_k^t}}^{-1}.
\end{equation*}
As $\varphi_k^t\ \ron\ \varphi_{-k}^{-t} = Id$, one has
\begin{equation*}
\trak{F_k^t} F_{-k}^{-t}=Id.
\end{equation*}
Combining this relation with the symplecticity of $F_k^t$, one gets the following relation for the matrix $Q\Par{\trak{\Lambda_k}(t),\trak{\bar\Lambda_k}(t)}$
\begin{equation*}
Q\Par{\trak{\Lambda_k}(t),\trak{\bar\Lambda_k}(t)} = (F_{-k}^{-t})^T F_{-k}^{-t}.
\end{equation*}
Therefore
\begin{equation*}
\Theta_\ep(\Phi_k,\Phi_k)(s,\sigma,r,\delta) = - 2\sigma  \cdot  r/\sqrt \ep +i \Par{F_{-k}^{-t}(s,\sigma) (r,\delta)}^2.	
\end{equation*}
Moving to the amplitude $ A(\Phi_k,\Phi_k) = c_n^2 2^{\frac{5n}{2}}  \pi^{\nd} \Par{\det(\trak{\Lambda_k}+\trak{\bar\Lambda_k})}^{-\ud}$, one gets
by using \eqref{propos} and \eqref{propk}
$$ \Lambda_k(t) +\bar \Lambda_k(t) = 2 \Par{(\bar U_k^t)^{-1}}^T (U_k^t)^{-1}.$$
Hence
\begin{equation*}
A(\Phi_k,\Phi_k)= c_n^2 2^{2n} \pi^{\nd}  \left|\det \trak{U_k^t} \right|.
\end{equation*}
Plugging the form of the incident and reflected amplitudes in Lemma \ref{FirstGB} and using the $\Co^1$ smoothness of $a_k^{(')}$ on $\Bic$ yields by Lemmas \ref{robin} and \ref{approxSmooth} 
\begin{equation*}\begin{split}
&w_\ep\Par{I_\ep(\Pi_k \trak{a_{k}} \trak{T_\ep \kappa_\ep},\Phi_k),I_\ep(\Pi_k \trak{a_{k}} \trak{T_\ep \tau_\ep},\Phi_k)} \\
\approx& c_n^2 2^{2n} \pi^{\nd} \int_{\IR^{2n}} \Par{\Pi_k \trak{T_\ep \kappa_\ep}}_+ \, \Par{\Pi_k \trak{\overline{T_\ep \tau_\ep}}}_- e^{-i 2\sigma  \cdot  r/\sqrt \ep - \Par{F_{-k}^{-t} (r,\delta)}^2} dr d\delta \\
=:&J_{\ep,k}^t(\kappa_\ep,\tau_\ep).
\end{split}\end{equation*}
\textbf{Analysis of the transported FBI transforms.}
It remains to analyze the most difficult terms in the amplitude, which involve transported FBI transforms 
\begin{equation*}\begin{split} 
\Par{\Pi_k  \trak{T_\ep \kappa_\ep}}_+ &= \Par{\Pi_k \, T_\ep \kappa_\ep \ron \varphi^{-t}_{-k}}( s+ \sqrt \ep r , \sigma + \sqrt \ep \delta),\\
\text{and }\Par{\Pi_k  \trak{\overline{T_\ep \tau_\ep}}}_- &= \Par{\Pi_k \, \overline{T_\ep \tau_\ep} \ron \varphi^{-t}_{-k}}( s- \sqrt \ep r , \sigma - \sqrt \ep \delta).
\end{split}\end{equation*}
Let $\phi$ be a test function in $\Co_0^\infty(\Cob,\IR)$ and $\vartheta_{-k}^{-t}$ a map of $\Co_0^\infty(\IR^{2n},\IR^{2n})$ that coincides with $\varphi_{-k}^{-t}$ on $K_{z,\theta}^k(t) \cup \supp \phi$ (see Theorem 1.4.1 of \cite{HormanderPDO1}). We use Taylor's formula for this map to get for  $(s\pm\sqrt \ep r,\sigma \pm \sqrt \ep \delta) \in K_{z,\theta}^k(t)$ and $(s,\sigma)\in \supp \phi$ 
\begin{align*}
  \Par{x_{-k}^{-t}}_\pm &= x_{-k}^{-t}\pm \sqrt \ep \, D_y x_{-k}^{-t}\, r \pm \sqrt \ep \,D_\eta x_{-k}^{-t}\, \delta + \ep r_\ep^{x \pm},\\
\Par{\xi_{-k}^{-t}}_\pm &= \xi_{-k}^{-t}\pm \sqrt \ep \,D_y \xi_{-k}^{-t}\, r \pm \sqrt \ep \,D_\eta \xi_{-k}^{-t}\, \delta + \ep r_\ep^{\xi \pm},
\end{align*}
with
\begin{align*}
r_\ep^{x \pm}  (s,\sigma,r , \delta ) &= \somme{|\alp|=2}{} \frac{2}{\alp!} (r, \delta)^\alp
\int_{0}^1 (1-u)\dpp^\alp_{y} \vartheta_{-k}^{-t} \Par{\Par{s, \sigma} \pm u\sqrt\ep \Par{ r, \delta}} du, \\
r_\ep^{\xi \pm}(s,\sigma,r , \delta ) &= \somme{|\alp|=2}{} \frac{2}{\alp!} (r, \delta)^\alp
\int_{0}^1 (1-u)\dpp^\alp_{\eta} \vartheta_{-k}^{-t} \Par{\Par{s, \sigma} \pm u\sqrt\ep \Par{ r, \delta}} du. 
\end{align*}
The change of variables $(r',\delta') = F_{-k}^{-t}(s,\sigma)(r,\delta)$ in $J_{\ep,k}^t(\kappa_\ep,\tau_\ep)(s,\sigma)$ is thus appropriate.
Notice that for $(s,\sigma) \in \Cob$ one has the following relations \cite{LaSi}
\begin{equation*}
	D_y x_{-k}^{-u}(s,\sigma)^T \xi_{-k}^{-u}(s,\sigma) - \sigma =0 \text{ and } D_\eta x_{-k}^{-u}(s,\sigma)^T \xi_{-k}^{-u}(s,\sigma) = 0 \text{ for } u \in \IR.
\end{equation*}
In fact, one can show that the derivatives of the previous equations w.r.t. $u$ are zero. Besides, the equalities clearly hold true at $u=0$ for $k=0$, and at $u=T_k(s,\sigma)$ for $k=\pm 1$, as a consequence of \eqref{Uref}.
Hence, it follows that
\begin{equation*}
	\sigma \cdot r = \xi_{-k}^{-t}(s,\sigma) \cdot \Par{ D_y x_{-k}^{-t}(s,\sigma) r + D_\eta x_{-k}^{-t}(s,\sigma) \delta } = \xi_{-k}^{-t}(s,\sigma) \cdot r',
\end{equation*}
which leads in $\Cob$ to 
\begin{equation*}\begin{split}
&J_{\ep,k}^t(\kappa_\ep,\tau_\ep) \\
= & c_n^2 2^{2n} \pi^{\nd} \int_{\IR^{2n}} (\Pi_k)_+ T_\ep \kappa_\ep(x_{-k}^{-t}+ \sqrt  \ep r' + \ep  {r_\ep^{x +}}', \xi_{-k}^{-t}+\sqrt  \ep  \delta' + \ep  {r_\ep^{\xi +}}') \\
&\qquad \qquad \qquad (\Pi_k)_- \overline{T_\ep \tau_\ep} (x_{-k}^{-t}- \sqrt  \ep r' +  \ep {r_\ep^{x -}}', \xi_{-k}^{-t}-\sqrt  \ep  \delta' +  \ep {r_\ep^{\xi -}}') \\
&\qquad \qquad \qquad e^{- 2i\xi_{-k}^{-t} \cdot r' /\sqrt \ep -{r'}^2 -{\delta'}^2 } dr' d\delta', 
\end{split}\end{equation*}
where
$$ ({r_\ep^x}',{r_\ep^\xi}')(s,\sigma,r',\delta') = (r_\ep^x,r_\ep^\xi) (s,\sigma,r,\delta).$$ 
In order to use the change of variables $(s,\sigma) = \varphi_k^t(y,\eta)$ for $<J_{\ep,k}^t(\kappa_\ep,\tau_\ep),\phi>$,  we extend $\varphi_k^t$ outside $\Bic$ by the identity and still denote it $\varphi_k^t$, making $\varphi_k^t$ a one to one map from $\IR^{2n}$ to $\varphi_k^t(\IR^{2n})$.
 Then $\Pi_k\ron \varphi_k^t$ and $\phi\ron \varphi_k^t$ belong to $  \Co_0^\infty(\IR^{2n},\IR)$ and are supported in $\Bic$.
Expanding the FBI transforms gives
\begin{equation*}\begin{split}
&<J_{\ep,k}^t(\kappa_\ep,\tau_\ep),\phi> \\
=& c_n^4 2^{2n} \pi^{\nd} \ep^{-\tnd}\int_{\IR^{6n}}  \phi\ron \varphi_k^t(y,\eta) \kappa_\ep(z) \bar \tau_\ep(z') \\
&\qquad \qquad \qquad \qquad (\Pi_k \ron \varphi_k^t )(y+ \sqrt  \ep r' + \ep R_\ep^{x +},\eta+ \sqrt  \ep \delta' + \ep R_\ep^{\xi +})\\
&\qquad \qquad \qquad \qquad (\Pi_k \ron \varphi_k^t )(y- \sqrt  \ep r' + \ep R_\ep^{x -},\eta - \sqrt  \ep \delta' + \ep R_\ep^{\xi -}) \\ 
&\qquad \qquad \qquad \qquad  e^{  i \eta \cdot ( 2  \sqrt \ep r'+ \ep R_\ep^{x+} -\ep R_\ep^{x-} -z +z') /\ep +i \delta' \cdot (2y  -z -z'+ \ep R_\ep^{x+} +\ep R_\ep^{x-}) /\sqrt \ep } \\
&\qquad \qquad \qquad \qquad e^{i R_\ep^{\xi+} \cdot(y + \sqrt \ep r' +\ep R_\ep^{x+} -z ) - i R_\ep^{\xi-} \cdot(y - \sqrt \ep r'+ \ep R_\ep^{x-} -z' )}\\
&\qquad \qquad \qquad \qquad e^{-(y + \sqrt \ep r' +\ep R_\ep^{x+} -z )^2/(2 \ep) -(y - \sqrt \ep r'+ \ep R_\ep^{x-} -z' )^2/(2 \ep)} \\
&\qquad \qquad \qquad \qquad e^{- 2i\eta \cdot r' /\sqrt \ep -{r'}^2 -{\delta'}^2 } dr' d\delta' dz dz' dy d\eta,
\end{split}\end{equation*}
where
$$ (R_\ep^x, R_\ep^\xi)(y,\eta,r',\delta') = ({r_\ep^x}',{r_\ep^\xi}') (s,\sigma,r',\delta').$$ 
We perform the changes of variables
$$(x,u)=(\frac{z+z'}{2},\frac{z-z'}{\ep}) \text{ and } y'= (y-\frac{z+z'}{2})/\sqrt\ep $$
to obtain
\begin{equation*}\begin{split}
&<J_{\ep,k}^t(\kappa_\ep,\tau_\ep),\phi> \\
=& c_n^4 2^{2n} \pi^{\nd}  \int_{\IR^{6n}} \kappa_\ep(x+ \frac \ep 2 u) \bar \tau_\ep(x - \frac \ep 2 u) d_\ep e^{i\gamma_\ep - i \eta \cdot u} dr' d\delta' dx du dy' d\eta,
\end{split}\end{equation*}
where
\begin{equation*}\begin{split}
	&d_\ep(x,y',\eta,r',\delta') \\
	=&\phi\ron \varphi_k^t(x + \sqrt \ep y',\eta)  (\Pi_k\ron \varphi_k^t)(x + \sqrt \ep y'+ \sqrt  \ep r' + \ep {R_\ep^{x +}}', \eta+ \sqrt  \ep \delta' + \ep {R_\ep^{\xi +}}' ) \\
&	(\Pi_k\ron \varphi_k^t ) (x + \sqrt \ep y'- \sqrt  \ep r' + \ep {R_\ep^{x -}}', \eta - \sqrt  \ep \delta' + \ep {R_\ep^{\xi -}}'),
\end{split}\end{equation*}
\begin{equation*}\begin{split}
&\gamma_\ep(x,y',\eta,r',\delta',u) \\
=&   \eta \cdot ( {R_\ep^{x+}}' - {R_\ep^{x-}}' ) +  \delta' \cdot (2 y' + \sqrt \ep {R_\ep^{x+}}' +\sqrt \ep {R_\ep^{x-}}' )  \\
&+  \sqrt \ep {R_\ep^{\xi+}}' \cdot( y' + r' +\sqrt \ep {R_\ep^{x+}}' - \sqrt \ep \frac u 2) \\
&-  \sqrt \ep {R_\ep^{\xi-}}' \cdot(y' - r'+ \sqrt \ep {R_\ep^{x-}}' + \sqrt \ep \frac u 2)+i {r'}^2 +i{\delta'}^2\\
&+i ( y' +r'  +\sqrt \ep {R_\ep^{x+}}'  -\sqrt \ep u/2)^2 /2 +i (y'- r'  + \sqrt \ep {R_\ep^{x-}}' + \sqrt \ep u/2 )^2/2 , 
\end{split}\end{equation*}
and
$$ ({R_\ep^{x}}',{R_\ep^{\xi}}')(x,y',\eta,r',\delta') = (R_\ep^{x},R_\ep^{\xi})(x+\sqrt \ep y',\eta,r',\delta'). $$
Notice that $d_\ep(x,y',\eta,r',\delta')$ converges when $\ep \rightarrow 0$ to
$$d_0(x, \eta) = \phi\ron \varphi_k^t(x,\eta)  (\Pi_k\ron \varphi_k^t)^2(x,\eta) .$$
On the other hand, since $\ep r_x^\pm $ are the remainder terms in the Taylor expansions of 
\\
$x_{-k}^{-t}(s\pm \sqrt \ep r,\sigma \pm \sqrt \ep \delta)$ at order $2$, $r_\ep^{x+} -r_\ep^{x-}$ is of order $\sqrt \ep$ and so is ${R_\ep^{x+}}' - {R_\ep^{x-}}'$, leading to
$$\gamma_\ep(x,y',\eta,r',\delta',u) \underset{\tendz{\ep}}{\rightarrow}  \gamma_0(y',r',\delta') = 2 \delta' \cdot   y' +i {y'}^2 + 2 i {r'}^2+i{\delta'}^2.$$
One has
\begin{equation}
\label{beforeTCD}
\begin{split}
\Big|&<J_{\ep,k}^t(\kappa_\ep,\tau_\ep),\phi> \\
&- c_n^4 2^{2n} \pi^{\nd} \int_{\IR^{6n}} \kappa_\ep(x+ \frac \ep 2 u) \bar \tau_\ep(x - \frac \ep 2 u)  d_0 e^{i\gamma_0} e^{-i \eta \cdot u}  dr' d\delta'  du dy' dx d\eta \Big|   \\
\lesssim &\int_{\IR^{4n}} \Croch{\int_{\IR^n}|\kappa_\ep|(x+\frac{\ep}{2} u) |\tau_\ep|(x-\frac{\ep}{2} u)  dx }  \\
& \,\,\quad \quad  \underset{x}{sup}\Modul{\Fo_{\eta}{\Par{d_\ep  e^{i\gamma_\ep}- d_0 e^{i\gamma_0}}}(x,y',u,r',\delta',u) } dr' d\delta' du dy'. 
\end{split}
\end{equation}
Cauchy-Schwartz inequality w.r.t. $dx$ insures that the bracket integral is less than $\nord{\kappa_\ep} \nord{\tau_\ep}$.
Let us examine the term
$$ \int_{\IR^{4n}} \underset{x}{sup}\big|\Fo_{\eta }{\Par{d_\ep  e^{i\gamma_\ep}- d_0  e^{i\gamma_0}}}(x,y',u,r',\delta',u)\big| dr' d\delta' du dy' .
$$
For fixed $y',r',\delta'$ the functions $d_\ep$ and $d_0$ are compactly supported w.r.t. $(x,\eta)$ so 
\begin{equation*}\begin{split}
		         & \underset{x}{\sup}\Modul{\Fo_{\eta}{\Par{d_\ep  e^{i\gamma_\ep}- d_0  e^{i\gamma_0}}}(x,y',u,r',\delta',u)}\\
		\lesssim & \underset{(x,\eta)}{\sup} \Modul{\Par{d_\ep  e^{i\gamma_\ep}- d_0  e^{i\gamma_0}}(x,y',\eta,r',\delta',u)}. 
\end{split}\end{equation*}
Note that $\Modul{d_\ep  e^{i\gamma_\ep}- d_0  e^{i\gamma_0}}$ is dominated by $\Modul{d_\ep - d_0} + \Modul{d_0} \Modul{ e^{i\gamma_\ep-i\gamma_0}- 1}$.
The convergence of $d_\ep $ when $\ep\rightarrow 0$ to its limit $d_0$ is uniform w.r.t. $(x,\eta)$ and so is the convergence of $\gamma_\ep$ to  $\gamma_0 $ on the support of $d_0$. Thus $d_\ep  e^{i\gamma_\ep}$ converges to $d_0  e^{i\gamma_0}$ uniformly w.r.t. $(x,\eta)$. It follows that
\begin{equation*}
		\underset{x}{\sup}\Modul{\Fo_{\eta}{\Par{d_\ep  e^{i\gamma_\ep}- d_0  e^{i\gamma_0}}}(x,y',u,r',\delta',u)} \underset{\tendz{\ep}}{\rightarrow}  0 \text{ for every } y',u,r',\delta'.
\end{equation*}
On the other hand, successive integrations by parts give
\begin{equation*}
\int_{\IR^n} d_\ep  e^{i\gamma_\ep} e^{- i \eta \cdot u} d\eta = (1+u^2)^{-n} \int_{\IR^n} L\Par{ d_\ep  e^{i\gamma_\ep}}e^{- i \eta \cdot u} d\eta,
\end{equation*}
with $L$ a differential operator w.r.t. $\eta$, of order $2n$. Thus, 
\begin{equation}
\label{boundEps}
\begin{split}
&	\underset{x}{\sup}\Modul{\Fo_\eta{\Par{d_\ep  e^{i\gamma_\ep}}}(x,y',u,r',\delta',u)} \\
\lesssim & (1+u^2)^{-n} 		\underset{(x,\eta)}{\sup} \underset{|\alp| \leq 2n}{\max}\Modul{\dpp_\eta^\alp \Par{d_\ep e^{i\gamma_\ep}}(x,y',\eta,r',\delta',u)},
\end{split}
\end{equation}
for every  $y',r',\delta',u$. The quantities $(x+\sqrt \ep y',\eta)$ and $\sqrt \ep (r',\delta')$ are bounded on the support of $d_\ep$, so ${R_\ep^{x\pm}}'$,  ${R_\ep^{\xi\pm}}'$ and their derivatives w.r.t. $\eta$ are dominated by $(r',\delta')^2$. Hence for a given multiindex $\alp$, there exists $C>0$ s.t. 
\begin{align*}
	\Modul{\dpp_{\eta}^\alp d_\ep} \leq& C,\\
	\Modul{\dpp_{\eta}^\alp \gamma_\ep} \leq& C |(r',\delta')| \big(|(r',\delta')|	+ |y' +r'  + \sqrt \ep {R_\ep^{x+}}'  -\sqrt \ep u/2|\\
	&\qquad \qquad\, +|y' -r' + \sqrt \ep {R_\ep^{x-}}' -  \sqrt \ep u/2|\big)  \text{ if } |\alp| \geq 1,
\end{align*}
for all $(x,y',\eta,r',\delta') \in \supp d_\ep$ and $u \in \IR^n$. 
Thus, there exists $C$, $C'>0$ s.t. 
\begin{equation*}\begin{split}
	\Modul{\dpp_{\eta}^\alp \Par{d_\ep  e^{i\gamma_\ep}} } &\leq C  e^{-C'( y' +r'  +\sqrt \ep {R_\ep^{x+}}'  -\sqrt \ep u/2)^2 -C' (y'- r'  + \sqrt \ep {R_\ep^{x-}}' + \sqrt \ep u/2 )^2  -C' {r'}^2 -C' {\delta'}^2 }\\
&	\leq C  e^{-C'( 2y' +\sqrt \ep {R_\ep^{x+}}'+\sqrt \ep {R_\ep^{x-}}')^2  -C' {r'}^2 -C' {\delta'}^2 },
\end{split}\end{equation*}
for all  $(x,y',\eta,r',\delta') \in \supp d_\ep$  and $ u \in \IR^n$.
On the support of $d_\ep$, $\sqrt \ep {R_\ep^{x\pm}}'$ are dominated by $|(r',\delta')|$, which implies for some $C_0>0$ that
\begin{equation*}
	( 2y' +\sqrt \ep {R_\ep^{x+}}'+\sqrt \ep {R_\ep^{x-}}')^2 \geq 4 {y'}^2 - C_0 |(r',\delta')| |y'|.
\end{equation*}
Hence, if $|y'| \geq C_0 |(r',\delta')|$, $e^{-C'( 2y' +\sqrt \ep {R_\ep^{x+}}'+\sqrt \ep {R_\ep^{x-}}')^2} \leq e^{- C''{y'}^2}$. Otherwise,\\
 $e^{-C' {r'}^2-C' {\delta'}^2} \leq e^{-C'' {y'}^2-C'' {r'}^2-C'' {\delta'}^2}$.
In all cases, there exists $C',C''>0$ s.t.
\begin{equation*}
	\Modul{\dpp_{\eta}^\alp \Par{d_\ep  e^{i\gamma_\ep}} } \leq C'  e^{-C''{y'}^2 -C'' {r'}^2 -C'' {\delta'}^2 },
\end{equation*}
for every $x,y',\eta,r',\delta',u $ and $\ep \in ]0,\ep_0]$ with some $\ep_0>0$.
Using this in \eqref{boundEps} leads to
\begin{equation*}\begin{split}
	&	\underset{x}{\sup}\Modul{\Fo_\eta {\Par{d_\ep  e^{i\gamma_\ep}}}(x,y',u,r',\delta',u)} \lesssim (1+u^2)^{-n}  e^{ -C {y'}^2 -C {r'}^2 -C {\delta'}^2 }, 
\end{split}\end{equation*}
and repeating the same arguments for $\underset{x}{\sup}\Modul{\Fo_{\eta }\Par{d_0  e^{i\gamma_0}}}$ gives 
\begin{equation*}\begin{split}
	&	\underset{x}{\sup}\Modul{\Fo_{\eta }{\Par{d_\ep  e^{i\gamma_\ep}- d_0  e^{i\gamma_0}}}(x,y',u,r',\delta',u)} \lesssim (1+u^2)^{-n}  e^{ -C {y'}^2 -C {r'}^2 -C {\delta'}^2 }, 
\end{split}\end{equation*}
for every $y',u,r',\delta' $ and $\ep \in ]0,\ep_0]$. By the dominated convergence theorem, one obtains
\begin{equation*}
	\int_{\IR^{4n}} \underset{x}{\sup}|\Fo_{\eta}{\Par{d_\ep  e^{i\gamma_\ep} - d_0  e^{i\gamma_0}}}(x,y',u,r',\delta',u) | dy' du dr' d\delta' \underset{\tendz{\ep}}{\rightarrow}  0.
\end{equation*}
From the inequality \eqref{beforeTCD} concerning the distribution $J_{\ep,k}^t(\kappa_\ep,\tau_\ep)$, one finally has by plugging the expressions of $d_0$ and $\gamma_0$
\begin{equation*}\begin{split}
<J_{\ep,k}^t(\kappa_\ep,\tau_\ep),\phi> = c_n^4 2^{2n} \pi^{\nd}  \int_{\IR^{6n}} &  \kappa_\ep(x+ \frac \ep 2 u) \bar \tau_\ep(x - \frac \ep 2 u) \\
& e^{2 i \delta' \cdot   y' - {y'}^2 - 2{r'}^2-{\delta'}^2} e^{-i \eta \cdot u}  dr' d\delta' dx du dy' d\eta + o(1). 
\end{split}\end{equation*}
Integration w.r.t. $r',\delta',y',\eta$ yields
\begin{equation*}\begin{split}
&<J_{\ep,k}^t(\kappa_\ep,\tau_\ep),\phi> \\
=& (2\pi)^{-n} \int_{\IR^{2n}}   \Fo_{\eta }{\Par{\Pi_k^2 \ron \varphi_{k}^{t} \phi \ron \varphi_{k}^{t}}}(x,u)  \kappa_\ep(x+\frac{\ep}{2} u) \bar\tau_\ep(x-\frac{\ep}{2} u) dx du + o(1). 
\end{split}\end{equation*}
The integral in the r.h.s. is exactly the Wigner transform of $(\kappa_\ep,\tau_\ep)$ tested on $\Pi_k^2 \ron \varphi_{k}^{t}\, \phi \ron \varphi_{k}^{t}$.
\end{proof}

We are now able to compute the measure $\mu_{\ep,k}^t$ given in \eqref{distvP} by using the previous Proposition and the Lemma \ref{modDWig}
\begin{equation*}
\mu_{\ep,k}^t \approx  \Pi_k^2 \Par{ w_\ep\Croch{v_{\ep,\indx}^I - ic |D| u_{\ep,\indx}^I}}\ron \Par{  \varphi^{t}_{k}  }^{-1} \text{ in } \Cob.
\end{equation*}
Recalling the relation between the Wigner measure and the FBI transform (see Proposition 1.4 of \cite{GeLe})
\begin{equation}
\label{FBIWig}
	\int_{\IR^{2n}} |T_\ep a_\ep|^2 \theta dy d\eta \underset{\ep \rightarrow 0}{\rightarrow} <w[a_\ep], \theta>,  
\end{equation}
for  $\theta \in \Co_0^\infty(\IR^{2n},\IR)$  and  $(a_\ep)$ a uniformly bounded sequence in $L^2(\IR^{n})$, it follows that 
$w_\ep\Croch{v_{\ep,\indx}^I - ic |D| u_{\ep,\indx}^I}  \approx 0  \text{ in } (K_y \x K_\eta)^c$ or equivalently
$$w_\ep\Croch{v_{\ep,\indx}^I - ic |D| u_{\ep,\indx}^I} \ron \Par{  \varphi^{t}_{k}  }^{-1}  \approx 0  \text{ in } (K_{z,\theta}^k(t))^c.$$
Since $\Pi_k  \equiv 1 $ on $K_{z,\theta}^k(t)$, one deduces
\begin{equation*}
\mu_{\ep,k}^t \approx   w_\ep\Croch{v_{\ep,\indx}^I - ic |D| u_{\ep,\indx}^I}\ron \Par{  \varphi^{t}_{k}  }^{-1} \text{ in } \Cob.
\end{equation*}
By summing over $k=0,1$  and letting $\ep \rightarrow 0$, we get
\begin{equation*}
w[ v_t^+ (t,.)] =   \somme{k=0,1}{} w\Croch{v_{\ep,\indx}^I - ic |D| u_{\ep,\indx}^I}\ron \Par{\varphi_k^{t}}^{-1} \text{ in } \Cob.
\end{equation*}
For $u \in[-T,T]$ and $(y,\eta) \in K_y\x \Par{\IR^n\prive{0}}$, the incident and reflected flows are related to the broken bicharacteristic flow associated to $-i \dt -c|D|$ as follows: 
\begin{equation*}
\label{bbf}
	\varphi_b^u (y,\eta) = \left\{\begin{array}{cc} 
	\varphi_{-1}^u (y,\eta) &\text{if } u  < T_{-1}(y,\eta),\\ 
	\varphi_{0}^u (y,\eta)  &\text{if } T_{-1}(y,\eta) < u  < T_{1}(y,\eta), \\
	\varphi_{1}^u (y,\eta)  &\text{if } u  > T_{1}(y,\eta).
	\end{array} \right.
\end{equation*}
We extend $\varphi_b^u $ at times of reflections arbitrary. We define $\varphi_b^u $ in $(\Omega \backslash K_y)\x(\IR^n\prive{0})$ by successively reflecting the rays at the boundary.As only one incident/reflected ray can be in the interior of the domain at a fixed time $t\in [-T,T]$
$$\phi \ron \varphi_b^t = \somme{k=0,1}{} \phi \ron \varphi_k^t  \text{ in }K_y \x \IR^n\prive{0}.$$
It follows that 
\begin{equation*}
\label{FinalWigP}
w[v_t^+(t,.)] =   w\Croch{v_{\ep,\indx}^I - ic |D| u_{\ep,\indx}^I}\ron \Par{\varphi^{t}_b}^{-1} \text{ in } \Cob.
\end{equation*}

The computations for $v_t^-$ are similar. One has just to replace the index $k = 1$ by $k=-1$ and $\trak{p_{\ep,k}}$ by $\trak{q_{\ep,k}}$
in \eqref{mukDef} and to repeat the same techniques.
If we denote $\Upsilon^{\pm}_{\ep,\indx} = v^{I}_{\ep,\indx} \pm ic|{D}| u^{I}_{\ep,\indx}$, then one gets 
\begin{equation*}
w[ v_t^-(-t,.)]  =  w\Croch{\Upsilon^{+}_{\ep,\indx}}\ron \Par{\varphi^{-t}_b}^{-1} \text{ in } \Cob.
\end{equation*}
Using these results in \eqref{WigAsymp0} as $\ep \rightarrow 0$ leads to
\begin{equation}
\label{WigApp}\begin{split}
\nrjFun{ u_{\ep,\indx}^{appr}(t,.)}= \ud  w\Croch{\Upsilon^{+}_{\ep,\indx}}\ron \Par{\varphi^{-t}_b}^{-1} + \ud w\Croch{\Upsilon^{-}_{\ep,\indx}}\ron \Par{\varphi^{t}_b}^{-1} \text{ in } \Cob. 
\end{split}\end{equation}

\subsection{Proof of the main Theorem} \label{truncation}

A consequence of the estimate \eqref{wigDif} is
\begin{equation}
\label{wigDifP}
|<w (a_\ep,b_\ep),\theta>| \lesssim \liSu \nordvois{a_\ep}{\Omega} \liSu\nordvois{b_\ep}{\Omega},
\end{equation}
for $(a_\ep),\,(b_\ep)$ uniformly bounded sequences in $L^2(\IR^n)$ and $\theta \in \Co_0^\infty(\Cob,\IR)$. 
Applying this estimate to the difference between the derivatives of the exact and approximate solutions of the IBVP \eqref{MPb:gp1}-\eqref{MPb:gp2} with initial conditions \eqref{CIcut}, one deduces the measures associated to $\Par{\udtuepInd}$ and $\Par{\udxuepInd}$ and gets by \eqref{WigApp} 
\begin{equation*}
\nrjFun{u_{\ep,\indx}(t,.)} = \ud  w\Croch{\Upsilon^{+}_{\ep,\indx}}\ron \Par{\varphi^{-t}_b}^{-1} + \ud \Croch{\Upsilon^{-}_{\ep,\indx}}\ron \Par{\varphi^{t}_b}^{-1} \text{ in } \Cob.
\end{equation*}

\begin{remark}
Gaussian beam summation of first order beams allows to compute the microlocal energy density of the solution of the IBVP \eqref{MPb:gp} as $\ep \rightarrow 0$, under the hypotheses \eqref{CIboundBis},\eqref{CIcompBis} and \eqref{CIsmoothBis} on initial conditions. Summation of higher order beams may imply asymptotic formulas for the Wigner transforms and thus for the energy density. Higher order terms in the expansion of the Wigner transform were studied for instance in \cite{FiMa} and \cite{Pulvirenti} for WKB initial data.
\end{remark} 

Let us now study the microlocal energy density for the problem \eqref{MPb:gp} when $\ep \rightarrow 0$, by making the data $(u_{\ep,\indx}^I,v_{\ep,\indx}^I)$ approach $(\uuepI,\uvepI)$. The contribution of the sets $\{\eta \in \IR^n, |\eta|\geq r_\infty/4\}$ and $\{\eta \in \IR^n, |\eta| \leq 4 r_0 \}$ where $\gamma_{\indx} \not\equiv 1$ (remember the definition of $\gamma_{\indx}$ in \eqref{cutet}) to  $T_\ep\uuepI,T_\ep\uvepI$ is controlled asymptotically by the assumptions \eqref{oscil} and \eqref{wigZero}.\\
Set ${\Upsilon^{\pm}_\ep = \uvepI \pm i c|{D}| \uuepI}$ and denote $\phi^t=\phi o \varphi^{t}_b$. Then $\phi^t \in \Co_0^\infty(\IR^{2n},\IR)$ and one has
\begin{equation}
\label{wigDiff}
\begin{split}
&\Modul{<\nrjFun{u_\ep(t,.)},\phi>-\ud <w\Croch{\Upsilon_\ep^{+}},\phi^{-t}>-\ud <w\Croch{\Upsilon_\ep^{-}},\phi^{t}>}   \\
\leq &\ud\Modul{<w\Croch{\Upsilon_{\ep,\indx}^{+}}-w\Croch{\Upsilon_\ep^{+}},\phi^{-t}>}+
\ud\Modul{<w\Croch{\Upsilon_{\ep,\indx}^{-}}-w\Croch{\Upsilon_\ep^{-}},\phi^{t}>}\\
&+ \Modul{<w\Croch{\udtuep(t,.)}-w\Croch{\udtuepInd(t,.)},\phi>}\\
&+ \sum_{j=1}^n \Modul{<w\Croch{c \udxjuep(t,.)}-w\Croch{c \udxjuepInd(t,.)},\phi>}  \\
 &+\Modul{<\nrjFun{u_{\ep,\indx}(t,.)},\phi>-\ud <w\Croch{\Upsilon_{\ep,\indx}^{+}},\phi^{-t}>-\ud <w\Croch{\Upsilon_{\ep,\indx}^{-}},\phi^{t}>}.
\end{split}
\end{equation}
We use \eqref{wigDif} to get
\begin{equation*}\begin{split}
&|<w\Croch{\Upsilon_{\ep,\indx}^{+}}-w\Croch{\Upsilon_\ep^{+}},\phi^{-t}>| \\
\lesssim& \liSu\nordvois{\Upsilon_{\ep,\indx}^{+}-\Upsilon_\ep^{+}}{\IR^n} \liSu\Par{\nordvois{\Upsilon_{\ep,\indx}^{+}}{\IR^n}+\nordvois{\Upsilon_\ep^{+}}{\IR^n}}  \\
\lesssim& \liSu \nordom{v_\ep^I- v_{\ep,\indx}^{I}} + \liSu\norhom{u_\ep^I- u_{\ep,\indx}^{I}} .
\end{split}\end{equation*}
Similarly, by \eqref{wigDifP}
\begin{equation*}\begin{split}
&\Big|<w\Croch{\udtuep(t,.)}-w\Croch{\udtuepInd(t,.)},\phi>\Big|\\
\lesssim & \liSu \nordvois{\dt u_\ep(t,.)-\dt u_{\ep,\indx}(t,.)}{\Omega}\\
&\Par{\liSu \nordvois{\dt u_\ep(t,.)}{\Omega}+ \liSu \nordvois{\dt u_{\ep,\indx}(t,.)}{\Omega}},
\end{split}\end{equation*}
and for $j=1,\dots, n$
\begin{equation*}\begin{split}
&\Big|<w\Croch{\udxjuep(t,.)}-w\Croch{\udxjuepInd(t,.)},\phi>\Big|\\
\lesssim &\liSu \nordvois{\dpp_{x_j} u_\ep(t,.)-\dpp_{x_j} u_{\ep,\indx}(t,.)}{\Omega}\\
&\Par{\liSu \nordvois{\dpp_{x_j} u_\ep(t,.)}{\Omega}+ \liSu \nordvois{\dpp_{x_j} u_{\ep,\indx}(t,.)}{\Omega}}.
\end{split}\end{equation*}
The solution of the IBVP for the wave equation is given by a continuous unitary evolution group 
on the space $ {H}^1(\Omega,dx)\x L^2(\Omega,dx)$.
Hence
\begin{equation*}\begin{split}
&\nordom{\dt u_\ep(t,.) - \dt u_{\ep,\indx}(t,.)}\lesssim \nordom{v_\ep^I- v_{\ep,\indx}^{I}}+ \norhom{u_\ep^I- u_{\ep,\indx}^{I}},\\
&\nordom{\dpp_{x_j} u_\ep (t,.)-\dpp_{x_j} u_{\ep,\indx}(t,.)}\lesssim \nordom{v_\ep^I-  v_{\ep,\indx}^{I}}+ \norhom{u_\ep^I- u_{\ep,\indx}^{I}},
\end{split}\end{equation*}
for $j=1,\dots,n$.
Finally, by using \eqref{WigApp}, the estimate \eqref{wigDiff} is simplified into
\begin{equation}
\label{majTransport}\begin{split}
&|<\nrjFun{u_\ep(t,.)},\phi>-\ud < w\Croch{\Upsilon_\ep^{+}},\phi^{-t}>-\ud < w\Croch{\Upsilon_\ep^{-}},\phi^{t}>|  \\
\lesssim & \underset{\tendz\ep}{\limsup}\nordom{v_\ep^I-v_{\ep,\indx}^{I}} + \underset{\tendz\ep}{\limsup}\norhom{u_\ep^I-u_{\ep,\indx}^{I}}.
\end{split}
\end{equation}
We therefore need to estimate the difference between initial data \eqref{MPb:gp3} and \eqref{CIcut}. We start by the initial speed.
By the exponential decrease of $T_\ep^* \gamma_{\indx} T_\ep \uvepI$ on the support of $1-\rho$ (see \eqref{estim1}), one has
$$ \nordom{v^I_\ep - v_{\ep,\indx}^{I} } \lesssim \ep^\infty + \nordom{v^I_\ep - T_\ep^* \gamma_{\indx} T_\ep \uvepI}.$$
Because $T_\ep^*$ is bounded on $L^2(\IR^{2n}) \rightarrow L^2(\IR^{n})$ and $T_\ep^*T_\ep = Id$
\begin{equation*}
\nordvois{\uvepI - T_\ep^*   \gamma T_\ep  \uvepI}{\IR^n}\leq \underset{\displaystyle{\text{\ding{172}}}}{ \nordvois{   (1-\chi_{r_\infty/2})T_\ep \uvepI}{\IR^{2n}}}
+\underset{\displaystyle{\text{\ding{173}}}}{\nordvois{ \chi_{r_\infty/2} \chi_{4 r_0} T_\ep \uvepI}{\IR^{2n}}}.
\end{equation*}
Firstly, Lemma \ref{FBI2} yields
$$ (\text{\ding{172}})^2 = \nordvois{c_n (2 \pi)^{-\nd} \ep^{-\nq} \Par{1-\chi_{r_\infty/2}(\eta)} \int_{\IR^{2n}} \Fo\uvepI(\xi)e^{i \xi \cdot y-(\eta-\ep\xi)^2/(2\ep)} d\xi}{\IR^{2n}_{y,\eta}}^2.$$
It follows by Parseval equality that 
\begin{equation*}\begin{split}
(\text{\ding{172}})^2 = c_n^2 \ep^{-\nd}\int_{|\ep\xi| \leq r_\infty/8} \big(1&-\chi_{r_\infty/2}(\eta)\big)^2 |\Fo \uvepI(\xi)|^2 e^{-(\eta-\ep\xi)^2/\ep}d\xi d\eta \\
+c_n^2 \ep^{-\nd}\int_{|\ep\xi| \geq r_\infty/8}& \big(1-\chi_{r_\infty/2}(\eta)\big)^2 |\Fo\uvepI(\xi)|^2 e^{-(\eta-\ep\xi)^2/\ep}d\xi d\eta.
\end{split}\end{equation*}
The first integral in the r.h.s. is exponentially decreasing, which leads to
\begin{equation*}
\underset{\tendz\ep}{\limsup}\text{ \ding{172}} \lesssim \underset{\tendz\ep}{\limsup}\Par{\int_{|\ep \xi| \geq r_\infty/8}|\Fo \uvepI(\xi)|^2 d\xi}^{\ud}.
\end{equation*} 
Secondly, as $\dist\Par{\supp v_\ep^I, \supp (1-\rho)} >0$, one gets $\nordvois{ (1-\rho ) T_\ep \uvepI }{\IR^{2n}} \leq e^{-C/\ep}$ by Lemma \ref{FBIoutBis} and thus
$$ \underset{\tendz\ep}{\limsup}(\text{ \ding{173}})^2 = \underset{\tendz\ep}{\limsup} \nordvois{\rho(y)\chi_{r_\infty/2}(\eta) \chi_{4 r_0}(\eta)T_\ep \uvepI}{\IR^{2n}}^2.$$
It results from the relation \eqref{FBIWig} applied with $a_\ep = \uvepI$ that
$$(\text{\ding{173}})^2 \underset{\tendz{\ep}}{\rightarrow} <w\Croch{\uvepI},\rho^2 \ox\chi_{r_\infty/2}^2 \chi_{4 r_0}^2>.$$
Because $w\Croch{\uvepI}$ is a regular measure, assumption \eqref{wigZero} yields
\begin{equation*}
\forall \alp>0,\exists l_0(\alp)>0 \text{ s.t. } w\Croch{\uvepI}(\{|\xi|\leq  l_0(\alp)\})\leq \alp.
\end{equation*}
One deduces, for $4 r_0 \leq l_0(\alp)$, that
\begin{equation*}
\underset{\tendz\ep}{\limsup}\text{ \ding{173}}\lesssim \sqrt{ \alp},
\end{equation*}
which leads to 
\begin{equation*}
\underset{\tendz\ep}{\limsup}\nordom{v^I_\ep - v_{\ep,\indx}^{I}}\lesssim \underset{\tendz\ep}{\limsup} \Par{\int_{|\ep \xi| \geq r_\infty/8}|\Fo \uvepI(\xi)|^2 d\xi}^{\ud} + \sqrt{\alp}.
\end{equation*}

For the analysis of $u_\ep^I - u_{\ep,\indx}^I$ in $H^1(\Omega)$, we begin by estimating the spatial derivatives of the difference. 
It follows by using the relation \eqref{derivFBI} when differentiating the inverse FBI transform that
\begin{equation*}
\dpp_{x_j} u_\ep^I-\dpp_{x_j}u^{I}_{\ep,\indx} = \dpp_{x_j} u_\ep^I - (\dpp_{x_j}\rho) T_\ep^*\gamma_{\indx} T_\ep \uuepI - 
\rho T_\ep^* \gamma_{\indx} \dpp_{y_j} T_\ep \uuepI.
\end{equation*}
The term involving the derivative of $\rho$ is exponentially decreasing by Lemma \ref{cutx}.
Since the FBI transform of a derivative is the derivative of the FBI transform by \eqref{derivFBI}, one has to estimate $\nordom{\dpp_{x_j} u_\ep^I - \rho T_\ep^* \gamma_{\indx}  T_\ep  \udxjuepI}$.
Employing the same previous techniques yields for $j=1,\dots,n$
\begin{equation*}
\underset{\tendz\ep}{\limsup}\nordom{\dpp_{x_j} u^I_\ep - \dpp_{x_j} u_{\ep,\indx}^{I}}\lesssim \underset{\tendz\ep}{\limsup} \Par{\int_{|\ep \xi| \geq r_\infty/8}|\Fo \Par{\udxjuepI(\xi)}|^2 d\xi}^{\ud} + \sqrt{\alp},
\end{equation*}
if $4 r_0 \leq l_j(\alp)$ and $w\Croch{\udxjuepI}(\{|\xi|\leq  l_j(\alp)\})\leq \alp$. 
Set $r_0 = \frac 1 4 \underset{0 \leq j \leq n}{\min}l_j(\alp)$, then the Poincar\'e inequality yields the same bound for $\underset{\tendz\ep}{\limsup}\nordom{u_\ep^I-u^{I}_{\ep,\indx}}$.
\\
Coming back to \eqref{majTransport} we deduce that 
\begin{equation}
\label{majTransport2}\begin{split}
&\Big|<\nrjFun{u_\ep(t,.)},\phi>-\ud < w\Croch{\Upsilon_\ep^{+}},\phi^{-t}>-\ud < w\Croch{\Upsilon_\ep^{-}},\phi^{t}>\Big|  \\
\lesssim&  \sqrt{\alp} +\Par{\underset{\ep\rightarrow 0}{\limsup} \int_{\ep|\xi|\geq r_\infty/8}  |\Fo\Par{\uvepI}(\xi)|^2 d\xi}^{\ud} \\ &+\somme{j=1}{n}\Par{\underset{\ep\rightarrow 0}{\limsup} \int_{\ep|\xi|\geq r_\infty/8}  |\Fo\Par{\udxjuepI}(\xi)|^2 d\xi}^{\ud}.
\end{split}\end{equation}
The assumption \eqref{oscil} of $\ep-$oscillation means by definition that
\begin{align}
	\label{oscilDefv}
   \underset{\tendz\ep}{\limsup} \int_{\ep|\xi|\geq R}  |\Fo\Par{\uvepI}(\xi)|^2 d\xi &\underset{R\rightarrow +\infty}{\rightarrow} 0, \\
	\label{oscilDefu}
\underset{\tendz\ep}{\limsup} \int_{\ep|\xi|\geq R}  |\Fo\Par{\udxjuepI}(\xi)|^2 d\xi &\underset{R\rightarrow +\infty}{\rightarrow} 0 \text{ for } j=1,\dots,n.
\end{align}
Since the l.h.s. of the estimate \eqref{majTransport2} does not depend on $\alp$ nor $r_\infty$, one deduces by taking the limits $\tendz{\alp}$ and  $r_\infty \rightarrow \infty$ that
\begin{equation*}
 \nrjFun{u_\ep(t,.)} = \ud  w\Croch{\Upsilon_\ep^{+}} \ron \Par{\varphi^{-t}_b}^{-1} + \ud w\Croch{\Upsilon_\ep^{-}}\ron \Par{\varphi^{t}_b}^{-1} \text{ in } \Cob.
\end{equation*}

%% file: AppendixA.tex
\subsubsection{Higher order beams}

Higher order beams, possibly with more than one amplitude, can be constructed to satisfy better interior and boundary estimates. In this case, the eikonal equation \eqref{eikonalTwo} must be satisfied up to order $R \geq 2$ on the rays. If $ r \geq 3$, the equations 
\begin{equation}
\label{eikonr}
\dpp_x^\alp \Par{p(x,\dt \psi, \dpp_x \psi)}(t,x^t)=0, \, |\alp|= r,
\end{equation}
give systems of linear ODEs of order $1$ on $\Par{\dpp_x^{\alp}\psi(t,x^t)}_{|\alp|=r}$ with second members involving lower order spatial derivatives of the phase. In fact, the key observation is the equality
\begin{equation*}	
\begin{split}
&\dpp_\tau p(\varphi^t) \dt\dpp_x^\alp\psi(t,x^t) + 	\dpp_\xi p(\varphi^t) \cdot \dpp_x\dpp_x^\alp\psi(t,x^t)\\
=&2 c(x^t) |\xi^t| \dt\dpp_x^\alp\psi(t,x^t) + 2 c^2(x^t) \xi^t \cdot \dpp_x\dpp_x^\alp\psi(t,x^t) \\
=&2 c(x^t) |\xi^t| \frac{d}{dt}\Par{\dpp_x^\alp\psi(t,x^t)},
\end{split}
\end{equation*}
used for $|\alp|=r$ to eliminate the $r+1$-th order derivatives of $\psi$ in equation \eqref{eikonr}. To summarize, the requirements
\begin{align*}
	\dt\psi(t,x^t) &= -c(x^t)|\xi^t|, \, \dpp_x\psi(t,x^t) = \xi^t,\\
	p\Par{x,\dt \psi(t,x),\dpp_x\psi(t,x)}&=0 \text{ on } x=x^t \text{ up to order } R,
\end{align*}
uniquely determine the spatial derivatives of $\psi$ on the ray up to the order $R$ under the knowledge of their initial values on $(0,x^0)$. We refer to \cite{Ralston82} for further details.
\subsubsection{A general relation between incident and reflected beams phases}
By \eqref{incHess},  the Hessian matrix of the incident beam's phase is related to the Jacobian matrix of the incident flow. One can prove that its higher order derivatives are also related to the higher order derivatives of the incident flow. Computations exhibiting such relations can be found for instance in the Appendix of \cite{Norris2}.
We shall give a nice relation between an incident phase $\phsi$ and the associated reflected phase $\phsr$ for beams of any order. This relation is  intuitive true on geometrical grounds and it provides with the derivatives of the reflected phase up to order $R$, which might be useful in applications of Gaussian beams. 

Consider the following auxiliary function linking $\varphi_{1}^t$ to $\varphi_{0}^t$ for any fixed time $t$ 
\begin{equation*}
\begin{split}
s_1: \Bic &\rightarrow \Bic \\
(x,\xi)&\mapsto \varphi_0^{-T_1(x,\xi)}\ron \Ref \ron \varphi_0^{T_1(x,\xi)}(x,\xi).
\end{split}
\end{equation*}
For a given point $(x,\xi)\in \Bic$, $s_1(x,\xi)$ is its "image by the mirror" $\dom$. For instance, Chazarain used this type of auxiliary functions in \cite{Chazarain} to show propagation of regularity for wave type equations in a convex domain.

By the Implicit functions theorem, $T_1$ is $\Co^\infty$ on the open set $\Bic$ and so is $s_1$. Since $\varphi_{0}^t\ o\ s_1$ satisfies the same Hamiltonian equations as $\varphi_{1}^t$ and $\varphi_1^{T_1(x,\xi)}(x,\xi) = \varphi_{0}^{T_1(x,\xi)}\ o\ s_1(x,\xi)$ for $(x,\xi) \in \Bic$, one has
$$\varphi_1^t = \varphi_{0}^t\ o\ s_1.$$
Besides, noticing that $T_1(\varphi_0^t)= T_1 - t$, one has also
\begin{equation}
	\label{sronp}
	\varphi_1^t =  s_1 \ o\ \varphi_{0}^t.
\end{equation}
$\varphi_0^t$ and $\varphi_1^t$ are symplectic $\Co^\infty$ diffeomorphisms from $\Bic$ to $\Bic$ \cite{Ivrii}, and so is $s_1$.
One can define a similar auxiliary function $s_{-1}:\Bic \rightarrow \Bic$ s.t. $\varphi_{-1}^t = \varphi_{0}^t\ o\ s_{-1}$ and $\varphi_{-1}^t =  s_{-1} \ o\ \varphi_{0}^t$ for $t \in \IR$.

Let us introduce the components of $s_1$ as
\begin{equation*}
	s_1=(r,\lambda).
\end{equation*}
For $m\in \IN$, $f$, $g$ functions in $\Co^\infty\Par{\IR^n_u\x(\IR^{n}_\xi\prive{0}),\IC^p}$, $u_0\in \IR^n$ a fixed point and $V \in \Co^\infty(\IR^{n}_u,\IC^n_\xi)$ a phase function  s.t. $V(u_0) \in \IR^{n}_\xi\prive{0}$, we introduce the notation 
$$f\Par{u,V(u)} \asp{m}{u=u_0} g\Par{u,V(u)},$$ 
to denote that the formal partial derivatives of $f\Par{u,V(u)}$ and $g\Par{u,V(u)}$ up to the order $m$ coincide on $u_0$. The differentiation here is viewed formally, since $V$ may be complex valued out of $u_0$, which makes $f(u,V(u))$ and $g(u,V(u))$ not defined for $u \ne u_0$. However, on the exact point $u_0$, one can always use the formula of composite functions derivatives to get a formal expression of the derivatives. We will use the same notation
$$f\Par{t,x,V(t,x)} \asp{m}{x=x^t} g\Par{t,x,V(t,x)},$$ 
for functions $f,g \in \Co^\infty\Par{\IR_t\x \IR^n_x\x(\IR^{n}_\xi\prive{0}),\IC^p}$ and phase function $V \in \Co^\infty(\IR_t\x \IR^n_x, \IC^n_\xi)$ s.t. for $t \in \IR$, $V(t,x^t) \in \IR^{n}_\xi\prive{0}$  to denote that the formal partial derivatives of $f\Par{t,x,V(t,x)}$ and $g\Par{t,x,V(t,x)}$ w.r.t. $x$ up to order $m$ coincide on $(t,x^t)$ for all $t \in \IR$. We will be sloppy with respect to the notation of the dependence of the phase $V$ on its variables.

Consider an integer $R \geq 2$ and an incident phase $\phsi$ satisfying
\begin{equation*}\begin{split}
&\dt \phsi(t,x_0^t) = - c(x_0^t) |\xi_0^t|, \, \dpp_x \phsi(t,x_0^t) = \xi_0^t \text{ and } p(x,\dt \phsi, \dpp_x \phsi)  \asp{R}{x=x_0^t} 0.
\end{split}\end{equation*}
As a particular case, the phase $\psi_0$ is obtained by setting $R=2$ and choosing its initial value on the ray as zero and its initial Hessian matrix on the ray as $i Id$.

Let $\phsr \in \Co^\infty(\IR_t \x \IR_x^n,\IC)$ be the reflected phase associated to $\phsi$, that is the phase satisfying
\begin{equation*}\begin{split}
& \dt \phsr(t,x_1^t) = - c(x_1^t) |\xi_1^t|,\, \dpp_x \phsr(t,x_1^t) = \xi_1^t \text{ and } p(x,\dt \phsr, \dpp_x \phsr)  \asp{R}{x=x_1^t} 0,
\end{split}\end{equation*}
and having the same time and tangential derivatives as $\phsi$ at the instant and the point of reflection $(T_1,x_0^{T_1})$ up to the order $R$.

Since $\varphi_0^t$ and the reflection $\Ref$ conserve $c(x)|\xi|$ (see \eqref{refInv}), one has for every $(x,\xi)\in \Bic$ and $\tau \in \IR^*$
\begin{equation*}
	p\Par{r(x,\xi),\tau,\lambda(x,\xi)} = p(x,\tau,\xi) . 
\end{equation*}
Thus 
$$ p\Par{r(x,\dpp_x \phsi),\dt \phsi,\lambda(x,\dpp_x \phsi)}  \asp{\infty}{x=x_0^t} p(x,\dt \phsi,\dpp_x \phsi) , $$
which implies, by construction of $\phsi$
\begin{equation}
\label{pMirrorInc}
	 p\Par{r(x,\dpp_x \phsi),\dt \phsi,\lambda(x,\dpp_x \phsi)} \asp{R}{x=x_0^t} 0. 
\end{equation}
Compare this with the equation
\begin{equation*}
	 p\Par{r(x,\dpp_x \phsi),\dt \phsr\Par{t,r(x,\dpp_x \phsi)},\dpp_x\phsr\Par{t,r(x,\dpp_x \phsi)}} \asp{R}{x=x_0^t} 0
\end{equation*}
resulting from the construction of $\phsr$ and \eqref{sronp}. This suggests the following Lemma
\begin{lemma}
\label{Inc&Ref}
$$\dt \phsr\Par{t,r(x,\dpp_x \phsi)} \asp{R-1}{x=x_0^t} \dt \phsi \text{ and }\dpp_x \phsr\Par{t,r(x,\dpp_x \phsi)} \asp{R-1}{x=x_0^t} \lambda(x,\dpp_x \phsi).$$
\end{lemma}
A similar result linking the reflected phase associated to the ray $(t,x_{-1}^{-t})$ to $\phsi$ can be established.

\begin{proof}
The strategy of the proof is the following: we consider a phase function $\theta$ satisfying the relations announced in Lemma \ref{Inc&Ref} and we prove that $\theta$ fulfills the eikonal equation on the reflected ray up to order $R$ and has the correct derivatives at the instant and point of reflection. This proves that $\theta$ coincides with the reflected phase on the reflected ray up to the order $R$.

Denote $r\Par{x,\dpp_x \phsi(t,x)}$ by $\varrho(t,x)$ or simply by $\varrho$ if no confusion arises and let us first verify that for a fixed $k \geq 1$ there exists a phase function $\theta \in \Co^\infty(\IR_t\x \IR^n_x,\IC)$ s.t.
\begin{equation}
\label{defTheta}
 \dpp_x \theta(t,\varrho) \asp{k}{x=x_0^t} \lambda(x,\dpp_x \phsi).
\end{equation}
Let $A(t,x,\xi) =  D_x r(x,\xi) + D_\xi r(x,\xi) \dpp_x^2 \phsi(t,x)$
and $B(t,x,\xi) =  D_x \lambda(x,\xi) + D_\xi \lambda(x,\xi) \dpp_x^2 \phsi(t,x)$. Then $D_x \varrho(t,x) = A(t,x,\dpp_x \phsi)$, $ D_x [\lambda(x,\dpp_x \phsi(t,x))] = B(t,x,\dpp_x \phsi)$ and for $v\in \Co^\infty(\IR_t\x \IR^n_x,\IC^p)$ one has
\begin{equation*}
	D_x\Par{v(t,\varrho)} \asp{\infty}{x=x_0^t}  D_x v(t,\varrho) A(t,x,\dpp_x\phsi).
\end{equation*}
Hence, $\theta$ exists if $A(t,x_0^t,\xi_0^t)$ is non singular and
\begin{equation}
\label{symm}
 B(t,x,\dpp_x\phsi)       A(t,x,\dpp_x\phsi)^{-1} \asp{k-1}{x=x_0^t} \Par{A(t,x,\dpp_x\phsi)^{T}}^{-1} B(t,x,\dpp_x\phsi)^T.
\end{equation}
From \eqref{sronp} one gets
$$ A(t,x_0^t,\xi_0^t) (D_y x_0^t+ iD_\eta x_0^t)= D_y x_1^t+ iD_\eta x_1^t.$$
Since $\varphi^t_1$ is symplectic, the matrix $\left(\begin{array}{cc} D_y x_1^t  &D_\eta x_1^t  \\ D_y \xi_1^t  &D_\eta \xi_1^t \\ \end{array} \right)$ is symplectic. This implies in particular the relation 
$$ D_\eta \xi_1^t (D_y x_1^t)^T- D_y \xi_1^t ({D_\eta x_1^t})^T =Id$$
and the symmetry of $D_y x_1^t (D_\eta x_1^t)^T$. Thus, $\ker (D_\eta x_1^t)^T \cap \ker (D_y x_1^t)^T =\{0\}$ and at the same time,
$$(D_y x_1^t+ iD_\eta x_1^t) (D_y x_1^t+ iD_\eta x_1^t)^* = D_y x_1^t (D_y x_1^t)^T + D_\eta x_1^t (D_\eta x_1^t)^T.$$
This proves that $D_y x_1^t+ iD_\eta x_1^t$ is invertible and so is $A(t,x_0^t,\xi_0^t)$.
On the other hand, 
\begin{equation*}
\left(\begin{array}{c} A\\B \end{array}\right) = \left( \begin{array}{cc} D_x r       & D_\xi r \\ D_x \lambda & D_\xi \lambda \end{array}\right)
\left( \begin{array}{c} Id\\ \dpp_x^2\phsi \end{array} \right). 
\end{equation*}
Let $M(x,\xi)=  \left(\begin{array}{cc} D_x r(x,\xi)       & D_\xi r(x,\xi) \\ D_x \lambda(x,\xi) & D_\xi \lambda(x,\xi) \end{array} \right)$.
Then
\begin{equation*}
[A^T B-B^T A] = \left( \begin{array}{c} Id\\ \dpp_x^2\phsi \end{array} \right)^T M^T J M \left(\begin{array}{c} Id\\ \dpp_x^2\phsi \end{array}\right).
\end{equation*}
Since $M^T J M = D s_1^T J D s_1  $, the symplecticity of $s_1$ leads to
\begin{equation*}
M^T J M = J. 
\end{equation*}
Hence 
\begin{equation*}
[A^T B-B^T A] = \left(\begin{array}{c}Id\\ \dpp_x^2\phsi \end{array} \right)^T  J  \left(\begin{array}{c} Id\\ \dpp_x^2\phsi \end{array} \right) = 0 
\end{equation*}
and the requirement \eqref{symm} is fulfilled.

The relation \eqref{defTheta} fixes the derivatives of $\dt\dpp_x\theta$ on $(t,x_1^t)$ up to order $k-1$. Indeed, using the compatibility condition
\begin{equation*}\begin{split}
&\frac{d}{dt} \Croch{ f\Par{t,x,\dpp_x \phsi(t,x)}|_{x=x_0^t}} \\
=& \dt  \Croch{f\Par{t,x,\dpp_x \phsi(t,x)}}|_{x=x_0^t}  + \dpp_x  \Croch{f\Par{t,x,\dpp_x \phsi(t,x)}}|_{x=x_0^t} \cdot \dot{x}_0^t 
\end{split}\end{equation*} 
on the maps $(t,x,\xi) \mapsto \dpp_x \theta\Par{t,r(x,\xi)}$, $(x,\xi) \mapsto \lambda(x,\xi)$ and their derivatives yields recursively by \eqref{defTheta}
$$ \dt \Croch{\dpp_x \theta(t,\varrho)} \asp{k-1}{x=x_0^t} D_\xi \lambda(x,\dpp_x \phsi) \dt \dpp_x \phsi.$$
Thus
$$ \dt \dpp_x \theta(t,\varrho) +  \dpp_x^2 \theta(t,\varrho) D_\xi r(x,\dpp_x \phsi) \dt\dpp_x \phsi \asp{k-1}{x=x_0^t} D_\xi \lambda(x,\dpp_x \phsi) \dt \dpp_x \phsi.$$
Using the relations $\dpp_x^2 \theta (t,\varrho) \asp{k-1}{x=x_0^t} \Par{BA^{-1}}(t,x,\dpp_x \phsi)$ and \eqref{symm} in the previous equation yields
\begin{equation*}\begin{split}
 &\dt \dpp_x \theta(t,\varrho) \\
 \asp{k-1}{x=x_0^t} &\Croch{D_\xi \lambda(x,\dpp_x \phsi)- \Par{\Par{A^T}^{-1} B^T}(t,x,\dpp_x \phsi) D_\xi r(x,\dpp_x \phsi) } \dt \dpp_x \phsi. 
\end{split}\end{equation*}
Since
\begin{equation*}
A^TD_\xi \lambda-B^T D_\xi r = \left(\begin{array}{c} Id\\ \dpp_x^2\phsi \end{array} \right)^TM^T J M \left( \begin{array}{c} 0\\ Id \end{array}\right)=Id,
\end{equation*}
it follows that 
$ A(t,x,\dpp_x \phsi)^T \dt \dpp_x \theta(t,\varrho) \asp{k-1}{x=x_0^t}  \dt \dpp_x \phsi $.
Note that
\begin{equation}
\label{dScalOMirr}
	\dpp_x\Par{u(t,\varrho)} \asp{\infty}{x=x_0^t} A(t,x,\dpp_x\phsi)^T \dpp_x u(t,\varrho) \text{ for } u\in \Co^\infty(\IR_t\x \IR^n_x,\IC),
\end{equation}
so one gets
$$ \dpp_x\Par{\dt \theta(t,\varrho)} \asp{k-1}{x=x_0^t}  \dt \dpp_x \phsi .$$
Setting $\dt \theta(t,x_1^t) = \dt \phsi(t,x_0^t)$ implies then that
\begin{equation}
\label{dtTheta}
 \dt \theta(t,\varrho) \asp{k}{x=x_0^t}  \dt  \phsi. 	
\end{equation}
Putting together \eqref{pMirrorInc}, \eqref{defTheta} and \eqref{dtTheta} shows that the phase $\theta$ satisfies
\begin{equation*}
	 p\Par{\varrho,\dt \theta(t,\varrho),\dpp_x \theta(t,\varrho)} \asp{R}{x=x_0^t} 0 
\end{equation*}
under the further assumption $k \geq R$.

Let $\pi(t,x)= p\Par{x,\dt \theta(t,x),\dpp_x \theta(t,x)}$. Since $ \dpp_x \Par{\pi(t,\varrho)}(t,x_0^t)=0$ and 
\\
$A(t,x_0^t,\xi_0^t)$ is non singular, it follows by \eqref{dScalOMirr} that $\dpp_x \pi(t,x_1^t)$ is zero. More generally, for $m\geq 1$, the formula of composite functions' high derivatives yields
\begin{equation*}\begin{split} 
 \dpp_{x_{i_1}}\dots\dpp_{x_{i_m}}  \Croch{\pi\Par{t,\varrho(t,x)}}(t,x_0^t) =& \somme{j_1,\dots,j_m=1}{n}  \dpp_{x_{j_1}}\dots \dpp_{x_{j_m}} \pi(t,x_1^t) \overset{n} { \underset {k=1} {\prod} } A_{j_k i_k}(t,x_0^t,\xi_0^t) \\
 & + z_{i_1 \dots i_m}(t),
\end{split}\end{equation*}
where $z_{i_1 \dots i_m}$ depends on derivatives of $\pi$ on $(t,x_1^t)$ of order lower than $m$. For $m \leq R$, the l.h.s. is zero so one can show recursively on $|\beta|\leq R$ that $\dpp_x^\beta \pi(t,x_1^t)=0$. One thus has the following eikonal equation on $\theta$
\begin{equation*}
p(x,\dt \theta,\dpp_x \theta) \asp{R}{x=x_1^t} 0. 
\end{equation*}
To compare the time and tangential derivatives of $\theta$ and $\phsi$ at $(T_1,x_0^{T_1})$, let us introduce a $\Co^\infty$ parametrization of a neighborhood $\Um$ of $x_0^{T_1}$ in $\dom$
\begin{equation*}
	\sigma :\Nm \rightarrow \IR^n ,
\end{equation*}
where $\Nm $ is an open subset of $\IR^{n-1}$, $\sigma(\Nm)=\Um$ and $\sigma$ is a diffeomorphism from $\Nm$ to $\Um$. 
For $x\in \IR^n$ close to $x_{0}^{T_1}$, we may write $x=\sigma(\hat v)+v_n \nu\Par{\sigma(\hat v)},$ with $\hat v \in \Nm$ and $v_n\in \IR$.
Denote $\sigma(\hat v_1) = x_{0}^{T_1}$ and set $\theta_b(t,\hat v) = \theta\Par{t,\sigma(\hat v)}$ and $\Par{\phsi}_b(t,\hat v) = {\phsi}\Par{t,\sigma(\hat v)}$ the phases at the boundary near $x_0^{T_1}$. Since $r(X,\Xi)=X$ for $(X,\Xi) \in \overset{o}{T^*\IR^n}|_{\dom}$, it follows that
\begin{equation*}
	\varrho\Par{t,\sigma(\hat v)} \asp{\infty}{(t,\hat v) = (T_1,\hat v_1)} \sigma(\hat v),
\end{equation*}
which implies by \eqref{dtTheta} that
\begin{equation*}
 \dt \theta_b \asp{k}{(t,\hat v) = (T_1,\hat v_1)}  \dt  \Par{\phsi}_b.
\end{equation*}
Similarly $\lambda(X,\Xi) = \Xi - 2\Par{\Xi \cdot \nu(X)}\nu(X)$ for $(X,\Xi) \in \overset{o}{T^*\IR^n}|_{\dom}$, leading to
\begin{equation}
\label{dTangSig}
	D\sigma(\hat v)^T \lambda\Par{\sigma(\hat v), \dpp_x\phsi\Par{t,\sigma(\hat v)}} \asp{\infty}{(t,\hat v) = (T_1,\hat v_1)} 	D\sigma(\hat v)^T \dpp_x\phsi\Par{t,\sigma(\hat v)}.
\end{equation}
Since $\dpp_{\hat v} \theta_b(t,\hat v) = D \sigma(\hat v)^T \dpp_x \theta \Par{t,\sigma(\hat v)}$ and a similar relation holds true for $\dpp_{\hat v} \Par{\phsi}_b$, one gets from \eqref{defTheta} and \eqref{dTangSig} that
$\dpp_{\hat v} \theta_b \asp{k}{(t,\hat v) = (T_1,\hat v_1)}  \dpp_{\hat v}  \Par{\phsi}_b.$
Hence $\theta_b$ and $\Par{\phsi}_b$ have the same time and tangential derivatives at $(T_1,\hat v_1)$ from the order $1$ to the order $k+1$.
\\
If we assume that $\theta(T_1,x_0^{T_1})=\phsi(T_1,x_0^{T_1}),$ 
then 
\begin{equation*}
 \theta_b \asp{k+1}{(t,\hat v) = (T_1,\hat v_1)}  \Par{\phsi}_b,
\end{equation*}
and $\theta$ satisfies all the requirements that determine the reflected phase associated to $\phsi$ and concentrated on $(t,x_1^t)$. 
The phases $\theta$ and $\phsr$ are thus equal on $(t,x_1^t)$ up to the order $R$.
\end{proof}

\subsubsection{First order reflected beams' phases and amplitudes}

Lemma \ref{Inc&Ref} gives at order one
\begin{equation*}
\begin{split}
 &\dpp_x^2 \psi_1(t,x_1^t) \Par{D_x r(x_0^t,\xi_0^t)+D_\xi r(x_0^t,\xi_0^t) \dpp_x^2 \psi_0(t,x_0^t)} \\
=& D_x \lambda(x_0^t,\xi_0^t)+D_\xi \lambda(x_0^t,\xi_0^t) \dpp_x^2 \psi_0(t,x_0^t).
\end{split}
\end{equation*}
One obtains by plugging the expression \eqref{incHess} of $\dpp_x^2\psi_0(t,x_0^t)$
$$ \dpp_x^2 \psi_1(t,x_1^t) \Par{D_x r(x_0^t,\xi_0^t)U_0^t+D_\xi r(x_0^t,\xi_0^t) V_0^t} = D_x \lambda(x_0^t,\xi_0^t)U_0^t+D_\xi \lambda(x_0^t,\xi_0^t)V_0^t.$$
From \eqref{sronp}, it follows that
\begin{equation*}
\dpp_x^2 \psi_k(t,x_k^t)=V_k^t (U_k^t)^{-1} \text{ where }U_k^t=D_y x_k^t+i D_\eta x_k^t \text{ and }V_k^t=D_y \xi_k^t+i D_\eta \xi_k^t,
\end{equation*}
and a similar relation holds true for $\dpp_x^2 \psi_{-1}(t,x_{-1}^t)$.

The reflected amplitudes evaluated on the associated rays satisfy transport equations which are similar to \eqref{eqAmpl} and may be written as
\begin{equation*}
\frac{d}{dt}\Par{{a_0^{k}}^{(')}(t,x_k^t)}+\ud\Tr \Croch{\Par{H_{21}(x^t_k,\xi^t_k)U_k^t + H_{22}(x_k^t,\xi^t_k)V_k^t}(U_k^t)^{-1}}{a_0^{k}}^{(')}(t,x_k^t)=0.
\end{equation*}
One can obtain a similar equation to \eqref{eqU} on $U_k^t$ involving $H_{21}(x_k^t,\xi_k^t)$ and 
\\
$H_{22}(x_k^t,\xi_k^t)$, by using the relation $\varphi_k^t=\varphi_{0}^t \ o\ s_k$.
On the whole
\begin{equation*}
{a_0^{k}}^{(')}(t,x_k^t)= {a_0^{k}}^{(')}(T_k,x_0^{T_k}) \Par{\frac{\det U_k^t}{\det U_k^{T_k}}}^{-\ud}, \, k= \pm 1,
\end{equation*}
where the square root is obtained by continuity from $1$ at $t=T_k$.
\\
On the other hand, for $k= \pm 1$
\begin{equation*}
	d_{-m_B}^0 + d_{-m_B}^k = b(x,\dpp_x \psi_0) a_0^0 + b(x,\dpp_x \psi_k) a_0^k, 
\end{equation*}
where $b$ denotes the principal symbol of $B$. Thus, the condition \ref{refAmp} p.\pageref{refAmp} required for the construction of the reflected amplitudes implies that
${a_0^{k}}^{(')}(T_k,x_0^{T_k})= s {a_0^{0}}^{(')}(T_k,x_0^{T_k})$, with $s = -1$ for Dirichlet condition and $s = 1$ for Neumann condition.

In order to find the relationship between $U_k^{T_k}$ and $U_{0}^{T_{k}}$ for $k= \pm 1$, we differentiate the equality $x_k^{T_k}=x_{0}^{T_k}$
\begin{equation*}\begin{split}
D_{y,\eta} x_k^{T_k} + \dot{x}_k^{T_k} (\dpp_{y,\eta} T_k)^T = D_{y,\eta} x_0^{T_k} + \dot{x}_0^{T_k} (\dpp_{y,\eta} T_k)^T,
\end{split}\end{equation*}
and compute the derivatives of $T_k$ from the condition $x_0^{T_k} \in \dom$
\begin{equation*}
	\dpp_{y,\eta} T_k = - \frac{1}{\Par{\dot{x_0}^{T_k} \cdot \nu(x_0^{T_k})}}\, (D_{y,\eta} x_0^{T_k})^T \nu(x_0^{T_k})
\end{equation*}
to get after elementary computations
\begin{equation}
\label{Uref}
	U_k^{T_k} = \Par{Id-2\nu(x_{0}^{T_k})\nu(x_{0}^{T_k})^T} U_{0}^{T_{k}}.
\end{equation}
Hence
\begin{equation*}
a_0^{k}(t,x_k^t) = -s i \Par{\det U_k^t}^{-\ud} \text{ and } {a_0^{k}}'(t,x_k^t)= s (c(y)|\eta|)^{-1} \Par{\det U_k^t}^{-\ud} \text{ for } k= \pm 1,
\end{equation*}
where the square root is defined by continuity from $i [\det U_0^{T_k}]^{-\ud}$ at $t=T_k$.

%% file: AppendixB.tex
We briefly recall a simple version of the integral operators with complex phases used in \cite{BoAkAl09} and the estimates established therein. We then use these results to prove Lemma \ref{dtudxu}.

For $t\in[0,T]$, let $K_{z,\theta}(t)$ be a compact of $\IR^{2n}$ and consider the set $$E_1=\{(t,x,z,\theta)\in [0,T]\x\IR^{3n},\, (z,\theta)\in K_{z,\theta}(t),\, |x-z|\leq 1\},$$ which we assume compact. Let $\Phi$ be a phase function smooth on an open set containing $E_{1}$ and satisfying \eqref{propPhaseBis} for $t\in [0,T]$ and $(z,\theta)\in K_{z,\theta}(t)$. Then there exists $r[\Phi]\in]0,1]$ s.t.
\begin{equation*}
	\Im \Phi(t,x,z,\theta) \geq C (x-z)^2 \text{ for } t \in [0,T], \, (z,\theta) \in K_{z,\theta}(t) \text{ and } |x-z| \leq r[\Phi].
\end{equation*}
Let $l_\ep \in \Co^\infty([0,T]\x\IR^{3n},\IC)$ satisfying
\begin{equation}
\label{propAmpBis}
\begin{split} 
	&\text{for } t\in [0,T], l_\ep(t,x,z,\theta) = 0  \text{ if } (z,\theta) \notin K_{z,\theta}(t) \text{ or }|x-z| > r[\Phi],\\
  &\ep^{\frac{k}{2} }\dpp^k_{x_j} l_\ep \text{ is uniformly bounded in } L^\infty([0,T] \x \IR^{3n}) \text{ for every } 1\leq j\leq n \text{ and } k\in \IN.
 \end{split}
\end{equation}
If $O^\alp\Par{l_\ep(t,.),\Phi(t,.)/\ep}$ denotes, for a given multiindex $\alp$ and $t\in[0,T]$, the operator
\begin{equation*}\begin{split}
&\Croch{O^\alp\Par{l_\ep(t,.),\Phi(t,.)/\ep}h}(x)\\
 =&\int_{\IR^{2n}}  h(z,\theta) l_\ep(t,x,z,\theta)  (x-z)^\alp e^{i\Phi(t,x,z,\theta)/\ep} dz d\theta,\, h\in L^2(\IR^{2n}),
\end{split}\end{equation*}
then, under the previous hypotheses on $\Phi$ and $l_\ep$, we have the following estimate:
\begin{proposition}\emph{(\cite{BoAkAl09}, Lemma 3.3)} 
\label{appOp}
\\
$\|O^\alp\Par{ l_\ep(t,.),\Phi(t,.)/\ep}\|_{L^2(\IR^{2n})\rightarrow L^2(\IR^n)}\lesssim \, \ep^{\tnq+\frac{|\alp|}{2}}$ uniformly w.r.t. $t\in[0,T]$.
\end{proposition}
This estimate allows to prove Lemma \ref{dtudxu}.

\begin{proof}[Proof of Lemma \ref{dtudxu}]
		Consider the integrals \eqref{dtutilde} giving the derivatives of $u^{appr}_{\ep,\indx}$ and fix $j,k$ and $\alp$. The transported phase $\trak{\psi_k}$ is smooth and satisfies by \eqref{PhaseOneBis}, \eqref{PhaseZeroBis} and \eqref{PhaseTwoBis} the  properties \eqref{propPhaseBis} for $t\in[0,T]$ and $(z,\theta) \in K_{z,\theta}^k(t)$. We fix some $r[\trak{\psi_k}]\in]0,1]$ so that $	\Im \trak{\psi_k}(t,x,z,\theta)  \geq C (x-z)^2$ for $t \in [0,T]$, $(z,\theta) \in K_{z,\theta}^k(t)$ and $|x-z| \leq r[\trak{\psi_k}]$.
\\
For $t\in [0,T]$, $\Pi_k \trak{\rhou \ox \phiu}(t,z,\theta) \trak{\Par{ r_{j,\alp}^k}}(t,x,z,\theta)$ depends smoothly on its variables and vanishes for $|x-z|>d$ or $(z,\theta) \notin K_{z,\theta}^k(t)$. Hence, upon choosing $d \leq  r[\trak{\psi_k}]$, the amplitude $\Pi_k \trak{\rhou \ox \phiu} \trak{\Par{r_{j,\alp}^k}}$  satisfies the properties formulated in \eqref{propAmpBis}.
Let us check if $\Char{\Bic}\trak{f_\ep}= \Char{\Bic}\trak{T_\ep v^I_{\ep,\indx}}, \Char{\Bic}\ep^{-1} \trak{T_\ep u^I_{\ep,\indx}}$ is uniformly bounded in $L^2(\IR^{2n})$. Clearly $T_\ep v^I_{\ep,\indx}$ is, and the property holds true for $\ep^{-1} T_\ep u^I_{\ep,\indx}$ by Lemma \ref{estimu}. One can then use the approximation operators $O^\alp$ to write the integral \eqref{dtutilde} as 
\begin{equation*}\begin{split}
 & \ep^{-\tnq+j}  \int_{\IR^{2n}} \Pi_k \trak{\rhou \ox  \phiu}\, \trak{f_\ep} \trak{(r_{j,\alp}^k)} (x-z)^\alp e^{i\trak{\psi_k} /\ep} dz d\theta \\
=& \ep^{-\tnq+j} O^\alp\Par{\Pi_k \trak{\rhou \ox \phiu} \trak{(r_{j,\alp}^k)}(t,.),\trak{\psi_k}(t,.)/\ep}  \,\Char{\Bic} \trak{f_\ep}.
\end{split}\end{equation*}
The estimate established in Proposition \ref{appOp} yields
\begin{equation*}
\| \ep^{-\tnq+j}  \int_{\IR^{2n}} \Pi_k \trak{\rhou \ox  \phiu}\, \trak{f_\ep} \trak{(r_{j,\alp}^k)} (x-z)^\alp e^{i\trak{\psi_k} /\ep} dz d\theta\|_{L^2_x} \lesssim \ep^{\frac{|\alp|}{2}+j}.
\end{equation*} 
Hence, only $\trak{(r_{0,0}^k)}$ contributes to $\dpp_{t,x} {u}^{appr}_{\ep,\indx}$, the residue being of order $\sqrt \ep$. One has
\begin{equation*}\begin{split}
&r_{0,0}^k(t,x,y,\eta)= \frac i 2 c_n \beta_k  \dpp_{t,x} \psi_k(t,x_k^t)  \chi_d(x-x_k^t) a_k^{(')}(t,y,\eta),
\end{split}\end{equation*}
and by \eqref{PhaseOneBis}
\begin{equation*}\begin{split}
&\dt \psi_k(t,x_k^t)  = - c(x_k^t)|\xi_k^t|, \, \dpp_x \psi_k(t,x_k^t)  = \xi_k^t.
\end{split}\end{equation*}
It follows that  
\begin{equation*}\begin{split}
&\dt {u}^{appr}_{\ep,\indx}(t,x)\\
=&\ud\ep^{-\tnq}c_n\int_{\IR^{2n}}  \sum_{k=0,1}(-i) \beta_k c(z)|\theta| \chi_d(x-z)\Pi_k(t,z,\theta) \trak{\rhou\ox \phiu}(t,z,\theta)    \\
&\,\,   \qquad \qquad \qquad \qquad \trak{p_{\ep,k}}(t,z,\theta)e^{i\trak{\psi_k}(t,x,z,\theta)/\ep}  dz d\theta  \\
&+ \ud\ep^{-\tnq}c_n\int_{\IR^{2n}} \sum_{k=0,-1} i \beta_k c(z)|\theta|  \chi_d(x-z) \Pi_k(-t,z,\theta) \trak{ \rhou \ox \phiu}(-t,z,\theta)\\
&\qquad \qquad \qquad \qquad \qquad \trak{q_{\ep,k}}(-t,z,\theta)e^{i\trak{\psi_{k}}(-t,x,z,\theta)/\ep}) dz d\theta \\
& +O(\sqrt \ep)
\end{split}\end{equation*}
in $L^2(\IR^n)$, uniformly for $t\in [0,T]$ and
\begin{equation*}\begin{split}
&\dpp_{x_j} {u}^{appr}_{\ep,\indx}(t,x)\\
=&\ud\ep^{-\tnq}c_n\int_{\IR^{2n}}  \sum_{k=0,1} i \beta_k \theta_j \chi_d(x-z) \Pi_k(t,z,\theta) \trak{\rhou\ox \phiu}(t,z,\theta)  \\
&\,     \qquad \qquad \qquad \qquad \trak{p_{\ep,k}}(t,z,\theta)e^{i\trak{\psi_k}(t,x,z,\theta)/\ep}  dz d\theta  \\
&+ \ud\ep^{-\tnq}c_n\int_{\IR^{2n}} \sum_{k=0,-1}  i \beta_k \theta_j \chi_d(x-z) \Pi_k(-t,z,\theta) \trak{ \rhou \ox \phiu}(-t,z,\theta)   \\
&\qquad \qquad \qquad \qquad \qquad \trak{q_{\ep,k}}(-t,z,\theta)e^{i\trak{\psi_{k}}(-t,x,z,\theta)/\ep} dz d\theta  \\
& +O(\sqrt \ep)
\end{split}\end{equation*}
in $ L^2(\IR^n)$, uniformly w.r.t. $t\in [0,T]$.

One can get rid of the cut-off $\chi_d(x-z)$ appearing in $\dt {u}^{appr}_{\ep,\indx}(t,x)$ 
\\
and $\dpp_{x_j}{u}^{appr}_{\ep,\indx}(t,x)$ by using the estimate \eqref{noCutof}.
\end{proof}

%% file: AppendixC.tex
\begin{lemma} 
\label{FBI2}
For $u$ in $L^2(\IR^n)$
\begin{equation*}
T_\ep u(y,\eta)= c_n (2\pi)^{-\nd} \ep^{-\nq} \int_{\IR^n} \Fo u(\xi)e^{-(\eta-\ep\xi)^2/(2\ep)} e^{i \xi\cdot y} d\xi.
\end{equation*}
\end{lemma}

\begin{proof}
The equality is proven by Parseval formula.
\end{proof}

\begin{lemma}\emph{(\cite{BoAkAl09}, Lemma 2.4)}
\label{FBIoutBis}
Let $a$ be a positive real, $E$ a measurable subset of $\IR^n$ and $K\subset \IR^n$ a compact set s.t. $\dist(K,E) \geq a$. If $u \in L^2(\IR^n_x)$ is supported in $K$ then
$$\| T_\ep u\|_{L^2(E \x \IR^n_\eta)} = c_n \ep^{-\nq } \|\Char{E}(y) u(x) e^{-(x-y)^2/(2\ep)}\|_{L^2_{y,x}} \lesssim e^{-a^2/(4\ep)} \|u\|_{L^2_x}.$$
\end{lemma}

\begin{proof}
The proof consists of writing the FBI transform as the Fourier transform w.r.t. $x$ of some auxiliary function and using Parseval equality.
\end{proof}

\begin{lemma}
\label{cutx}
Let $\theta$ be a cut-off of $\Co_0^\infty(\IR^n_\eta,\IR)$, $E$ a measurable subset of $\IR^n$ and $K\subset \IR^n$ a compact set s.t. $\dist(K,E) >0$. If $u \in L^2(\IR^n)$ is supported in $K$ then
$$ \nordvois{T_\ep^* \theta(\eta) T_\ep u}{E} \lesssim e^{-C/\ep} \nordvois{u}{\IR^n}.$$
\end{lemma}

\begin{proof}
The kernel of $ \Char{E} T_\ep^* \theta(\eta) T_\ep \Char{K} : L^2(\IR^n_x) \mapsto L^2(\IR^n_{w})$ is 
\begin{equation*}\begin{split}
k_\ep(w,x) &= \ep^{-\tnd}c_n^2 \Char{E}(w) \Char{K}(x)\int_{\IR^{2n}} \theta(\eta)e^{i\eta\cdot(w-x)/\ep -(y-x)^2/(2\ep) -(w-y)^2/(2\ep)} dy d\eta \\
&= \Char{E}(w) \Char{K}(x) \ep^{-n}(2\pi)^{-n} \Fo\theta(\frac{x-w}{\ep}) e^{ -(x-w)^2/(4\ep) }.
\end{split}\end{equation*}
For $w\in \IR^n$, one has by Cauchy-Schwartz inequality
\begin{equation*}\begin{split}
\int_{\IR^n} |k_\ep(w,x)| dx &\leq \nordvois{\Fo\theta}{\IR^n} (2\pi)^{-n}\ep^{-\nd} \Par{\int_{\IR^n} \Char{E}(w) \Char{K}(x) e^{-(x-w)^2/(2\ep)} dx}^{\ud} \\
&\lesssim e^{-C/\ep}. 
\end{split}\end{equation*}
Similarly, $\int_{\IR^n} |k_\ep(w,x)| dw$ is dominated by $e^{-C/\ep}$, so one gets by Schur's Lemma
\begin{equation*}\begin{split}
 \nordvois{T_\ep^* \theta(\eta) T_\ep u}{E_{w}}  \lesssim e^{-C/\ep} \nordvois{u}{\IR^n_x}.
\end{split}\end{equation*}
\end{proof}

\begin{lemma}
\label{cuteta}
Let $E$ be a measurable subset of $\IR^n$ and $K\subset \IR^n$ a compact set s.t. $\dist(K,E) >0$.
If $\theta$ is a cut-off of $\Co_0^\infty(\IR^n_{\eta},\IR)$ supported in $K$ then
$$ \|T_\ep T_\ep^* \theta(\eta) T_\ep \|_{L^2(\IR^n) \rightarrow L^2( \IR^n \x E )} \lesssim e^{-C/\ep} .$$
\end{lemma}

\begin{proof}
Consider the operator $H_\ep :L^2(\IR^{2n}_{y,\eta}) \rightarrow  L^2(\IR^{2n}_{x,\xi})$ defined by
$$H_\ep f (x,\xi)=\Char{E}(\xi) T_\ep T_\ep^* \Par{\Char{K}(\eta)  f(y,\eta)}(x,\xi).$$
It is easy to compute its kernel $h_\ep$
\begin{equation*}\begin{split}
h_\ep(x,\xi,y,\eta) &= c_n^2 \pi^{\nd} \ep^{-n} \Char{E}(\xi) \Char{K}(\eta) e^{i(\xi+\eta)\cdot(x-y)/(2\ep) -(x-y)^2/(4\ep) -(\xi-\eta)^2/(4\ep)}. 
\end{split}\end{equation*}
Hence, $\int_{\IR^{2n}} |h_\ep(x,\xi,y,\eta)| dx d\xi \lesssim e^{-C/\ep}$ and $\int_{\IR^{2n}} |h_\ep(x,\xi,y,\eta)| dy d\eta \lesssim e^{-C/\ep}$.
For $u \in L^2(\IR^n)$ , it follows by Schur's Lemma that
\begin{align*}
 \nordvois{H_\ep T_\ep u}{\IR^{2n}_{x,\xi}} = \nordvois{T_\ep T_\ep^* \theta(\eta) T_\ep u}{ \IR^n_x \x E_{\xi}} &\lesssim e^{-C/\ep} \nordvois{T_\ep u}{\IR^{2n}_{y,\eta}}\\
 & \lesssim e^{-C/\ep} \nordvois{u}{\IR^n}.
\end{align*}
\end{proof}

\begin{lemma}\emph{(\cite{BoAkAl09}, Lemma 3.4)}
\label{estimu}
$\nordvois{ \ep^{-1} T_\ep u_{\ep,\indx}^I}{\IR^{2n}} \lesssim 1$.
\end{lemma}

\begin{proof} 
Differentiating \eqref{FBI1Bis} w.r.t. $y_j$, $0\leq j \leq n$, yields
\begin{equation*}
\begin{split}
\ep^{\ud} \dpp_{y_j} (T_\ep u_{\ep,\indx}^I) = i \eta_j \ep^{-\ud} T_\ep u_{\ep,\indx}^I -c_n\ep^{-\tnq}\int_{\IR^n}&u^I_{\ep,\indx}(x)\ep^{-\ud}(y_j-x_j) \\
&e^{i \eta \cdot(y-x)/\ep- (y-x)^2/(2\ep)}dx.
\end{split}
\end{equation*}
The l.h.s. is bounded in $L^2_{y,\eta}$ because $\dpp_{y_j} (T_\ep u_{\ep,\indx}^I) = T_\ep (\dpp_{x_j} u_{\ep,\indx}^I)$.
The second term of the r.h.s. is the Fourier transform of a bounded function in $L^2_x$, thus it can be estimated using Parseval equality. One gets
\begin{equation*}
\nordye{\ep^{-\tnq}\int_{\IR^n}u^I_{\ep,\indx}(x)\ep^{-\ud}(y_j-x_j) \,e^{i \eta\cdot(y-x)/\ep- (y-x)^2/(2\ep)}dx} \lesssim \|u_{\ep,\indx}^I\|_{L^2_x}.
\end{equation*}
Thus $\nordye{\ep^{-\ud} \eta_j T_\ep u_{\ep,\indx}^I} \lesssim 1$ and consequently by $\eqref{CIsmooth2}$
\begin{equation*}
	\nordye{\ep^{-\ud} T_\ep u_{\ep,\indx}^I} \lesssim 1.
\end{equation*}
Hence $\nord{u_{\ep,\indx}^I} \lesssim \sqrt \ep$. 
Reproducing the same arguments on the equality
\begin{equation*}
\begin{split}
 \dpp_{y_j} (T_\ep u_{\ep,\indx}^I) = i \eta_j \ep^{-1} T_\ep u_{\ep,\indx}^I -c_n\ep^{-\tnq}\int_{\IR^n}&\Par{\ep^{-\ud}u^I_{\ep,\indx}}(x) \ep^{-\ud}(y_j-x_j) \\
 &e^{i \eta \cdot(y-x)/\ep-(y-x)^2/(2 \ep)}dx
 \end{split}
\end{equation*}
leads to $\nordvois{u_{\ep,\indx}^I}{\IR^n} \lesssim  \ep$.
\end{proof}

\begin{lemma}
\label{modDWig}
Let $(a_\ep)$ and $(b_\ep)$ be two sequences uniformly bounded in $L^2(\IR^n)$ and $H^1(\IR^n)$ respectively. If $\ep^{-1} b_\ep$ is uniformly bounded in $L^2(\IR^n)$, then
$$w_\ep(a_\ep,|D|b_\ep) \approx |\xi| w_\ep(a_\ep,\ep^{-1}b_\ep) \text{ on } \IR^n \x (\IR^n\prive{0}).$$
\end{lemma}

\begin{proof} 
Let $\phi$ be a test function in $\Co_0^\infty(\IR^n \x (\IR^n\prive{0}),\IR)$ and denote $c_\ep = |D| b_\ep$.
We use another expression of $<w_\ep(a_\ep,c_\ep),\phi>$ exhibiting the Fourier transform of $c_\ep$:
\begin{equation*}
<w_\ep(a_\ep,c_\ep),\phi> = (2 \pi)^{-n} \int_{\IR^{2n}} \Fo_\xi\phi(x-\frac{\ep}{2}v,v) a_\ep(x) \bar c_\ep(x-{\ep}v) dv dx.
\end{equation*}
Since $\Fo_\xi\phi$ is rapidly decreasing
$$\underset{x}{\sup}\Modul{\Fo_\xi\phi(x-\frac{\ep}{2}v,v) - \Fo_\xi\phi(x,v)} \lesssim \ep (1+v^2)^{-n-1}.$$
By Cauchy-Schwartz inequality w.r.t. $dx$
$$ \int_{\IR^{2n}} |\Par{\Fo_\xi\phi(x-\frac{\ep}{2}v,v) - \Fo_\xi\phi(x,v)} a_\ep(x) \bar c_\ep(x-{\ep}v)| dv dx \lesssim \ep \nord{a_\ep} \nord{c_\ep}.$$ 
It follows that
\begin{equation*}
<w_\ep(a_\ep,c_\ep),\phi> = (2 \pi)^{-n} \int_{\IR^{3n}} \phi(x,\xi) e^{-i v \cdot \xi} a_\ep(x) \bar c_\ep(x-{\ep}v) dv dx d\xi + o(1).
\end{equation*}
Integrating w.r.t. $v$ leads to
\begin{equation*}
<w_\ep(a_\ep,c_\ep),\phi> = (2 \pi)^{-n} \ep^{-n} \int_{\IR^{2n}} \phi(x,\xi) e^{-i x \cdot \xi/\ep } a_\ep(x) \overline{\Fo c_\ep}(\xi/\ep)  dx d\xi + o(1),
\end{equation*}
and replacing $\Fo  c_\ep(\xi/\ep) $ by $ \ep^{-1}|\xi|\Fo  b_\ep(\xi/\ep)$ ends the proof.
\end{proof}

%% file: GBSforWigner.bbl
\begin{thebibliography}{99}																														\bibitem{Alexandre}
		\newblock R. Alexandre,
		\newblock \textit{Oscillations in {P}{D}{E} with singularities of codimension one. Part
  					{I} : review of the symbolic calculus and basic definitions}, preprint.


\bibitem{ArEn10} 
		\newblock G. Ariel, B. Engquist, N.M. Tanushev and R. Tsai,
		\newblock \textit{Gaussian beam decomposition of high frequency wave fields using expectation-maximization}, preprint.


\bibitem{Babich} 
		\newblock V.M. Babi{\v{c}},
		\newblock \emph{Eigenfunctions concentrated in a neighborhood of a closed geodesic}, (Russian)
		\newblock in ``Math. Problems in Wave Propagation Theory" (V.M. Babich, editor), 
							Zap. Nau{\v{c}}n. Sem. Leningrad. Otdel. Mat. Inst. Steklov., (1968), 15--63.


\bibitem{BoAkAl09} 
    \newblock S. Bougacha, J.-L. Akian and R. Alexandre, 
    \newblock \emph{Gaussian beams summation for the wave equation in a convex domain},
    \newblock Commun. Math. Sci., \textbf{7} (2009), 973--1008.

\bibitem{Burq} 
    \newblock N. Burq,
		\newblock \emph{Contr\^olabilit\'e exacte des ondes dans des ouverts peu
  	r\'eguliers}, (French) [Exact controllability of waves in nonsmooth domains],
		\newblock  Asympt. Anal., \textbf{14} (1997), 157--191.

\bibitem{Bu97} 
    \newblock N. Burq,
    \newblock  \emph{Mesures semi-classiques et mesures de d{\'{e}}faut}, (French) [Semiclassical measures and defect measures],
    \newblock  S{\'{e}}minaire Bourbaki, Ast{\'{e}}risque, (1997), 167--195.

\bibitem{Burq03} 
    \newblock N. Burq,
		\newblock \emph{Quantum ergodicity of boundary values of eigenfunctions: a control
  	theory approach},
		\newblock  Canad. Math. Bull., \textbf{48} (2005), 3--15.

\bibitem{BuLe} 
    \newblock N. Burq and G. Lebeau,
		\newblock \emph{Mesures de d{\'{e}}faut de compacit\'e, application au syst{\`{e}}me de
	  {L}am{\'{e}}}, (French) [Microlocal defect measures and application to the Lame system],
		\newblock Ann. Sci. {\'{E}}cole Norm. Sup., \textbf{34} ( 2001), 817--870.

\bibitem{Castella} 
		\newblock F. Castella,
		\newblock \emph{The radiation condition at infinity for the high frequency
	  {H}elmholtz equation with source term: a wave packet approach},
		\newblock J. Funct. Anal., \textbf{223} (2005), 204--257.
             
\bibitem{CePoPs}
		\newblock V. {\v{C}}erven{\'{y}}, M.M. Popov and I. P{\v{s}}en{\v{c}}{\'{i}}k,
		\newblock \emph{Computation of wave fields in inhomogeneous media - {G}aussian beam
  	approach},
		\newblock Geophys. J. R. Astr. Soc., \textbf{70} (1982), 109--128.

\bibitem{Chazarain} 
    \newblock J. Chazarain,
		\newblock \emph{Param\'etrix du probl\`eme mixte pour l'\'equation des ondes \`a
  						l'int\'erieur d'un domaine convexe pour les bicaract\'eristiques}, (French),
		\newblock in ``Journ\'ees \'Equations aux d\'eriv\'ees partielles de Rennes", Ast{\'{e}}risque, Soc. Math. France, (1976), 165--181.

\bibitem{CoRaRo} 
    \newblock M. Combescure, J. Ralston  and D. Robert,
		\newblock \emph{A proof of the {G}utzwiller semiclassical trace formula using
  coherent states decomposition},
		\newblock Commun. Math. Phys., \textbf{202} (1999), 463--480.

\bibitem{Duyckaerts04} 
		\newblock T. Duyckaerts,
		\newblock \textit{Stabilization of the linear system of magnetoelasticity}, 
		\newblock preprint.


\bibitem{FiMa} 
		\newblock S. Filippas and G.N. Makrakis,
		\newblock \emph{Semiclassical {W}igner function and geometrical optics},
		\newblock Multiscale Model. Simul., \textbf{1} (2003), 674--710.

\bibitem{Fouassier07} 
		\newblock E. Fouassier,
		\newblock \emph{High frequency limit of {H}elmholtz equations: refraction by sharp
  	interfaces},
		\newblock J. Math. Pures Appl., \textbf{87} (2007) 144--192.
		
\bibitem{GaMa} 
		\newblock I. Gasser and P.A. Markowich,
		\newblock \emph{Quantum hydrodynamics, {W}igner transform and the classical limit},
		\newblock Asympt. Anal., \textbf{14} (1997), 97--116.

\bibitem{Gerard91a} 
		\newblock P. G{\'{e}}rard,
		\newblock \emph{Mesures semi-classiques et ondes de {B}loch}, (French) [Semiclassical measures and Bloch waves],
		\newblock in ``S\'eminaire sur les \'Equations aux D\'eriv\'eees Partielles",
					   \'Ecole Polytech., (1991).

					   
\bibitem{Gerard91b} 
		\newblock P. G{\'{e}}rard,
		\newblock \emph{Microlocal defect measures},
		\newblock Commun. Partial Differential Equations, \textbf{16} (1991), 1761--1794.

\bibitem{GeLe} 
		\newblock P. G\'{e}rard and E. Leichtnam,
		\newblock \emph{Ergodic properties of eigenfunctions for the {D}irichlet problem},
		\newblock Duke Math. J., \textbf{71} (1993), 559--607.
		
\bibitem{GeMaMaPo} 
		\newblock P. G{\'{e}}rard, P.A. Markowich, N.J. Mauser and F. Poupaud,
		\newblock \emph{Homogenization limits and {W}igner transforms},
		\newblock Comm. Pure Appl. Math., \textbf{50} (1997), 323--379.

\bibitem{HormanderPDO1} 
		\newblock L. H{\"{o}}rmander,
		\newblock ``The {A}nalysis of {L}inear {P}artial {D}ifferential {O}perators
  						I. {D}istribution {T}heory and {F}ourier {A}nalysis",
		\newblock Springer-Verlag, Berlin, 1983.

\bibitem{Ivrii} 
		\newblock V. Ivrii,
		\newblock ``Microlocal {A}nalysis and {P}recise {S}pectral {A}symptotics",
		\newblock Springer Monographs in Mathematics,
		\newblock Springer Verlag, Berlin, 1998.

\bibitem{KaKuLa} 
		\newblock  A. Katchalov, Y. Kurylev and M. Lassas,
		\newblock  ``Inverse {B}oundary {S}pectral {P}roblems", 
		\newblock Chapman \& Hall/CRC Monographs and Surveys in Pure and Applied Mathematics, 123,
		\newblock Chapman \& Hall/CRC, Boca Raton, 2001.

\bibitem{KaPo81}
		\newblock A.P. Katchalov and M.M. Popov,
		\newblock \emph{The application of the {G}aussian beam summation method to the
  						computation of high-frequency wave fields},
		\newblock Dokl. Akad. Nauk, \textbf{258} (1981), 1097--1100.

\bibitem{Klimes84}
		\newblock  L. Klime{\v{s}},
		\newblock \emph{Expansion of a high-frequency time-harmonic wavefield given on an
  						initial surface into {G}aussian beams},
		\newblock Geophys. J. R. astr. Soc., \textbf{79} (1984), 105--118.

\bibitem{LaSi} 
		\newblock A. Laptev and I.M. Sigal,
		\newblock \emph{Global {F}ourier integral operators and semiclassical asymptotics},
		\newblock Rev. Math. Phys., \textbf{12} (2000), 749--766.

\bibitem{LeQi2009} 
		\newblock S. Leung and J. Qian,
		\newblock \emph{Eulerian {G}aussian beams for {S}chr{\"{o}}dinger equations in the semi-classical regime},
		\newblock J. Comput. Phys., \textbf{228} (2009), 2951--2977.

\bibitem{LeQi2010}
		\newblock S. Leung and J. Qian,
		\newblock \emph{The backward phase flow and {F}{B}{I}-transform-based {E}ulerian {G}aussian beams for the {S}chr{\"{o}}dinger equation},
		\newblock J. Comput. Phys., \textbf{229} (2010), 8888--8917.

\bibitem{LiPa} 
		\newblock P.-L. Lions and T. Paul,
		\newblock \emph{Sur les mesures de {W}igner}, (French) [On Wigner measures],
		\newblock Rev. Mat. Iberoamericana, \textbf{9} (1993), 553--618.


\bibitem{LiRaAcou} 
		\newblock H. Liu and J. Ralston,
		\newblock \emph{Recovery of high frequency wave fields for the acoustic wave equation},
		\newblock Multiscale Model. Simul., \textbf{8} (2009/10), 428--444.

\bibitem{LiRuTa2010}
		\newblock  H. Liu, O. Runborg and N.M. Tanushev,
		\newblock \emph{Error {E}stimates for {G}aussian {B}eam {S}uperpositions},
		\newblock preprint.


\bibitem{MaZu} 
		\newblock F. Maci{\`{a}} and E. Zuazua,
		\newblock \emph{On the lack of observability for wave equations: a {G}aussian beam
  						approach},
		\newblock Asymptot. Anal., \textbf{32} (2002), 1--26.

\bibitem{MaMa} 
		\newblock P.A. Markowich and N.J. Mauser,
		\newblock \emph{The classical limit of a self-consistent {Q}uantum-{V}lasov equation
  						in $3${D}},
		\newblock Math. Models Methods Appl. Sci., \textbf{3} (1993), 109--124.

\bibitem{MaMaPo} 
		\newblock P.A. Markowich, N.J. Mauser and F. Poupaud,
		\newblock \emph{A {W}igner-function approach to (semi)classical limits: electrons in
  						a periodic potential},
		\newblock J. Math. Phys., \textbf{35} (1994), 1066--1094.

\bibitem{MaPiPo}
		\newblock  P.A. Markowich, P. Pietra and C. Pohl.,
		\newblock \emph{Weak limits of finite difference schemes of {S}chr{\"{o}}dinger-type
						  equations},
		\newblock Pubbl. Ian, \textbf{1035} (1997), 1--57.

\bibitem{Martinez} 
		\newblock A. Martinez,
		\newblock "An {I}ntroduction to {S}emiclassical and {M}icrolocal	{A}nalysis",
		\newblock Springer-Verlag, New York, 2002.

\bibitem{Miller} 
		\newblock L. Miller,
		\newblock \emph{Refraction of high-frequency waves density by sharp interfaces and
	  					semiclassical measures at the boundary},
		\newblock J. Math. Pures Appl., \textbf{79} (2000), 227--269.

\bibitem{MoRu08} 
		\newblock M. Motamed and O. Runborg,
		\newblock \emph{Taylor expansion and discretization errors in Gaussian beam  superposition}, 
		\newblock Wave Motion, \textbf{47} (2010), 421--439.

\bibitem{Norris2}
		\newblock A.N. Norris,
		\newblock \emph{Elastic {G}aussian wave packets in isotropic media},
		\newblock Acta Mech., \textbf{71} (1988), 95--114.

\bibitem{PaRh} 
		\newblock G. Papanicolaou and L. Ryzhik,
		\newblock \emph{Waves and {T}ransport},
		\newblock in ``Hyperbolic Equations and Frequency Interactions",
						 IAS/Park City Math. Ser., 5, Amer. Math. Soc., (1999), 305--382.


\bibitem{PaUr96} 
		\newblock  T. Paul and A. Uribe,
		\newblock \emph{On the pointwise behavior of semi-classical measures},
		\newblock Comm. Math. Phys., \textbf{175} (1996), 229--258.


\bibitem{Pulvirenti} 
		\newblock M. Pulvirenti,
		\newblock \emph{Semiclassical expansion of {W}igner functions},
		\newblock J. Math. Phys., \textbf{47} (2006), 52103--52114.

\bibitem{QiYi2010}
		\newblock  J. Qian and L. Ying,
		\newblock \emph{Fast multiscale {G}aussian wavepacket transforms and multiscale {G}aussian beams for the wave equation},
		\newblock Multiscale Model. Simul., \textbf{} (to appear).

                      
\bibitem{Ralston82} 
		\newblock J. Ralston,
		\newblock \emph{Gaussian beams and the propagation of singularities},
		\newblock in ``Studies in partial differential equations", MAA Stud. Math., 23, Math. Assoc. America, (1982), 206--248.
		



\bibitem{Robinson} 
		\newblock S.L. Robinson,
		\newblock \emph{Semiclassical mechanics for time-dependent {W}igner functions},
		\newblock J. Math. Phys., \textbf{34} (1993), 2185--2205.

\bibitem{Tanushev08} 
		\newblock N.M. Tanushev,
		\newblock \emph{Superpositions and higher order {G}aussian beams},
		\newblock Commun. Math. Sci., \textbf{6} (2008), 449--475.

\bibitem{TaEnTs09} 
		\newblock N.M. Tanushev, B. Engquist and R. Tsai,
		\newblock \emph{Gaussian beam decomposition of high frequency wave fields},
		\newblock J. Comput. Phys., \textbf{228} (2009), 8856--8871.

\bibitem{TaQiRa} 
		\newblock  N.M. Tanushev, J. Qian and J.V. Ralston,
		\newblock \emph{Mountain waves and {G}aussian beams},
		\newblock Multiscale Model. Simul., \textbf{6} (2007), 688--709.

\bibitem{Tartar} 
		\newblock L. Tartar,
		\newblock \emph{$H$-measures, a new approach for studying homogenization, oscillations
  						and concentration effects in partial differential equations},
		\newblock Proc. Roy. Soc. Edinburgh Sect. A, \textbf{115} (1990), 193--230.

\bibitem{Wigner}
		\newblock E. Wigner,
		\newblock \emph{On the quantum correction for thermodynamic equilibrium},
		\newblock Phys. Rev., \textbf{40} (1932), 749--759.

\bibitem{Wilkinson} 
		\newblock M. Wilkinson,
		\newblock \emph{A semiclassical sum rule for matrix elements of classically chaotic
  						systems},
		\newblock J. Phys. A, \textbf{20} (1987), 2415--2423.

\bibitem{Zelditch} 
		\newblock S. Zelditch,
		\newblock \emph{Uniform distribution of eigenfunctions on compact hyperbolic
						  surfaces},
		\newblock Duke Math. J., \textbf{55} (1987), 919--941.    	\end{thebibliography}
